\documentclass[reqno,a4paper,11pt]{amsart}
\usepackage[left=2cm,right=2cm,top=2cm,bottom=2cm,bindingoffset=0cm]{geometry}
\usepackage{mathtools}
\usepackage{amsmath}
\usepackage{amssymb}
\usepackage{amsthm}
\usepackage{amsbsy}
\usepackage{citehack}
\usepackage{longtable}
\usepackage{hyperref}
\usepackage{esint}
\usepackage{graphics}
\usepackage[dvips]{graphicx}
\usepackage{psfrag}
\usepackage{epsfig}
\usepackage{subfig}

\numberwithin{equation}{section}

\newcommand{\R}{\mathbb{R}}
\newcommand{\N}{\mathbb{N}}

\newcommand{\h}{\mathfrak{H}}
\newcommand{\vi}{\mathfrak{V}}

\newtheorem{teo}{Theorem}[section]
\newtheorem{defin}[teo]{Definition}
\newtheorem{rem}[teo]{Remark}
\newtheorem{prop}[teo]{Proposition}
\newtheorem{lemma}[teo]{Lemma}
\newtheorem{cor}[teo]{Corollary}

\begin{document}

\author{Giuliano Lazzaroni}\author{Riccardo Molinarolo}\author{Francesco Solombrino}
\address[G.\ Lazzaroni]{Dipartimento di Matematica e Informatica ``Ulisse Dini'',
Universit\`a degli Studi di Firenze, Viale Morgagni 67/a, 50134 Firenze, Italy}
\email{giuliano.lazzaroni@unifi.it}
\address[R.\ Molinarolo and F.\ Solombrino]{Dipartimento di Matematica e Applicazioni ``Renato Caccioppoli'',
Universit\`a degli Studi di Napoli Federico II, Via Cintia, Monte S.\ Angelo,
80126 Napoli, Italy}
\email{riccardo.molinarolo@unina.it}
\email{francesco.solombrino@unina.it}
\thanks{\today}
\title[Radial solutions for a dynamic debonding model in dimension two]{Radial solutions for a dynamic debonding model \\ in dimension two}

\begin{abstract}
In this paper we deal with a debonding model for a thin film in dimension two, where the wave equation on a time-dependent domain is coupled with a flow rule (Griffith's principle) for the evolution of the domain. We propose a general definition of energy release rate, which is central in the formulation of Griffith's criterion. Next, by means of an existence result, we show that such definition is well posed in the special case of radial solutions, which allows us to employ representation formulas typical of one-dimensional models.
\end{abstract}

\maketitle

 \par

\bigskip

\noindent   
{{\bf 2020 Mathematics Subject Classification:}} 

\bigskip

\section*{Introduction}
%
In recent years, mathematical analysis gave a fundamental contribution to the theory of dynamic fracture, by means of existence and uniqueness results which show that mechanical models are well posed. This issue was studied in several frameworks: sharp crack \cite{NicSae07,DMLarToa15,DMLarToa17,CapLucTas,CapSap}, phase field \cite{LarOrtSue10,LRTT14,Caponi-NoDEA}, delamination \cite{RosRou11,Rou13a,RosTho16}. 
However, further understanding is needed of propagation criteria capable of predicting crack paths without geometric constraints. In contrast, in the quasistatic setting, where inertial effects are neglected, different notions of solutions were discussed, see \cite{BouFraMar08,MieRou15} and references therein.

In this paper we deal with a closely related model, that is dynamic debonding \cite{Fre90}.
We consider a flexible, inextensible, thin film, initially attached to a planar rigid substrate. The film is progressively peeled off by applying a tension and an opening to its edge. The free part of the film (debonded region, subject to inertia) is parametrized in the reference configuration by a time-dependent domain where (the third component of) the displacement satisfies the equation of the vibrating membrane, i.e., the wave equation. The part of the film still attached to the substrate is called bonded region. The interface between the two parts is called debonding front. Its evolution is unknown and governed by energetic criteria, which results in a flow rule coupled to the equation of motion of the free part (as typical in dynamic fracture). The problem is to determine the evolution of the debonding front and of the displacement.

Problems of dynamic debonding were considered in \cite{DouMarCha08,LBDM12,DMLazNar16,LazNar-qs,LazNar-speed,LazNar-init,RiNa2020,Riva-MJM,Riva-JNLS} assuming that the debonding front is a line orthogonal to the $x_1$-axis and the displacement is parallel to the $x_1$-axis and depends only on $x_1$ and not on $x_2$. Under such assumptions the setting is one-dimensional, so one can exploit the properties of the wave equation in dimension one, for instance the formula of d'Alembert, in order to completely solve the problem.
For numerical modeling we mention e.g.\ \cite{AbdDeb} and references therein.

Here we attack the two-dimensional model by considering a special case with radial solutions. More precisely, we assume that at the initial time the bonded region is a disk whose center is the origin (so the initial debonding front is a circle). Outside of such disk, the initial displacement is radial (i.e., it depends only on the distance from the center of the disk), as well as the initial speed and the time-dependent boundary condition.
Under such assumptions, and in absence of body forces, we show that the problem of dynamic debonding has a radial solution, i.e., at every time the debonding front is a circle centered at the origin and the displacement is a radial function.

To obtain this result, we first consider a prescribed evolution of the debonding front, so we fix
a nondecreasing, Lipschitz function $t\mapsto\rho(t)$ and consider the time-dependent domain $B_{R}\setminus B_{R-\rho(t)}$.
We show that there is a unique radial solution to the wave equation on the time-dependent domain $B_{R}\setminus B_{R-\rho(t)}$, complemented with radial initial and boundary conditions (see Theorem \ref{teoexistenceh}).
Notice that the existence of solutions to the wave equation in a growing domain could be proven by abstract methods, see e.g.\ \cite{BBL98}. On the other hand, if $\rho$ is sufficiently regular, uniqueness follows from methods by Ladyzenskaya \cite{Lady}, which in our case show a posteriori that the unique solution is radial (Remark \ref{rmk:uniq} and Appendix \ref{app:uniq}). 
Anyhow, in our paper we provide a proof based on explicit representation formulas, which turn out useful in the analysis of the coupled problem.

In fact, when the evolution of the debonding front $t\mapsto\rho(t)$ is not a priori known, we show that it can be selected by a flow rule, called Griffith's principle and based on energetic criteria (see Definition \ref{def:sol-coupled} and Theorem \ref{thm:main}).
In order to state such propagation law, a central quantity is the energy release rate, which accounts for the energy variation due to an infinitesimal growth of the domain.
After proving that this sort of energy derivative exists, one can state Griffith's criterion, which requires that the domain is nondecreasing in time and may grow only if the energy release rate equals a material parameter measuring the toughness of the glue between the film and the substrate.
There is a strong coupling between the flow rule and the wave equation, since the energy release rate implicitly depends on the displacement and its derivatives.

In this work, we limit ourselves to considering solutions where the bonded region is a disk. (By our first results, this implies that the displacement is a radial function, provided the initial data are sufficiently regular; see Remark \ref{rmk:radial-disp}.)
This ansatz allows us to simplify the setting: indeed, since we know the shape of the debonded region, which is described by a single parameter, we can explicitly compute the energy release rate.
Specifically, using polar coordinates we can pass to a one-dimensional problem, where the wave equation contains some damping terms similar to that considered in \cite{RiNa2020}, weighted with a different kernel.
By means of a further nontrivial change of variables (see Section \ref{sec:prescribed}),
we can recast the problem in a form that is suitable for the methods developed in \cite{DMLazNar16,LazNar-init,RiNa2020} based on fixed point theorems and representation formulas for the solutions of the wave equation.
Moreover, following \cite{RiNa2020} it is possible to include a damping term accounting for friction produced by air resistance (i.e., the term with coefficient $\alpha$ in \eqref{princeq} below).
It is also possible to account for toughness discontinuities.
Because of our ansatz, uniqueness holds only among those configurations with a radial debonding front.

The scope of our work is to give a general definition of energy release rate and Griffith's criterion for debonding models in dimension two. The well posedness of the problem is here tested in the special case of radial solutions, which simplifies the setting as hinted above. 
To some extent, the ansatz of radial solutions is comparable to the restriction to prescribed crack paths in fracture mechanics. 
However, in our work the definition of energy release rate is formulated without a priori assumptions on the solutions, so Griffith's criterion can be stated for a very wide class of possible debonding fronts (see Definition \ref{defenergyreleaserate} and Proposition \ref{prop:ERR-general}). 
We thus believe that our results represent a first step towards the understanding of more general models in dimension two, where the circular symmetry is possibly broken.

\section{The problem with prescribed debonding front} \label{sec1}

We consider a flexible, inextensible, thin film, initially attached to a planar substrate parametrized in the reference configuration on the $(x_1,x_2)$-plane and progressively peeled off from the substrate. In this section we assume that the evolution of the debonded region and of the debonded front is prescribed on a time interval $[0,T]$; in Section \ref{sec:coupled} we will remove such prescription.

 \begin{figure} 
 \centering
\subfloat[]{
 \psfrag{r}{\hspace{-1em} $\rho(t)$}
 \psfrag{R}{\hspace{-.5em} $R$}
 \psfrag{Rr}{\hspace{-.5em} $R{-}\rho(t)$}
 \includegraphics[width=.35\textwidth]{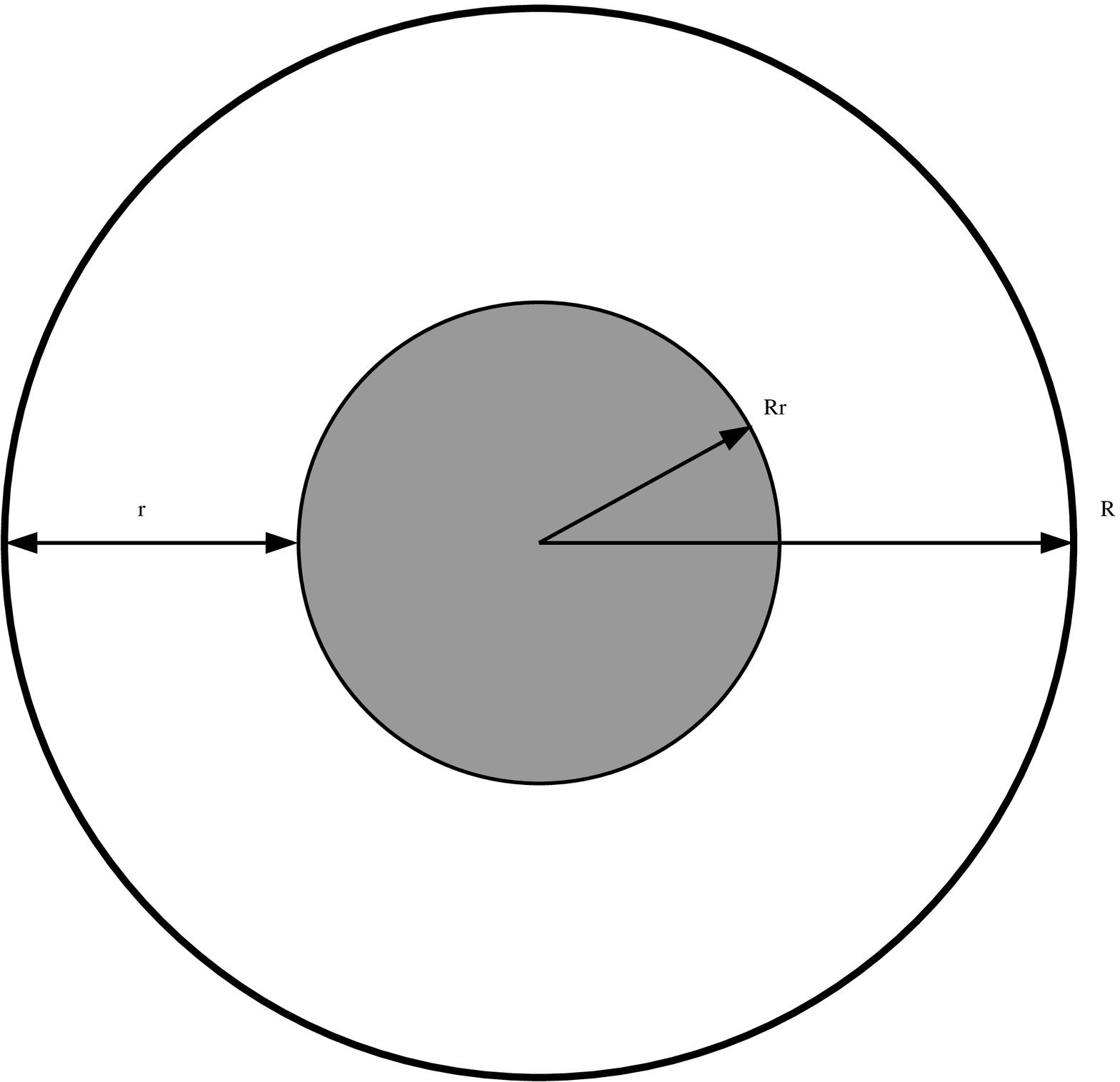}
}
\hspace{.1\textwidth}
\subfloat[]{
 \psfrag{w}{\hspace{-1em} $w(t)$}
 \includegraphics[width=.5\textwidth]{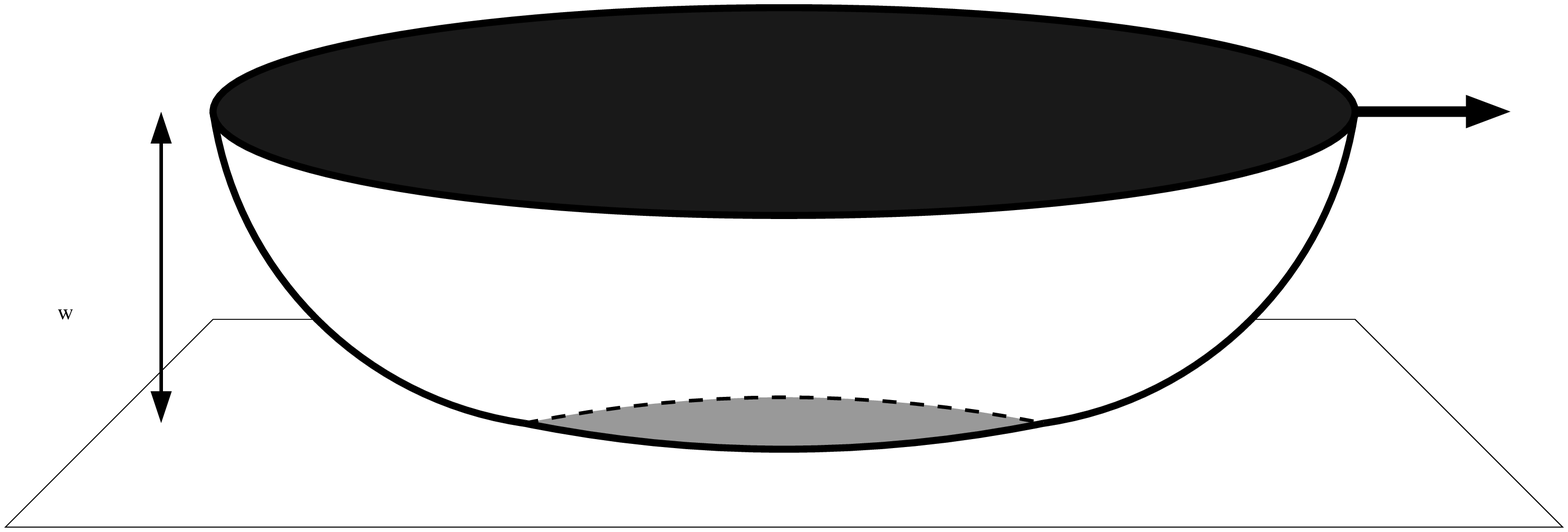}
}
 \caption{{\sc (a)} Reference configuration for a circular film peeled off from a substrate. {\sc (b)} Deformed configuration displaying the tension and the opening displacement exerted on the edge of the film.}
 \label{fig:reference}
 \end{figure}

We assume that the debonded region is parametrized on a growing annulus (Figure \ref{fig:reference}), whose width is given by a function 
$\rho: [0, T] \to [\rho_0, R)$, where
$R>\rho_0>0$ are fixed. We assume 
\begin{equation}\label{hprho}
\rho \in C^{0,1}([0,T] ; [\rho_0,R) ), \ 
\rho(0)=\rho_0 \mbox{ and } 0 \leq \dot{\rho}(t) < 1  \mbox{ for a.e. } t \in [0,T] .
\end{equation}
For $r < R$, we define $\mathcal{C}_{r,R} := \{ x \in \R^2 \, | \, r < |x| < R \}$ and 
\begin{equation}\label{def:domain-time}
 \mathcal{O}_\rho :=  \{ (t,x) \in [0,+\infty) \times \R^2 \,|\, 0 < t < T, \ x \in \mathcal{C}_{R-\rho(t),R}\}.
\end{equation}
At time $t$, the debonded region is parametrized in the reference configuration on $\mathcal{C}_{R-\rho(t),R}$.

Moreover, we fix a coefficient $\alpha\ge0$ which governs a term related to the friction produced by air resistance on the vibrating film. 
Thus, for $t \in [0, T]$ and $x=(x_1,x_2) \in \mathcal{C}_{R-\rho(t),R}$, the transverse component $u(t,x)$ of the displacement satisfies the following damped wave equation, complemented by initial and boundary conditions:
\begin{equation}\label{princeq}
\begin{dcases}
u_{tt}(t,x) - \Delta_x u(t,x) + \alpha \, u_t(t,x) = 0  \qquad & t \in (0,T), \, R-\rho(t) < |x| < R,
\\
u(t,x) = w(t) \qquad &  t \in (0,T), \, |x|= R,
\\
u(t,x) = 0 \qquad &  t \in (0,T), \, |x|= R-\rho(t) ,
\\
u(0,x) = u_0(x) \qquad &  R-\rho_0 < |x| < R,
\\
u_t(0,x) = u_1(x) \qquad &  R-\rho_0 < |x| < R.
\end{dcases}
\end{equation}
We look for solutions $u \in W^{1,2}(\mathcal{O}_\rho)$. 

We remark that the boundary condition $u(t,x)=0$ on $\{|x|=R-\rho(t)\}$ models the fact that the film is still bonded to the substrate on that circle (and thus, on the whole disk $\{|x|\le R-\rho(t)\}$).
On the other hand, the boundary condition $u(t,x)=w(t)$ on $\{|x|=R\}$ models the fact that the film is peeled off from the substrate through a tension exerted on its edge and an opening displacement $w(t)$, a given function of time corresponding to a time-dependent load. For simplicity we assume that no volume force is present.
The initial conditions on the transverse displacement and its velocity are given by two functions of space $u_0(x)$ and $u_1(x)$, respectively, as usual for the wave equation. 
The requirement that $\dot{\rho}(t) < 1 $ in \eqref{hprho} corresponds to the physical requirement that the debonding speed is subsonic, i.e., less than the wave speed. 

We will seek for radial solutions of problem \eqref{princeq}, i.e., such that there exists $U: [0,T] \times (0,R) \to \R$ with $u(t,x) = U(t,r)$ for every $t \in [0,T]$ and $|x| = r$.  Hence we define the following spaces: 
\begin{align*}
&L^2_{\text{rad}}(\mathcal{C}_{r,R}) :=
\{ u \in L^2(\mathcal{C}_{r,R}) \, | \, \exists \,\tilde{u}: (r,R) \to \R \mbox{ such that } u(x) = \tilde{u}(|x|) \mbox{ for a.e. } x \in \mathcal{C}_{r,R} \},
\\
&W^{1,2}_{\text{rad}}(\mathcal{C}_{r,R}) := \{ u \in W^{1,2}(\mathcal{C}_{r,R}) \, | \, \exists \,\tilde{u}: (r,R) \to \R \mbox{ such that } u(x) = \tilde{u}(|x|) \mbox{ for a.e. } x \in \mathcal{C}_{r,R} \},
\\
&W^{1,2}_{\text{rad}}(B(0,R)):= \{ u \in W^{1,2}(B(0,R)) \, | \, \exists \,\tilde{u}: (0,R) \to \R \mbox{ such that } u(x) = \tilde{u}(|x|) \mbox{ for a.e. } x \in B(0,R) \}.
\end{align*}
We will assume that 
\begin{equation}\label{uinitialcondition}
w \in W^{1,2}(0,T), \quad u_0 \in W^{1,2}_{\text{rad}}(\mathcal{C}_{R-\rho_0,R}), \quad u_1 \in L^2_{\text{rad}}(\mathcal{C}_{R-\rho_0,R}),
\end{equation}
with the compatibility conditions
\begin{equation}\label{uinitialconditioncomp}
u_0(R) = w(0), \quad u_0(R-\rho_0)=0.
\end{equation}
We are now in a position to state the notion of solution to \eqref{princeq}. To this end, a standard argument allows one to make precise the initial condition on the velocity: given a solution $u \in W^{1,2}(\mathcal{O}_\rho)$, we consider the restriction of $u(t)$ to $\mathcal{C}_{\rho_0,R}$; since $u_t,u_x\in L^2((0,T);L^2(\mathcal{C}_{\rho_0,R}))$, then $u_{tt}=\Delta_x u(t,x) - \alpha \, u_t(t,x)\in L^2((0,T);W^{-1,2}(\mathcal{C}_{\rho_0,R}))$,
hence $u_t\in W^{1,2}((0,T);W^{-1,2}(\mathcal{C}_{\rho_0,R}))\subset C^0([0,T];W^{-1,2}(\mathcal{C}_{\rho_0,R}))$. 

\begin{defin}\label{def:sol}
	We say that a function $u \in W^{1,2}(\mathcal{O}_\rho)$ is a solution of \eqref{princeq} if 
	\begin{enumerate}
		\item $u_{tt}(t,x) - \Delta_x u(t,x) + \alpha \, u_t(t,x) = 0 $ holds in the sense of distributions in $\mathcal{O}_\rho$,
		\item the boundary conditions are satisfied in the trace sense,
		\item  the initial conditions are satisfied in the sense of $L^2(\mathcal{C}_{R-\rho_0,R})$ and $W^{-1,2}(\mathcal{C}_{R-\rho_0,R})$ respectively.
	\end{enumerate}
We say that $u$ is a radial solution if in addition $u(t,\cdot)\in W^{1,2}_{\text{rad}}(\mathcal{C}_{R-\rho(t),R})$ for a.e.\ $t\in[0,T]$. 
\end{defin}

It will be more convenient for our analysis to introduce the following transformation:
\begin{equation}\label{defv}
 v(t,r) = u(t,x)  \quad \mbox{ for } t \in (0,T), \,|x|=R-r,
\end{equation}
where the initial conditions are defined in an analogous way:
\begin{equation} \label{eqinitialdatav}
v_0(r) = u_0(x), \quad v_1(r) = u_1(x) \quad \mbox{ with } |x|=R-r. 
\end{equation}
Then, passing to polar coordinates, we get that $u$ solves \eqref{princeq} if and only if $v$ solves the following system:
\begin{equation}\label{princeq2}
\begin{dcases}
v_{tt}(t,r) - v_{rr}(t,r) + \frac{1}{R-r} \, v_r(t,r) + \alpha \, v_t(t,r) = 0  \qquad & t \in (0,T), \, 0 < r < \rho(t),
\\
v(t,0) = w(t) \qquad &  t \in (0,T),
\\
v(t,\rho(t)) = 0 \qquad &  t \in (0,T),
\\
v(0,r) = v_0(r) \qquad &  0 < r < \rho_0,
\\
v_t(0,r) = v_1(r) \qquad &  0 < r < \rho_0.
\end{dcases}
\end{equation}
We notice that if $w,u_0$ and $u_1$ satisfy  \eqref{uinitialcondition} and \eqref{uinitialconditioncomp}, then
\begin{equation}\label{vinitialcondition}
w \in W^{1,2}(0,T), \quad v_0 \in  W^{1,2}(0,\rho_0), \quad v_1 \in L^2(0, \rho_0),  
\end{equation}
and the following compatibility conditions hold:
\begin{equation}\label{vinitialconditioncomp}
v_0(0) = w(0), \quad v_0(\rho_0) = 0.
\end{equation}

\subsection{Equivalent reformulation} \label{sec:prescribed}
We now introduce the function
\begin{equation}\label{defh}
h(t,r) := (R-r)^{\frac{1}{2}} \, e^{\frac{\alpha}{2}t} \,v(t,r) \qquad t \in (0,T), \, 0<r<R .
\end{equation}
Simple computations give the following relations:
\begin{equation*}\label{eqderivativev}
\begin{split}
v_t(t,r) &=  (R-r)^{-\frac{1}{2}} \, e^{-\frac{\alpha}{2}t}\left(h_t(t,r) -\frac{\alpha}{2} \,h(t,r) \right),
\\
v_r(t,r) &=  (R-r) ^{-\frac{1}{2}} \, e^{-\frac{\alpha}{2}t}\left( h_r(t,r) + \frac{1}{2} (R-r)^{-1}\,h(t,r)\right),
\\
v_{tt}(t,r) &=  (R-r)^{-\frac{1}{2}} \, e^{-\frac{\alpha}{2}t}\left(
h_{tt}(t,r) - \alpha \,h_t(t,r) + \frac{\alpha^2}{4} \,h(t,r)  \right),
\\
v_{rr}(t,r) &= (R-r)^{-\frac{1}{2}} \, e^{-\frac{\alpha}{2}t}\left(
h_{rr}(t,r) + (R-r)^{-1} \,h_r(t,r) + \frac{3}{4} (R-r)^{-2} \,h(t,r)\right).
\end{split}
\end{equation*}
Then from \eqref{princeq2} we get the auxiliary problem
\begin{equation}\label{princeqh}
\begin{dcases}
h_{tt}(t,r) - h_{rr}(t,r) - \frac{1}{4} \left( \alpha^2 + \frac{1}{\left(R-r\right)^2}\right) \, h(t,r) = 0  \qquad & t \in (0,T), \, 0 < r < \rho(t),
\\
h(t,0) = z(t) \qquad &  t \in (0,T),
\\
h(t,\rho(t)) = 0 \qquad &  t \in (0,T),
\\
h(0,r) = h_0(r) \qquad &  0 < r < \rho_0,
\\
h_t(0,r) = h_1(r) \qquad &  0 < r < \rho_0,
\end{dcases}
\end{equation}
where the boundary conditions and the initial data are given by
\begin{equation}\label{eqinitialdatah}
\begin{dcases}
z(t) := R^{\frac{1}{2}} \, e^{\alpha\frac{t}{2}} \, w(t),
\\
h_0 (r) := (R-r)^{\frac{1}{2}} \, v_0(r),
\\
h_1 (r) := (R-r)^{\frac{1}{2}} \, \left( v_1(r) +\frac{\alpha}{2} v_0(r)\right).
\end{dcases}
\end{equation}
We mention that if $w,v_0$ and $v_1$ satisfy \eqref{vinitialcondition} and \eqref{vinitialconditioncomp}, then $z,h_0$ and $h_1$ satisfy
\begin{equation}\label{hinitialcondition}
z \in W^{1,2}(0,T), \quad h_0 \in W^{1,2}(0,\rho_0), \quad h_1 \in L^2(0, \rho_0),  
\end{equation}
with the compatibility conditions
\begin{equation}\label{hcompatcondition}
h_0(0) = z(0), \quad h_0(\rho_0) = 0.
\end{equation}
Moreover $u$ is a solution of problem \eqref{princeq} if and only if $h$ is a solution of \eqref{princeqh}.

In \cite{RiNa2020} the authors studied the equations
\begin{equation*} 
h_{tt} - h_{xx} + \alpha \, h_t=0
\quad \text{and} \quad
h_{tt} - h_{xx} - \frac{\alpha^2}4 \, h=0 .
\end{equation*}
When dealing with \eqref{princeqh}, which features a nonconstant kernel multiplying a first derivative, we will follow an approach similar to the one of \cite{RiNa2020},
based on representation formulas for the damped one-dimensional wave equation.
Thus, as in \cite{DMLazNar16} we introduce two functions defined for $t \in [0,T]$, 
\begin{equation*}
\phi(t):= t-\rho(t) \mbox{ and } \psi(t):= t+\rho(t) .
\end{equation*}
Since $\psi$ is strictly increasing, we can then define
\begin{equation*} 
\omega:=[0,T{+}\rho(T)] \to [-\rho_0,T{-}\rho(T)], \quad \omega(t):=\begin{dcases}\phi\circ\psi^{(-1)}(t) & \ \text{if}\ t\ge\rho_0, \\ -\rho_0 & \ \text{if}\  t<\rho_0, \end{dcases} 
\end{equation*}
and we notice that $\omega$ is a Lipschitz function whose derivative satisfies for a.e. $t \in [0,T]$
\begin{equation*}
0 \leq \dot{\omega}(t) = \frac{1-\dot{\rho}(\psi^{(-1)}(t))}{1+\dot{\rho}(\psi^{(-1)}(t))} \leq 1.
\end{equation*}
We introduce the sets
\begin{equation} \label{def:Omega}
\begin{aligned}
&\Omega:= \{ (t,r) \in (0,T)\times(0,R) \,|\, 0 < r < \rho(t) \},
\\
&\Omega'_1:=  \{ (t,r) \in \Omega \,|\, t \leq r \mbox{ and } t+r \leq \rho_0 \},
\\
&\Omega'_2:=  \{ (t,r) \in \Omega \,|\, t > r \mbox{ and } t+r < \rho_0 \},
\\
&\Omega'_3:=  \{ (t,r) \in \Omega \,|\, t < r \mbox{ and } t+r > \rho_0 \},
\\
&\Omega':= \Omega'_1 \cup \Omega'_2 \cup \Omega'_3.
\end{aligned}
\end{equation}
Moreover we define the dependence cone of the point $(t,r)$, given by 
 \begin{figure}[b] 
 \centering
{\tiny
\subfloat[]{
 \psfrag{r}{$r$}
 \psfrag{t}{$t$}
 \psfrag{L}{$\rho$}
 \psfrag{0}{\hspace{.5em}$\rho_0$}
 \psfrag{1}{$(t,r)$}
 \psfrag{2}{\hspace{-.5em}$r{-}t$}
 \psfrag{3}{\hspace{-.5em}$t{+}r$}
 \includegraphics[width=.3\textwidth]{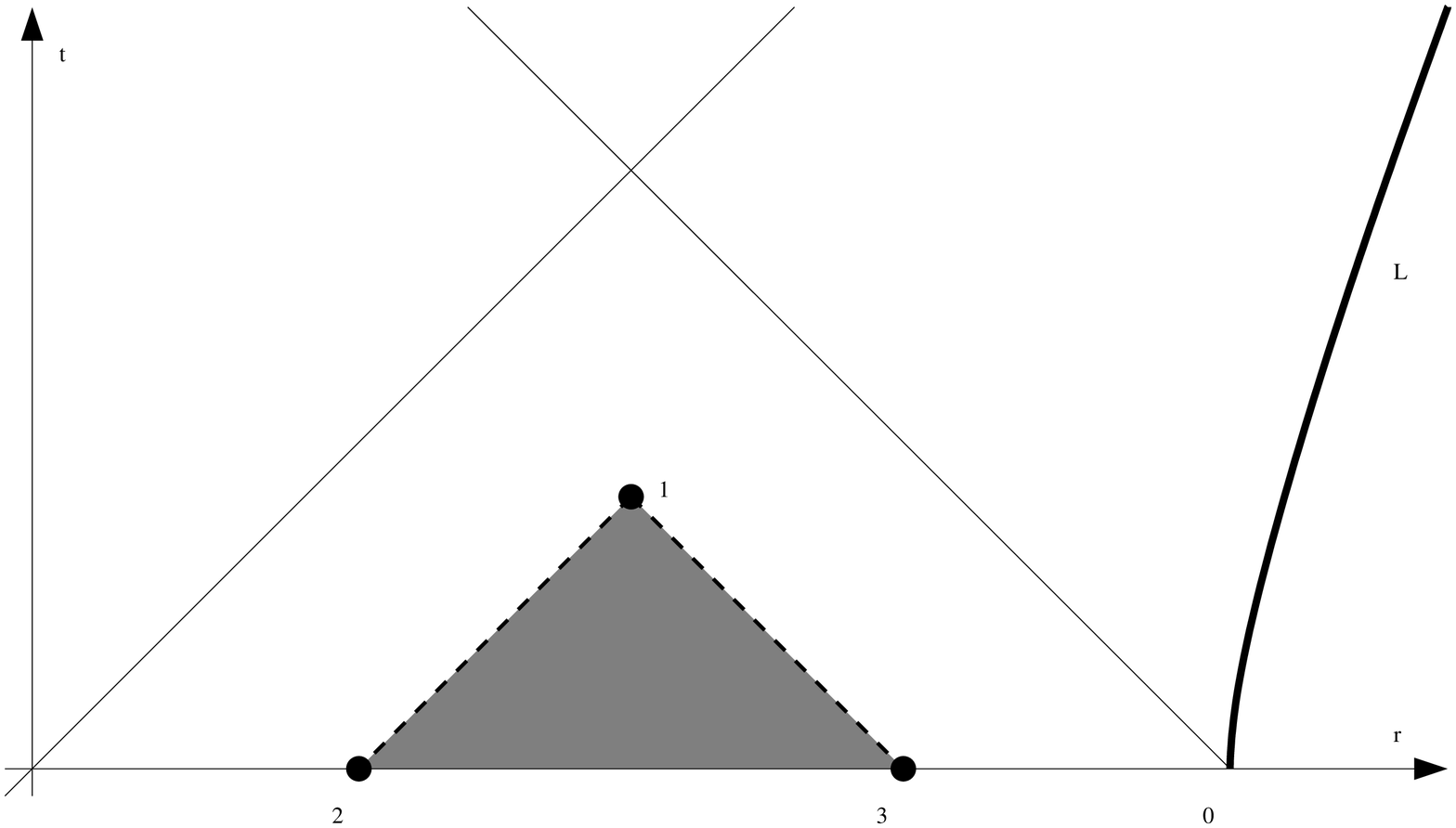}
}
\hspace{.02\textwidth}
\subfloat[]{
 \psfrag{r}{$r$}
 \psfrag{t}{$t$}
 \psfrag{L}{$\rho$}
 \psfrag{0}{\hspace{.5em}$\rho_0$}
 \psfrag{3}{\hspace{-.5em}$t{+}r$}
 \psfrag{4}{$(t,r)$}
 \psfrag{5}{$t{-}r$}
 \psfrag{6}{\hspace{-.65em} $t{-}r$}
 \includegraphics[width=.3\textwidth]{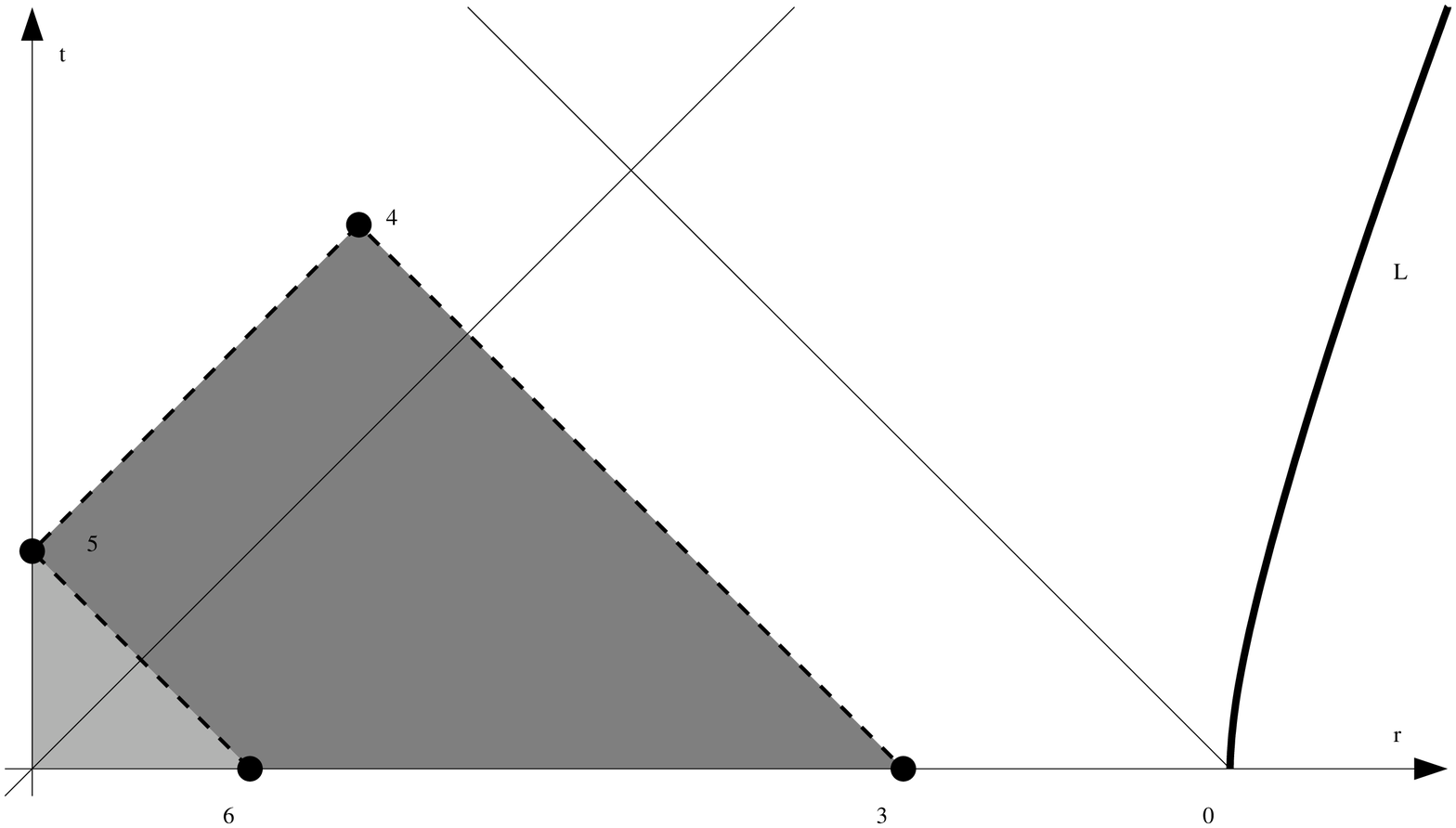}
}
\hspace{.02\textwidth}
\subfloat[]{
 \psfrag{r}{$r$}
 \psfrag{t}{$t$}
 \psfrag{L}{$\rho$}
 \psfrag{p}{$\psi^{-1}(t{+}r)$}
 \psfrag{0}{\hspace{.5em}$\rho_0$}
 \psfrag{2}{\hspace{-.5em}$r{-}t$}
 \psfrag{7}{$(t,r)$}
 \psfrag{9}{\hspace{-1.5em}$|\omega(t{+}r)|$}
 \includegraphics[width=.3\textwidth]{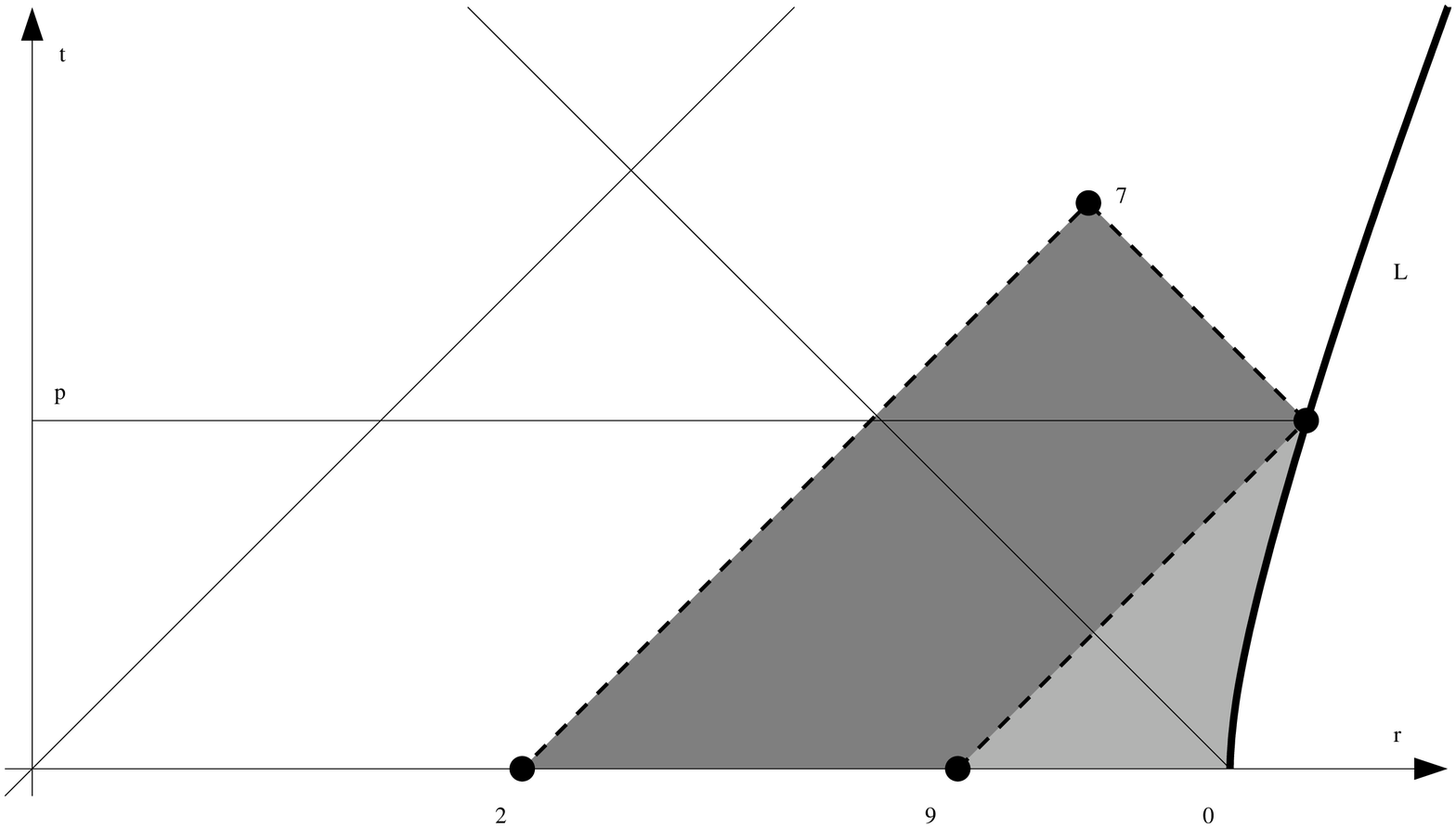}
}
}
 \caption{Sets appearing in formula \eqref{propformulah} in the three cases 
$(t,r)\in\Omega'_1$ ({\sc a}), $(t,r)\in\Omega'_2$ ({\sc b}), $(t,r)\in\Omega'_3$ ({\sc c}). 
For convenience the $t$-axis is vertical.
In each case the region in dark gray is the set $P(t,r)$, while the dependence cone $C(t,r)$ is the union of the regions in dark and light gray.}
 \label{fig:cone}
 \end{figure}
\begin{equation}\label{C(t,r)}
	C(t,r) := \{(\tau,\sigma) \in \Omega \,|\, 0\leq \tau\leq t \mbox{ and } r-t+\tau \leq \sigma \leq r+t-\tau \}.
\end{equation}
It will turn useful to define the set $P(t,r)$ as 
\begin{equation}\label{R(t,r)}
	P(t,r) := 
	\begin{cases}
		C(t,r) &\mbox{ if } t<r \mbox{ and } t+r < \rho_0,
		\\
		C(t,r) \setminus C(t-r,0) &\mbox{ if } t>r \mbox{ and } t+r < \rho_0,
		\\
		C(t,r) \setminus C(\psi^{-1}(t{+}r),\rho(\psi^{-1}(t{+}r))) &\mbox{ if } t<r \mbox{ and } t+r > \rho_0.
	\end{cases}
\end{equation}
See Figure \ref{fig:cone}.

 Consider now the undamped wave equation in the time-dependent interval $(0,\rho(t))$, complemented with initial and boundary conditions as in \eqref{princeqh}, 
\begin{equation} \label{eq:undamp}
\begin{dcases}
\h_{tt}(t,r) - \h_{rr}(t,r) = 0  \qquad & t \in (0,T), \, 0 < r < \rho(t),
\\
\h(t,0) = z(t) \qquad &  t \in (0,T),
\\
\h(t,\rho(t)) = 0 \qquad &  t \in (0,T),
\\
\h(0,r) = h_0(r) \qquad &  0 < r < \rho_0,
\\
\h_t(0,r) = h_1(r) \qquad &  0 < r < \rho_0.
\end{dcases}
\end{equation}
 In \cite{DMLazNar16} it has been shown that this problem has a unique solution $\h$, which satisfies the following d'Alembert formula in $\Omega'$: 
\begin{equation}\label{A(t,r)}
\h(t,r)=
\begin{dcases}
\frac{1}{2} h_0(r{-}t) +\frac{1}{2} h_0(r{+}t) + \frac{1}{2} \int_{r-t}^{r+t} h_1(s) \,ds &\mbox{ if }(t,r) \in\Omega'_1,
\\
z(t{-}r) -\frac{1}{2} h_0(t{-}r) +\frac{1}{2} h_0(r{+}t) + \frac{1}{2} \int_{t-r}^{t+r} h_1(s) \,ds &\mbox{ if }(t,r) \in\Omega'_2,
\\
\frac{1}{2} h_0(r{-}t) - \frac{1}{2} h_0(-\omega(r{+}t)) + \frac{1}{2} \int_{r-t}^{-\omega(r+t)} h_1(s) \,ds &\mbox{ if }(t,r) \in\Omega'_3.
\end{dcases}
\end{equation}
Indeed, according to the principle of causality, $\h$ depends on the data interior to the  dependence cone.

By a standard computation (which we detail for the reader's convenience), we may derive a law holding for any solution of problem \eqref{princeqh}.

\begin{prop}\label{propformulah}
	A function $h \in W^{1,2}(\Omega')$ is a solution of problem \eqref{princeqh} in $\Omega'$ if and only if 
	\begin{equation}\label{formulah}
	h(t,r) = \h(t,r) + \frac{1}{2} \iint_{P(t,r)} \frac{1}{4} \left( \alpha^2 + \frac{1}{\left(R-\sigma\right)^2}\right) \, h(\tau , \sigma) \,d\sigma \,d\tau  \quad \mbox{ for a.e.}\,(t,r) \in \Omega' ,
	\end{equation}
	where $P(t,r)$ and $\h$ are given by \eqref{R(t,r)} and \eqref{A(t,r)}, respectively.
\end{prop}
\begin{proof}
	Let $h \in W^{1,2}(\Omega')$ be a solution of problem \eqref{princeqh} and consider the following change of variables,
	\begin{equation}\label{transaffine}
	\begin{cases}
	\xi = t-r,
	\\
	\eta = t+r.
	\end{cases}
	\end{equation}
We remark {\it en passant} that the line $t=0$ corresponds to $\xi+\eta=0$.
	Then the function $\tilde{h}(\xi,\eta) := h\left( \frac{\xi+\eta}{2}, \frac{\eta-\xi}{2} \right) $ satisfies 
	\begin{equation}\label{h_xi,eta = h}
	\partial_{\xi,\eta} \, \tilde{h}= \frac{1}{16} \left( \alpha^2 + \frac{1}{\left( R - \frac{\eta-\xi}{2}\right)^2}\right) \, \tilde{h}
	\end{equation}
	in the sense of distributions in $\Lambda'$, where $\Lambda'$ is the image of $\Omega'$ through the affine transformation \eqref{transaffine}. 
Let us denote by $\Lambda'_i$ the image of $\Omega'_i$ through the change of coordinates ($i=1,2,3$).
Now define the right-hand side of \eqref{h_xi,eta = h} by $\widetilde{H}$ and denote by $\widetilde{C}(\xi,\eta)$ and $\widetilde{P}(\xi,\eta)$  the sets given by \eqref{C(t,r)} and \eqref{R(t,r)}, respectively, in terms of the new coordinates. 
We fix a point $(\overline{t},\overline{r})\in\Omega'$, corresponding to $(\overline{\xi},\overline{\eta})\in\Lambda'$, and distinguish the three cases where such point lies in $\Lambda_1'$, $\Lambda_2'$, or $\Lambda_3'$. 
	
	{\bf (1) Case $(\overline{\xi},\overline{\eta}) \in \Lambda_1'$:} Observe that in this case one has $\overline{\xi}\le 0$ and  $\widetilde{P}(\overline{\xi},\overline{\eta})=\widetilde{C}(\overline{\xi},\overline{\eta})$ is defined by the conditions  $-\overline{\eta}\le \xi \le \overline{\xi}$ and $|\xi|\le \eta\le \overline{\eta}$, where we took into account that $\xi\le 0$ on the given domain (see Figure \ref{fig:proof}).	
	By  double integration one gets:	
	\begin{equation*}
	\tilde{h} (\overline{\xi},\overline{\eta}) = 
	\tilde{h}(-\overline{\eta},\overline{\eta}) + \int_{-\overline{\eta}}^{\overline{\xi}} \tilde{h}_\xi(\xi,|\xi|) \,d\xi +  \iint_{\widetilde{P}(\overline{\xi},\overline{\eta})} \widetilde{H}(\xi,\eta) \,d\xi\,d\eta.
	\end{equation*}
	Then, by a change of variables in the last integral from the domain of integration $\widetilde{P}(\overline{\xi},\overline{\eta})$ to $P(\overline{t},\overline{r})$, by observing that $\tilde{h}(-\overline{\eta},\overline{\eta}) = h_0(\overline{r} +\overline{t})$ and by computing 
	\begin{align*}
	\int_{-\overline{\eta}}^{\overline{\xi}} \tilde{h}_\xi(\xi,|\xi|) \,d\xi & = - \frac{1}{2} \int_{\overline{r}+\overline{t}}^{\overline{r} - \overline{t}} (h_t(0,r) - h_r(0,r) ) \,dr 
	\\&= \frac{1}{2} h_0(\overline{r} - \overline{t}) - \frac{1}{2} h_0(\overline{r}+\overline{t}) + \frac{1}{2} \int_{\overline{r} - \overline{t}}^{\overline{r}+\overline{t}} h_1(s) \,ds,
	\end{align*}
	one obtains that \eqref{formulah} holds in $\Omega'_1$ (see also Figure \ref{fig:cone}).

 \begin{figure} 
 \centering
 \psfrag{r}{$r$}
 \psfrag{t}{$t$}
 \psfrag{L}{$\rho$}
 \psfrag{x}{$\xi$}
 \psfrag{y}{$\eta$}
 \psfrag{0}{\hspace{.5em}$\rho_0$}
 \psfrag{1}{$(\overline\xi,\overline\eta)$}
 \psfrag{2}{\hspace{-.5em}$(\overline\xi,|\overline\xi|)$}
 \psfrag{3}{\hspace{-2em}$(-\overline\eta,\overline\eta)$}
 \psfrag{4}{$(\underline\xi,\overline\eta)$}
 \psfrag{5}{$(\underline\xi,\underline\xi)$}
 \psfrag{6}{\hspace{-1.5em} $(-\underline\xi,\underline\xi)$}
 \psfrag{7}{$(\overline\xi,\underline\eta)$}
 \psfrag{8}{$(\underline\omega,\underline\eta)$}
 \psfrag{9}{$(\underline\omega,|\underline\omega|)$}
 \includegraphics[width=.9\textwidth]{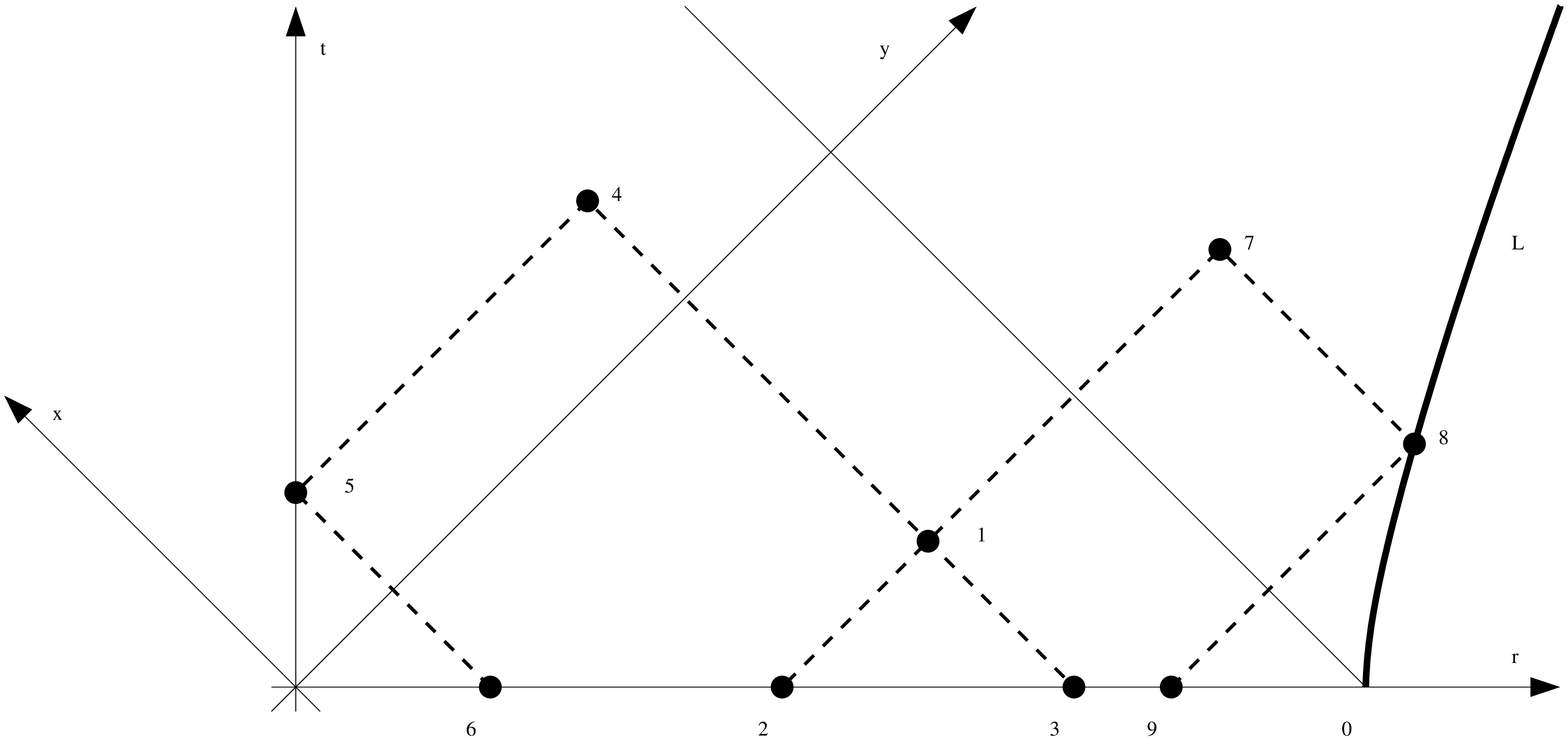}
 \caption{Some points appearing in the proof of Proposition \ref{propformulah}, represented in the $(t,r)$-plane (with $r$ on the horizontal axis) and marked using the $(\xi,\eta)$-coordinates. The bold curve is the graph of $t\mapsto\rho(t)$. In the picture we consider three points $(\overline\xi,\overline\eta)\in\Lambda'_1$, $(\underline\xi,\overline\eta)\in\Lambda'_2$, $(\overline\xi,\underline\eta)\in\Lambda'_3$;
notice that $\overline\xi<0<\underline\xi$. 
The other points, lying on the characteristic lines,  are the boundary points in the double integrations in the proof. 
Here we use the shorthand notation $\underline\omega:=\omega(\underline\eta)$.}
 \label{fig:proof}
 \end{figure}

	{\bf (2) Case $(\overline{\xi},\overline{\eta}) \in \Lambda_2'$:}  Observe that in this case  $\widetilde{P}(\overline{\xi},\overline{\eta})$ is defined by the conditions  $\overline{\xi}\le \eta\le \overline{\eta}$ and  $-\eta\le \xi \le \overline{\xi}$ (see Figure \ref{fig:proof}).
	By  double integration one gets:
	\begin{equation*}
	\tilde{h} (\overline{\xi},\overline{\eta}) = 
	\tilde{h}(\overline{\xi},\overline{\xi}) + \int_{\overline{\xi}}^{\overline{\eta} }\tilde{h}_\eta(-\eta,\eta) \,d\eta +  \iint_{\widetilde{P}(\overline{\xi},\overline{\eta})} \widetilde{H}(\xi,\eta) \,d\xi\,d\eta.
	\end{equation*}	
	Then, by changing the variables in the last integral from the domain of integration $\widetilde{P}(\overline{\xi},\overline{\eta})$ back to $P(\overline{t},\overline{r})$, by observing that
	$\tilde{h}(\overline{\xi},\overline{\xi}) = z(\overline{t} - \overline{r})$ and $\tilde{h}(-\overline{\eta},\overline{\eta}) = h_0( \overline{t} + \overline{r})$ and by computing 
\begin{align*}
	\int_{\overline{\xi}}^{\overline{\eta}} \tilde{h}_\eta(-\eta,\eta) \,d\eta& =  \frac{1}{2} \int_{\overline{t}-\overline{r}}^{\overline{t} + \overline{r}} (h_t(0,r) + h_r(0,r) ) \,dr 
	\\&= -\frac{1}{2} h_0(\overline{t} - \overline{r}) + \frac{1}{2} h_0(\overline{r}+\overline{t}) + \frac{1}{2} \int_{\overline{t} - \overline{r}}^{\overline{r}+\overline{t}} h_1(s) \,ds,
	\end{align*}
	one obtains that \eqref{formulah} holds in $\Omega'_2$  (see also Figure \ref{fig:cone}).
	
	{\bf (3) Case $(\overline{\xi},\overline{\eta}) \in \Lambda_3'$:} Observe that in this case one has $\overline{\xi}< 0$ and  $\widetilde{P}(\overline{\xi},\overline{\eta})$ is defined by the conditions  $\omega(\overline{\eta})\le \xi \le \overline{\xi}$ and $|\xi|\le \eta\le \overline{\eta}$, where we took into account that $\xi<0$ on the given domain (see Figure \ref{fig:proof}). 
	By  double integration one gets:
	\begin{equation*}
	\tilde{h} (\overline{\xi},\overline{\eta}) = 
	\tilde{h}(\omega(\overline{\eta}),\overline{\eta}) + \int_{\omega(\overline{\eta})}^{\overline{\xi}} \tilde{h}_\xi(\xi,|\xi|) \,d\xi +  \iint_{\widetilde{P}(\overline{\xi},\overline{\eta})} \widetilde{H}(\xi,\eta) \,d\xi\,d\eta.
	\end{equation*}
	Then, by a change of variables in the last integral from $\widetilde{P}(\overline{\xi},\overline{\eta})$ to $P(\overline{t},\overline{r})$, by observing that $\tilde{h}(\omega(\overline{\eta}),\overline{\eta})=0$ and by computing
	\begin{align*}
	\int_{\omega(\overline{\eta})}^{\overline{\xi}} \tilde{h}_\xi(\xi,|\xi|) \,d\xi &= - \frac{1}{2} \int_{-\omega(\overline{r}+\overline{t})}^{\overline{r}-\overline{t}} (h_t(0,r) - h_r(0,r) )(r) \,dr 
	\\&= \frac{1}{2} h_0(\overline{r}-\overline{t}) - \frac{1}{2} h_0(-\omega(\overline{r}+\overline{t})) + \frac{1}{2} \int_{\overline{r}-\overline{t}}^{-\omega(\overline{r}+\overline{t})} h_1(s) \,ds,
	\end{align*}
	one obtains that \eqref{formulah} holds in $\Omega'_3$ (see also Figure \ref{fig:cone}).
	
Finally, to prove the converse implication it is sufficient to show that
$J(\overline\xi,\overline\eta):=\iint_{\widetilde{P}(\overline{\xi},\overline{\eta})} \widetilde{H}(\xi,\eta) \,d\xi\,d\eta$ solves
$\partial^2_{\xi,\eta} J(\overline\xi,\overline\eta)=\widetilde H(\overline\xi,\overline\eta)$.
We detail the computation only in case {\bf (3)}, the others being similar. In this case, 
\[
J(\overline\xi,\overline\eta)= \int_{\omega(\overline{\eta})}^{\overline{\xi}} \underbrace{\int_{|\xi|}^{\overline\eta} \widetilde H(\xi,\eta) \,d\eta}_{=:j(\xi,\overline\eta)} \,d\xi .
\]
We then have $\partial_{\eta} J(\overline\xi,\overline\eta)=-j(\omega(\overline\eta),\overline\eta)\,\dot\omega(\overline\eta) + 
\int_{\omega(\overline{\eta})}^{\overline{\xi}} \widetilde H(\xi,\overline\eta) \,d\xi$.
By deriving with respect to $\overline\xi$ the conclusion follows immediately.
\end{proof}

Below we shall use the following formula for the solutions of problem \eqref{princeq2}.

\begin{rem}\label{remv}
	Let $v$ be a solution of problem \eqref{princeq2}. Let $\vi$ be the  solution of the pure wave equation  given by \eqref{eq:undamp} replacing the boundary data $h_0$, $h_1$, and $z$ by $v_0$, $v_1$, and $w$, respectively.  Then the very same argument of Proposition \ref{propformulah} shows that 
	\begin{equation}\label{formulav}
	v(t,r) = \vi(t,r) + \frac{1}{2} \iint_{P(t,r)} \left( -\frac{v_r(\tau , \sigma)}{\left(R{-}r\right)} - \alpha \, v_t(\tau , \sigma) \right) \,d\sigma \,d\tau  \quad \mbox{ for a.e.}\,(t,r) \in \Omega' .
	\end{equation}
\end{rem}

We summarize the regularity of the terms appearing in \eqref{formulah} in the following lemma. (For a detailed proof we refer to \cite[Lemmas 1.10 and 1.11]{RiNa2020}.)

\begin{lemma}\label{lemmaAandPhi}
	The following hold true.
	\begin{enumerate}
		\item[(i)] Let $\h$ be defined  as in \eqref{eq:undamp}. Then $\h \in C^0(\overline{\Omega}{}') \cap W^{1,2}(\Omega')$. Moreover, setting $\h \equiv 0$ outside $\overline{\Omega}$,
		\begin{equation*}
		\h \in C^0( [0,T] ; W^{1,2}(0,R)) \cap C^1( [0,T] ; L^2(0,R)).
		\end{equation*}
		
		\item[(ii)] Let $H \in L^2(\Omega')$ and for every $(t,r) \in\Omega'$ define
		\begin{equation}\label{defmathcalF}
		\Phi[H](t,r):=\iint_{P(t,r)} H(\tau,\sigma) \,d\tau d\sigma .
		\end{equation}
		Then $\Phi[H] \in C^0(\overline{\Omega}{}')\cap W^{1,2}(\Omega')$. Moreover, setting $\Phi[H]\equiv 0$ outside $\overline{\Omega}$,
		\begin{equation*}
		\Phi[H] \in C^0( [0,T] ; W^{1,2}(0,R)) \cap C^1( [0,T] ; L^2(0,R)).
		\end{equation*}
	\end{enumerate}
	
\end{lemma}

\subsection{Existence of solutions for prescribed debonding front} 
In order to obtain existence and uniqueness of a solution to problem \eqref{princeqh}, 
we first seek a solution in a small time interval $[0,\widetilde{T}]\subset[0,T]$; afterwards we shall extend the solution to $[0,T]$. 
As a consequence, we have existence and uniqueness of a \emph{radial} solution to problem \eqref{princeq};
moreover, if $\rho$ is sufficiently regular, the results in Appendix \ref{app:uniq} show that such radial function is the unique solution to problem \eqref{princeq} in the sense of Definition \ref{def:sol}.
The solution to \eqref{princeqh} is found as a fixed point of a certain linear operator. More precisely, let
\begin{equation*}
\mathcal{Y} := \{ h \in C^0(\overline{\Omega}{}') \,|\, h \mbox{ satisfies the initial and boundary conditions in } \eqref{princeqh} \}
\end{equation*}
and consider 
\begin{equation*}
\mathcal{L}: \, \mathcal{Y} \to C^0(\overline{\Omega}{}'), \, h \mapsto \mathcal{L}[h] 
\end{equation*}
defined by
\begin{equation}\label{formulaL}
\mathcal{L}[h](t,r):= \h(t,r) + \frac{1}{2} \iint_{P(t,r)} \frac{1}{4}\left( \alpha^2 + \frac{1}{\left(R {-}\sigma\right)^2}\right) \, h(\tau , \sigma) \,d\sigma \,d\tau \quad \mbox{ for a.e.}\,(t,r) \in \Omega'.
\end{equation}
Then we have the following result, the proof of which is based on \cite[Proposition 1.13]{RiNa2020}; we detail the proof in order to show  how the kernel $\frac{1}{4}\left( \alpha^2 + \frac{1}{\left(R{-}\sigma\right)^2}\right)$ affects the constant in \eqref{cond rhoT/4}.
We employ the following notation:
\[
\Omega_{S}:= \{ (t,r) \in \Omega \,|\, t < S \}
\]
for $S\in(0,T)$.

\begin{prop}\label{propcontration}
	Let $\tilde{\rho} := \min \left\{\frac{\rho_0}{2} , \frac{R {-}\rho_0}{2}\right\}$. Let $\widetilde{T} \in (0, \tilde{\rho})$ satisfy
	\begin{equation}\label{cond rhoT/4}
	\frac{\rho_0 \widetilde{T}}{4} \left( \alpha^2 + \frac{4}{\left(R{-}\rho_0\right)^2}\right) < 1.
	\end{equation}
	Then operator $\mathcal{L}$ defined by \eqref{formulaL} is a contraction from $\mathcal{Y} \cap C^0(\overline{\Omega}_{\widetilde{T}})$ into $C^0(\overline{\Omega}_{\widetilde{T}})$.
\end{prop}

\begin{proof}
	By Lemma \ref{lemmaAandPhi} the operator $\mathcal{L}$ maps $\mathcal{Y} \cap C^0(\overline{\Omega}_{\widetilde{T}})$ into itself. Now let $h^{(1)},h^{(2)} \in \mathcal{Y} \cap C^0(\overline{\Omega}_{\widetilde{T}})$ and let $(t,r) \in \overline{\Omega}_{\widetilde{T}}$.  Since $h^{(1)},h^{(2)} \in \mathcal{Y}$, one observes that in $\mathcal{L}[h^{(1)}]-\mathcal{L}[h^{(2)}]$ there is a cancellation of the term $\h$ which only depends on the boundary conditions; hence, 
	\begin{equation*}
	\begin{split}
	|\mathcal{L}[h^{(1)}](t,r) - \mathcal{L}[h^{(2)}](t,r)| &\leq \frac{1}{8}\left|\iint_{P(t,r)} \left( \alpha^2 + \frac{1}{\left(R {-} \sigma\right)^2}\right) \, (h^{(1)}(\tau , \sigma) - h^{(2)}(\tau,\sigma)) \,d\sigma \,d\tau\right|
	\\
	&\leq \frac{1}{8} \left(\max_{(\tau,\sigma) \in P(t,r)} \left( \alpha^2 + \frac{1}{\left(R{-}\sigma \right)^2}\right) \right) \iint_{P(t,r)} | (h^{(1)}(\tau , \sigma) - h^{(2)}(\tau,\sigma)) | \,d\sigma \,d\tau
	\\
	&\leq \frac{ |{\Omega}_{\widetilde{T}}|}{8} \left( \alpha^2 + \frac{1}{(R{-}\rho(\widetilde{T}))^2}\right)  \|h^{(1)}-h^{(2)}\|_{ C^0(\overline{\Omega}_{\widetilde{T}})} 
	\\
	&\leq \frac{\rho_0 \widetilde{T}}{4} \left( \alpha^2 + \frac{4}{\left(R{-}\rho_0\right)^2}\right)  \|h^{(1)}-h^{(2)}\|_{ C^0(\overline{\Omega}_{\widetilde{T}})} .
	\end{split}
	\end{equation*}
 In the last inequality we have used the following facts: 
	\begin{itemize}
		\item Since $0 < \dot{\rho} < 1$ and  $\widetilde{T} < \tilde{\rho}$, it follows $\rho(\widetilde{T}) \leq 2\rho_0$, thus $|\Omega_{\widetilde{T}}| \leq 2\rho_0 \widetilde{T}$;
		\item Since $0 < \dot{\rho} < 1$ and $\widetilde{T} < \frac{R -\rho_0}{2}$, one has $\rho(\widetilde{T})  \leq \frac{R +\rho_0}{2}$, hence $R-\rho(\widetilde{T})  \geq \frac{R -\rho_0}{2}$.
	\end{itemize} 
	By assumption \eqref{cond rhoT/4} we conclude.
\end{proof}

\begin{teo}\label{teoexistenceh}
	Let $\rho$ be as in \eqref{hprho}. 
 Assume \eqref{hinitialcondition} and \eqref{hcompatcondition}. 
Then there exists a unique solution $h \in W^{1,2}(\Omega)$ of problem \eqref{princeqh}. Moreover $h$ satisfies \eqref{formulah}, it has a continuous representative on $\overline{\Omega}$, still denoted by $h$, and, setting $h\equiv 0$ outside $\overline{\Omega}$, it holds:
	\begin{equation*}
	h \in C^0([0,T]; W^{1,2}(0,R)) \cap C^1([0,T]; L^2(0,R)).
	\end{equation*}
\end{teo}

\begin{proof}
	Define
	\begin{equation*}
	T_1 := \frac{1}{2} \min\left\{\frac{\rho_0}{2}, \frac{R {-}\rho_0}{2}, \frac{4}{\rho_0} \left( \alpha^2 + \frac{4}{\left(R{-}\rho_0\right)^2}\right)^{-1} \right\}
	\end{equation*}
	Then by Proposition \ref{propcontration} we deduce the existence of a unique continuous function $h^{(1)}$ satisfying \eqref{formulah} in $ \overline{\Omega}_{T_1}$. By Lemma \ref{lemmaAandPhi} we deduce that $h^{(1)} \in W^{1,2}(\Omega_{T_1})$; moreover,
	\begin{equation}\label{201103}
	h^{(1)} \in C^0([0,T_1]; W^{1,2}(0,R)) \cap C^1([0,T_1]; L^2(0,R)).
	\end{equation}
	Proposition \ref{propformulah} ensures that $h^{(1)}$ solves problem \eqref{formulah} in $\Omega_{T_1}$.
	
	Now we can restart the argument from time $T_1$ replacing $\rho_0$ by $\rho_1=\rho(T_1)$, $h_0$ by $h^{(1)}(T_1,\cdot)$ and $h_1$ by $h^{(1)}_t(T_1,\cdot)$; indeed by \eqref{201103} it follows that $h^{(1)}(T_1,\cdot) \in W^{1,2}(0,\rho_1)$ and $h^{(1)}_t(T_1,\cdot) \in L^2(0,\rho_1)$ and that they satisfy the compatibility conditions $h^{(1)}(T_1,0)=z(T_1)$ and $h^{(1)}(T_1,\rho_1)=0$. 
	Arguing as before, we get the existence of a unique solution $h^{(2)}$ of \eqref{princeqh} in $\Omega_{T_2}\setminus\Omega_{T_1}$ with 
	\begin{equation*}
	T_2 := T_1 + \frac{1}{2} \min\left\{\frac{\rho_1}{2}, \frac{R {-}\rho_1}{2}, \frac{4}{\rho_1} \left( \alpha^2 + \frac{4}{\left(R{-}\rho_1\right)^2}\right)^{-1} \right\},
	\end{equation*}
	and
	\begin{equation*}
	h^{(2)} \in C^0([T_1,T_2]; W^{1,2}(0,R)) \cap C^1([T_1,T_2]; L^2(0,R)).
	\end{equation*}
	Then the function
	\begin{equation*}
	\overline{h}(t,r) =
	\begin{dcases}
	h^{(1)}(t,r) &\mbox{ if } (t,r) \in \overline{\Omega}_{T_1},
	\\
	h^{(2)}(t,r) &\mbox{ if }  (t,r) \in \overline{\Omega}_{T_2} \setminus \overline{\Omega}_{T_1},
	\end{dcases}
	\end{equation*}
	belongs to $C^0([0,T_2]; W^{1,2}(0,R)) \cap C^1([0,T_2]; L^2(0,R))$ and it is the only solution of \eqref{princeqh} in $\Omega_{T_2}$.
	
	To conclude and prove existence and uniqueness of a solution $h \in W^{1,2}(\Omega)$ of problem \eqref{princeqh}, defined on the whole $[0, T]$, we only need to show that, iterating the procedure, we may reach the fixed time horizon $T$ in a finite number of step. Indeed, define recursively 
	\begin{equation*}
	\begin{dcases}
	T_{k+1} = T_k + \frac{1}{2} \min\left\{\frac{\rho_k}{2}, \frac{R {-}\rho_k}{2}, \frac{4}{\rho_k} \left( \alpha^2 + \frac{4}{\left(R{-}\rho_k\right)^2}\right)^{-1} \right\}, \quad \mbox{if } k\geq 1,
	\\
	T_0 = 0,
	\end{dcases}
	\end{equation*}
	with $\rho_k = \rho(T_k)$ for every $k \in \N$. This family is well defined as long as $T_k \le T$. 
 Since $\rho_0\le\rho_k\le\rho(T)<R$, the remainder of the sequence $T_{k+1} - T_k$ is bounded from below by a positive constant, hence $T_{k+1}>T$ for some $k\ge1$.
\end{proof}

The following corollary is an immediate consequence of Theorem \ref{teoexistenceh}.

\begin{cor} \label{corollario}
	Let $\rho$ be as in \eqref{hprho}. 
\begin{itemize}
\item[(i)] Assume \eqref{vinitialcondition} and \eqref{vinitialconditioncomp}. 
Then there exists a unique solution $v \in W^{1,2}(\Omega)$ of problem \eqref{princeq2}. Moreover, $v$ satisfies \eqref{formulav}, it has a continuous representative on $\overline{\Omega}$, still denoted by $v$, and, setting $v\equiv 0$ outside $\overline{\Omega}$, it holds:
	\begin{equation*}
	v \in C^0([0,T]; W^{1,2}(0,R)) \cap C^1([0,T]; L^2(0,R)).
	\end{equation*}
\item[(ii)] Assume \eqref{uinitialcondition} and \eqref{uinitialconditioncomp}. 
Then there exists a unique radial solution $u \in W^{1,2}(\mathcal{O}_\rho)$ of problem \eqref{princeq2}. 
\end{itemize}
\end{cor}

\begin{rem} \label{rmk:uniq}
The results of this section ensure uniqueness of solutions to \eqref{princeq} only among \emph{radial} functions.
However, if the prescribed debonding front $\rho$ is of class $C^{2,1}$, we may apply Proposition \ref{prop-app-uniq} in the Appendix below. By combining such results, we obtain that there is a unique solution $u$ to \eqref{princeq} in the sense of Definition \ref{def:sol} (without a priori restrictions on the solutions). Moreover, such solution is radial and the function $(t,r)\mapsto u(t,(R{-}r,0))$ satisfies \eqref{formulav}.
\end{rem}

\section{Energy criterion for the debonding evolution}
The flow rule governing the evolution of the debonding front is based on a stability criterion involving the kinetic energy, the potential energy, and the dissipations. We define the energy terms in the setting of problem \eqref{princeq}, thus 
we use now the $x$-coordinates in the plane. For given $\rho: [0,T] \to [\rho_0,R)$ satisfying
\eqref{hprho}, let $u \in W^{1,2}(\mathcal{O}_\rho)$ be the unique radial solution found in Corollary \ref{corollario}(ii),  corresponding to the data $u_0$, $u_1$, and $w$ introduced in \eqref{uinitialcondition} and \eqref{uinitialconditioncomp}.
For $t \in[0,T]$ the internal energy is given by
\begin{equation}\label{eqE}
\mathcal{E}(t) := \frac{1}{2} \iint_{R - \rho(t) <|x|<R}(u_t^2(t,x) + |\nabla_x u(t,x)|^2 ) \,dx
\end{equation}
and the energy dissipated by the friction of air is given by
\begin{equation}\label{eqA}
\mathcal{A}(t):= \alpha \int_{0}^{t} \iint_{R - \rho(\tau) <|x|<R} u_t^2(\tau,x) \,dx\,d\tau.
\end{equation}
We define the total energy of $u$,
\begin{equation}\label{defTotal}
\mathcal{T}(t):= \mathcal{E}(t) + \mathcal{A}(t).
\end{equation}

\subsection{Energy balance}
We now provide a formula for the time derivative of the total energy, which gives a first expression for the energy balance. To this end it is convenient to resort to the functions $v$ and $h$ defined in \eqref{defv} and \eqref{defh}, respectively. The  terms containing first derivatives in the equation for $v$, \eqref{princeq2}, are denoted by 
	\begin{equation}\label{eqkernelFv}
	G(\tau,\sigma):=  -\frac{1}{R{-}\sigma} \, v_r(\tau , \sigma) - \alpha \,v_t(\tau , \sigma) \in L^2(\Omega') .
	\end{equation} 
 The corresponding term in the equation for $h$, \eqref{princeqh}, is denoted by
	\begin{equation}\label{eqkernelFh}
	F(\tau,\sigma):=  \frac{1}{4} \left( \alpha^2 + \frac{1}{\left(R{-}\sigma\right)^2}\right)  h(\tau , \sigma) \in L^2(\Omega') .
	\end{equation}
Moreover, by Corollary \ref{corollario}(i), we have that 
	\begin{equation}\label{regularityGF}
	 G \in  C^0([0,T]; L^2(0,R)) , \qquad F \in  C^0([0,T]; W^{1,2}(0,R)).
	\end{equation}
Henceforth, for simplicity we provide some formulas only in the interval $[0,\frac{\rho_0}{2}]$: in fact,
in the subsequent results we argue in small time intervals, as done in Proposition \ref{propcontration} and Theorem \ref{teoexistenceh}. 
\par
The following proposition holds true (for the proof see Appendix \ref{secproofprop}). 

\begin{prop}\label{propderivT}
The total energy $\mathcal{T}$ defined as in \eqref{defTotal} belongs to $AC([0,T])$.
	
	Moreover, for a.e.\ $t \in [0,\frac{\rho_0}{2}]$ the following formulas hold true:
	\begin{align}
	\dot{\mathcal{T}}(t) =& - \pi \dot{\rho}(t)\, \frac{1-\dot{\rho}(t)}{1 + \dot{\rho}(t)} \,  (R-\rho(t)) 
	\left[ \dot{v}_0(\rho(t) {-} t) - v_1(\rho(t) {-} t)  - \int_{0}^{t}  G(\tau , \tau {+} \rho(t) {-}t )  \,d\tau \right]^2 
	\notag
	\\
	&
	+ 2\pi  R \, \dot{w}(t) \left[ \dot{w}(t) - \left(\dot{v}_0(t) + v_1(t)  + \int_{0}^{t} G(\tau , t{-}\tau) \,d \tau \right) \right]  
	\label{formuladerivativeT}
	\\
	=& - \pi \dot{\rho}(t) \, \frac{1-\dot{\rho}(t)}{1 + \dot{\rho}(t)} \, e^{-\alpha t} 
	\left[ \dot{h}_0(\rho(t) {-} t) - h_1(\rho(t) {-} t)  - \int_{0}^{t}  F(\tau , \tau {+} \rho(t) {-}t )  \,d\tau \right]^2
	\notag
	\\
	&
	+ 2\pi  R \, \dot{w}(t) \left[ \dot{w}(t) +\frac{1}{2} \left(\alpha {-} R^{-1}\right) w(t) - R^{-\frac{1}{2}} e^{-\frac{\alpha}{2} t} \left(\dot{h}_0(t) + h_1(t)  + \int_{0}^{t} F(\tau , t{-}\tau) \,d \tau \right) \right] .
	\label{formuladerivativeT-2}  
	\end{align}
\end{prop}

\begin{rem}\label{remT}
 The time derivative of the energy can be computed more directly if more regularity is assumed. Indeed,
by formally applying the Leibniz differentiation rule, integrating by parts, and using the equations \eqref{princeq2} and \eqref{princeqh} satisfied by $v$ and $h$, 
it is possible to see that for a.e.\ $t \in \left[0,\frac{\rho_0}{2}\right]$ 
	\begin{align*}
	\dot{\mathcal{T}}(t) =& \, 
	\pi (R{-}\rho(t)) \, \dot{\rho}(t) \left[v_t^2(t,\rho(t)) + v_r^2 (t,\rho(t))\right] 
	\\
	& + 2\pi (R{-}\rho(t)) \, v_t(t,\rho(t)) \, v_r (t,\rho(t)) - 2\pi R \, v_t(t,0) \, v_r (t,0)
	\\
	=& \, \pi \dot{\rho}(t)\, e^{-\alpha t}  \left(h_t^2(t,\rho(t)) + h_r^2 (t,\rho(t))\right) + 2\pi  e^{-\alpha t} \,h_t(t,\rho(t))\, h_r (t,\rho(t))
	\\
	&- 2\pi  e^{-\alpha t} \left(h_t(t,0) -\frac{\alpha}{2} z(t) \right)  \left(h_r(t,0) +\frac{1}{2} R^{-1} z(t) \right) .
	\end{align*} 
These formulas show that the term $(R{-}r)^{-1}v_r$ in \eqref{princeq2} and \eqref{eqkernelFv} gives no contribution to the energy balance:
indeed, some cancellations occur when integrating \eqref{formuladerivativeT} by parts. This is clear if one compares \eqref{princeq} and \eqref{princeq2}, since no dissipated energy should be associated to a term arising from the change of coordinates. In contrast, the damping term $\alpha v_t$ entails a dissipated energy accounted in $\mathcal{A}$.
\par
In order to justify the differentiation of $\mathcal{T}$,  it is more convenient to use the representation formulas for the solutions which lead to \eqref{formuladerivativeT} and \eqref{formuladerivativeT-2}. 
	Moreover \eqref{formuladerivativeT-2}
has also the advantage that it does not contain partial derivatives of $h$, provided one inserts the initial data $h_0$ and $h_1$.
\end{rem}

\begin{rem} \label{rmk:transl}
 The time derivative of the energy can be computed for a.e.\ $t \in [0,T]$ by formally translating the initial data: precisely, for fixed $t_0 > 0$, 
	\begin{align*}
	\dot{\mathcal{T}}(t) =& - \pi \dot{\rho}(t) \, \frac{1{-}\dot{\rho}(t)}{1 {+} \dot{\rho}(t)} \, (R{-}\rho(t)) 
	\left[ v_r(t_0,\rho(t) {-} t {+}t_0) - v_t(t_0,\rho(t) {-} t{+}t_0)  - \int_{t_0}^{t}  G(\tau , \tau {+} \rho(t) {-}t )  \,d\tau \right]^2 
	\notag
	\\
	&
	+ 2\pi  R \, \dot{w}(t) \left[ \dot{w}(t) - \left(v_r(t_0,t{-}t_0) + v_t(t_0,t{-}t_0)  + \int_{t_0}^{t} G(\tau , t{-}\tau) \,d \tau \right) \right]  
	\label{formuladerivativeTtras}
	\end{align*}
for a.e.\ $t \in \left[t_0, t_0 + \frac{\rho_0}{2} \right] \cap [0,T]$. A similar argument holds for \eqref{formuladerivativeT-2}.
\end{rem}

Since Proposition \ref{propderivT} guarantees the existence of the energy derivative only almost everywhere, we now present an improvement with a formula for the right derivative of $\mathcal{T}$ at a given point. To this end, 
we fix $\overline{t} \in (0,T)$ and consider a function $\overline{w} \in W^{1,2}(0,T)$ and a function $\overline{\rho}: [0,T] \to [\rho_0,R)$ satisfying \eqref{hprho} and such that 
\begin{equation}\label{hpoverrhow}
\overline{w} (t) = w(t) \quad \text{and} \quad \overline{\rho}(t) =\rho(t) \quad \text{for every } t \in [0,\overline{t}].
\end{equation} 
Let $u$ and $\overline{u}$ be the solutions of problem \eqref{princeq} corresponding to $\rho,u_0,u_1,w$ and to $\overline{\rho},u_0,u_1,\overline{w}$, respectively. 
 We regard the energies as functionals depending on $\overline{\rho}$ and on $\overline{w}$, thus  for every $t \in [0,T]$ we define
\begin{equation}\label{defEAoverrhow}
\begin{split}
\mathcal{E}(t;\overline{\rho},\overline{w}) &:= \frac{1}{2} \iint_{R-\rho(t) <|x|< R}(\overline{u}_t^2(t,x) + |\nabla_x \overline{u}(t,x)|^2 ) \,dx,
\\
\mathcal{A}(t;\overline{\rho},\overline{w}) &:= \alpha \int_{0}^{t} \iint_{R-\rho(\tau) <|x|< R} \overline{u}_t^2(\tau,x) \,dx\,d\tau,
\end{split}
\end{equation}
and
\begin{equation} \label{def:T-functional}
\mathcal{T}(t;\overline{\rho},\overline{w}) := \mathcal{E}(t;\overline{\rho},\overline{w}) + \mathcal{A}(t;\overline{\rho},\overline{w}).
\end{equation}
 The following result shows that \eqref{formuladerivativeT-2} can be extended to a given time $\overline{t}$, apart from an exceptional set depending only on the data and independent of $\overline\rho$, $\overline{w}$.  
 We omit the proof, which can be deduced from the one of \cite[Theorem 3.2]{RiNa2020} with minor modifications.

\begin{teo}\label{teoderivT}
Let $\rho$ be as in \eqref{hprho}. Assume \eqref{uinitialcondition} and \eqref{uinitialconditioncomp}. 
Moreover assume that there exist $\beta,\gamma \in \R$ such that 
\begin{equation}\label{condoverrhow}
\lim_{\varepsilon \to 0^+} \frac{1}{\varepsilon} \int_{\overline{t}}^{\overline{t}+ \varepsilon} \left|\dot{\overline{\rho}}(t) - \beta \right| \,dt =0, \quad \lim_{\varepsilon \to 0^+} \frac{1}{\varepsilon} \int_{\overline{t}}^{\overline{t}+ \varepsilon} \left|\dot{\overline{w}}(t) - \gamma \right|^2 \,dt =0.
\end{equation}
Then there exists a set $N \subseteq [0,T]$ of measure zero, depending on $\rho$, $u_0$, $u_1$, and $w$, such that for every $\overline{t} \in [0,T] \setminus N$ the following statement holds true: 
	
	Let $\overline{\rho},\overline{w},\overline{u}$ and $u$ be as above in \eqref{hpoverrhow} and \eqref{defEAoverrhow}. Let $\overline{v}$ and $\overline{h}$ be defined by \eqref{defv} and \eqref{defh}, respectively, for the solution $\overline{u}$. 
Then
	\begin{equation*}
	\dot{\mathcal{T}} (\overline{t}^+;\overline{\rho},\overline{w}) :=  \lim_{\varepsilon \to 0^+}  \frac{\mathcal{T}(\overline{t}+\varepsilon;\overline{\rho},\overline{w}) - \mathcal{T}(\overline{t};\overline{\rho},\overline{w})}{\varepsilon} 
	\end{equation*}
	exists. Moreover, if $\overline{t} \in [0,\frac{\rho_0}{2}] \setminus N$, one has the following explicit formula: 
	\begin{equation}\label{formuladerivativeTr}
	\dot{\mathcal{T}} (\overline{t}^+;\overline{\rho},\overline{w}) = - \pi \beta \,\frac{1{-}\beta}{1 {+} \beta}\,  e^{-\alpha \overline{t}} 
	\left[ \dot{h}_0(\rho(\overline{t}) {-} \overline{t}) - h_1(\rho(\overline{t}) {-} \overline{t})  - \int_{0}^{\overline{t}}  F(\tau , \tau {+} \rho(\overline{t}) {-} \overline{t} )  \,d\tau \right]^2
	+ \gamma\, \mathcal{Q}(\overline{t},\gamma) ,
	\end{equation}
	where $F$ is defined as in \eqref{eqkernelFh}  and 
	\begin{equation}\label{formulapotenza}
	 \mathcal{Q}(t,\gamma):=2\pi R \left[ \gamma +\frac{1}{2} \left(\alpha  {-} R^{-1}\right) w(\overline{t}) - R^{-\frac{1}{2}} e^{-\frac{\alpha}{2} \overline{t}} \left(\dot{h}_0(\overline{t}) + h_1(\overline{t})  + \int_{0}^{\overline{t}} F(\tau , \overline{t}{-}\tau) \,d \tau \right) \right] .
	\end{equation}
\end{teo}

It follows that $\dot w(t)\,\mathcal{Q}(t,\dot w(t))$ is  the power of external forces, see also \eqref{formulaexternallaodW}. The first term in \eqref{formuladerivativeTr} is related to the energy release rate as we outline below.

\subsection{Dynamic energy release rate} \label{sec:rmk-ERR}
The energy release rate \cite{Fre90} quantifies the energy gained by an infinitesimal debonding (or crack) growth, thus it is usually defined as the opposite of the derivative of the energy with respect to debonding elongation. 
In our model, the energy release rate is defined by taking into account both the potential and the kinetic energy.
In order to rigorously define this notion in our setting, we follow \cite{DMLazNar16} and \cite{RiNa2020} and use the \emph{time} derivative of the energy studied above,
 suitably scaled by a factor corresponding to the derivative of the debonded surface.

In this section we provide a definition of dynamic energy release rate that can be applied to a general setting without a priori assumptions on the solutions.
 For simplicity we state the definition in dimension two, however it is straightforward to extend our considerations to general dimension $n>2$. 
Let
	\begin{equation*}
	E_{\rho(t)} := \{ x \in \R^2 \,|\, R-\rho(t) \leq f(x) \leq R \} 
	\end{equation*}
	be a growing domain depending on a prescribed function $f$ (in our case $f(x) = |x|$).
Assume
\begin{equation} \label{hyp-ERR-general}
f \in C^1(\R^2), \quad \nabla f \neq 0 \ \text{in}\ E_{\rho(T)} .
\end{equation}
Let $\mathcal{T}$ be defined as in \eqref{defTotal} using the set $E_{\rho(t)}$ as domain of integration and recall the notation introduced in \eqref{defEAoverrhow}--\eqref{def:T-functional}.
We first consider the case where the debonding speed $\beta$ is strictly positive, where such speed is defined as a right derivative as in \eqref{condoverrhow}, i.e., $\dot\rho(\overline{t}^+)=\beta\in(0,1)$.

\begin{defin}\label{defenergyreleaserate} 
 Let $\rho$ be as in \eqref{hprho}. Assume \eqref{uinitialcondition} and \eqref{uinitialconditioncomp}. 
For a.e.\ $\overline{t} \in [0,T]$ and for every $\beta \in (0,1)$, the dynamic energy release rate corresponding to debonding speed $\beta$ is defined as
	\begin{equation}\label{primaformulaGgeneral}
	\mathcal{G}_\beta(\overline{t}) := \lim_{t \to \overline{t}^+} - \frac{\mathcal{T}(t;\overline{\rho},\overline{w}) - \mathcal{T}(\overline{t};\overline{\rho},\overline{w})}{|E_{\rho(t)} \setminus E_{\rho(\overline{t})}|} ,
	\end{equation}
	where $\overline{\rho}$ is an arbitrary Lipschitz extension of $\rho_{|[0,\overline{t}]}$ satisfying \eqref{hpoverrhow} and \eqref{condoverrhow}, while 
	\begin{equation*}
	\overline{w}(t) :=
	\begin{cases}
	w(t) \quad & \text{if } t \in [0,\overline{t}],
	\\
	w(\overline{t}) \quad & \text{if } t \in (\overline{t},T].
	\end{cases}
	\end{equation*}
\end{defin}

Whenever the right derivative $\dot{\mathcal{T}}(\overline{t}^+;\overline{\rho},\overline{w})$ exists (cf.\ Theorem \ref{teoderivT}), the following proposition guarantees that the dynamic energy release rate in \eqref{primaformulaGgeneral} is well defined. 

\begin{prop}\label{prop:ERR-general}
Assume that the right derivative $\dot{\mathcal{T}}(\overline{t}^+;\overline{\rho},\overline{w})$ exists.
If \eqref{hyp-ERR-general} holds, then the limit in \eqref{primaformulaGgeneral} exists, it only depends on $\overline{t}$ and $\beta\in(0,1)$, and 
\[
\mathcal{G}_\beta(\overline{t}) = - \frac{\dot{\mathcal{T}}(\overline{t}^+;\overline{\rho},\overline{w})}{\frac{d}{dt} |E_{\rho(t)}| \big|_{t=\overline{t}^+}} 
= - \frac1\beta \, \frac{\dot{\mathcal{T}}(\overline{t}^+;\overline{\rho},\overline{w})}{\frac{d}{d\rho} |E_{\rho}| \big|_{\rho=\rho(\overline{t})} }.
\]
\end{prop}

\begin{proof}
We need to show that the right derivative of $t\mapsto |E_{\rho(t)}| $ exists at $t=\overline{t}^+$. Indeed, 
	the coarea formula 
	\begin{equation*}
	\iint_{\R^2} g(x) |\nabla f(x)| \,dx = \int_{\R} \left( \int_{\{f=s\}} g(y) \, d\mathcal{H}^1(y) \right) \,ds
	\end{equation*}
	for the function $g(y)= \frac{\chi_{E_{\rho}}(y)}{|\nabla f(y)|}$ implies that 
	\begin{equation*}
	|E_{\rho}| = \int_{R-\rho}^{R} \left( \int_{\{f=s\}} \frac{1}{|\nabla f(y)|} \, d\mathcal{H}^1(y)  \right) \, ds.
	\end{equation*}
	Hence, recalling \eqref{hyp-ERR-general} we obtain the key formula
	\begin{equation*}
	\frac{d}{d\rho} |E_{\rho}| \bigg|_{\rho=\rho(\overline{t})} = \int_{\{f= R-\rho(\overline{t})\} } \frac{1}{|\nabla f(y)|} \, d\mathcal{H}^1(y).
	\end{equation*}
Clearly we have
	\begin{equation*}
	\frac{d}{dt} |E_{\rho(t)}| \bigg|_{t=\overline{t}^+ } = \dot{\rho}(\overline{t}^+)  \, \frac{d}{d\rho} |E_{\rho}| \bigg|_{\rho=\rho(\overline{t})} ,
	\end{equation*}
which only depends on $\overline{t}$ and on the right derivative $\dot{\rho}(\overline{t}^+)=\beta>0$. 
\end{proof}

We remark that our definition of energy release rate applies to the case already  studied in \cite{DouMarCha08,LBDM12,DMLazNar16,RiNa2020},
that is the case of a strip $(0,R)\times(0,1)$ with debonding front $x_1=R{-}\rho(t)$ and debonded region $(R{-}\rho(t),R)\times(0,1)$. 
In this setting one may take $f(x)=x_1$. 

In the radial case studied in this paper, 
one has $f(x)= |x|$, $|E_{\rho}|= \pi (R^2 - (R{-}\rho)^2)$, and
	\begin{equation*}
	\frac{d}{d\rho} |E_{\rho}| = 2\pi (R-\rho), \quad \quad \frac{d}{dt} |E_{\rho(t)}| \bigg|_{t= \overline{t}} = 2\pi (R-\rho(\overline{t})) \dot{\rho} (\overline{t}).
	\end{equation*}
Hence, \eqref{primaformulaGgeneral} reduces to 
\begin{equation*}\label{formulaG_beta}
	\mathcal{G}_\beta (\overline{t})= -\frac{1}{2\pi \beta (R{-}\rho(t))} \, \dot{\mathcal{T}} (\overline{t}^+;\overline{\rho},\overline{w}) .
\end{equation*}
This definition has to be completed for $\beta=0$. We do this exploiting the ansatz of radial solutions.

In general, one should verify the following property, which may depend on the representation of the solutions to the wave equation and is instrumental for deriving Griffith's criterion below: for a.e.\ $\overline{t}$ the function
\begin{equation}\label{property:G0}
(0,1)\ni\beta\mapsto\mathcal{G}_\beta(\overline{t}) \ \text{is bounded, continuous, strictly decreasing, and}\ \lim_{\beta\to1^-}\mathcal{G}_\beta(\overline{t})=0 .
\end{equation}
Then one defines
\begin{equation} \label{def:G0}
\mathcal{G}_0(\overline{t}):=\lim_{\beta\to0^+}\mathcal{G}_\beta(\overline{t})>\mathcal{G}_\beta(\overline{t}) \ \text{for every}\ \beta\in(0,1) .
\end{equation}
In the case of radial solutions, \eqref{property:G0} is a consequence of \eqref{formulah}, as shown in the following remark.

\begin{rem}\label{remG}
	By Theorem \ref{teoderivT}, we know that for a.e.\ $\overline{t} \in \left[0,\frac{\rho_0}{2}\right]$ 
	\begin{equation}\label{formulaG_beta*}
	\mathcal{G}_\beta(\overline{t}) = \frac{1}{2(R{-}\rho(\overline{t}))} \, \frac{1{-}\beta}{1 {+} \beta}\,  e^{-\alpha \overline{t}} 
	\left[ \dot{h}_0(\rho(\overline{t}) {-} \overline{t}) - h_1(\rho(\overline{t}) {-} \overline{t})  - \int_{0}^{\overline{t}}  F(\tau , \tau {+} \rho(\overline{t}) {-} \overline{t} )  \,d\tau \right]^2 .
	\end{equation}
	For $\beta=0$, arguing by continuity, for a.e.\ $\overline{t} \in \left[0,\frac{\rho_0}{2}\right]$ we have
	\begin{equation}\label{formulaG_0}
	\mathcal{G}_0(\overline{t}) = \frac{1}{2(R{-}\rho(\overline{t}))}  \,  e^{-\alpha \overline{t}} 
	\left[ \dot{h}_0(\rho(\overline{t}) {-} \overline{t}) - h_1(\rho(\overline{t}) {-} \overline{t})  - \int_{0}^{\overline{t}}  F(\tau , \tau {+} \rho(\overline{t}) {-} \overline{t} )  \,d\tau \right]^2 .
	\end{equation}
	We may extend this also for $\overline{t}>\frac{\rho_0}{2}$ as hinted in Remark \ref{rmk:transl}, obtaining \eqref{formulaG_beta**} for $\overline{t} \in [0,T] $;
	however in the rest of the paper we only need such formulas for small $\overline{t}$.
	In particular, it turns out that 
	\begin{equation}\label{formulaG_beta**}
	\mathcal{G}_\beta(\overline{t}) = \frac{1{-}\beta}{1 {+} \beta} \, \mathcal{G}_0(\overline{t})
	\end{equation}
	for a.e.\ $\overline{t} \in \left[0,T\right]$.
\end{rem}

\subsection{Griffith's criterion}  
The flow rule for the debonding evolution involves a material coefficient depending on the point in the reference configuration, given by
a  bounded measurable radial function $\overline\kappa: B(0,R) \to (0,+\infty)$,
quantifying the toughness of the glue between the film and the substrate.
The amount of energy dissipated during the debonding process in the time interval $(0,t)$, for $t \in [0,T]$, is  
\begin{equation}\label{formulaenergydissk}
\iint_{R-\rho(t)< |x| < R-\rho_0} \overline\kappa(|x|) \,dx = 2\pi \int_{\rho_0}^{\rho(t)} (R{-}r) \, \kappa(r) \,dr = 2\pi \int_{0}^{t} (R{-}\rho(s)) \, \kappa(\rho(s)) \, \dot{\rho}(s) \,ds ,
\end{equation}
 where $\kappa:  [\rho_0,R] \to (0,+\infty)$ is defined by $\kappa(r)=\overline\kappa(x)$ for $|x|=R-r$, cf.\ \eqref{defv}. 

We are now in a position to state the principle governing the evolution of the debonding front. 
Henceforth, the function $\rho$ is unknown and is a solution of the following system, that is a formulation of the classical criterion of Griffith featuring the dynamic energy release rate: 
\begin{equation}\label{Griffithcriterion}
\mbox{for a.e. } t \in [0,T]
\qquad 
\begin{dcases}
0\leq \dot{\rho}(t) < 1,
\\
\mathcal{G}_{\dot{\rho}(t)} (t) \leq \kappa(\rho(t)),
\\
\left[ \mathcal{G}_{\dot{\rho}(t)} (t) - \kappa(\rho(t)) \right] \dot{\rho}(t) =0 .
\end{dcases} 
\end{equation} 
Griffith's criterion requires that: the debonded region is set nondecreasing in time; the energy release rate is less than or equal to the toughness at the debonding front;
the debonded region may actually increase only if the energy release rate is critical.

Moreover, it is possible to see that Griffith's criterion \eqref{Griffithcriterion} is equivalent to postulating an energy-dissipation balance \textbf{(EDP)} and a maximum dissipation principle \textbf{(MDP)}: 
\begin{align*}
&\textbf{(EDP)} &&\mathcal{T} (t) + \iint_{R-\rho(t)< |x| < R-\rho_0} \overline\kappa(|x|) \,dx = \mathcal{T} (0) + \mathcal{W} (t) \quad \mbox{ for every } t \in[0,T], 
\\
&\textbf{(MDP)} &&\dot{\rho}(t)  = \max \{ \beta \in [0,1) \,|\, \kappa(\rho(t)) \beta = \mathcal{G}_\beta(t) \beta \} \quad \mbox{ for every } t \in[0,T], 
\end{align*}
where $\mathcal{W}$ is the work of the external loading.  For $t \in \left[0,\frac{\rho_0}{2}\right]$ it holds
\begin{equation}\label{formulaexternallaodW}
\mathcal{W}(t) = \int_0^t \dot w(s)\,\mathcal{Q}(s,\dot w(s)) \, ds ,
\end{equation}
where $\mathcal{Q}$ is defined in \eqref{formuladerivativeTr} and \eqref{formulapotenza}. 
We refer to \cite[Section 2.2]{DMLazNar16} for a proof of the equivalence between \textbf{(EDP)}--\textbf{(MDP)} and \eqref{Griffithcriterion}.
Here we only observe that one needs \eqref{property:G0} in order to show that \textbf{(MDP)} implies $\mathcal{G}_{\dot{\rho}(t)} (t) \leq \kappa(\rho(t))$ in \eqref{Griffithcriterion}. Indeed, this is trivial if $\dot\rho(t)>0$;
on the other hand, assuming $\kappa(\rho(t))<\mathcal{G}_{0} (t)$ would imply $\kappa(\rho(t))=\mathcal{G}_{\beta} (t)$ for some $\beta\in(0,1)$, which gives a contradiction if $\dot\rho(t)=0$.

Using \eqref{formulaG_beta**} it is easy to see that \eqref{Griffithcriterion} is equivalent to  the following ordinary differential equation:
	\begin{equation}\label{oderho}
	\dot{\rho}(t) = \max \left\{ 0, \ \frac{\mathcal{G}_0(t) - \kappa(\rho(t))}{\mathcal{G}_0(t) + \kappa(\rho(t))} \right\}  \quad \mbox{ for a.e. } t \in [0,T] .
	\end{equation}
 It will be convenient to resort again to the function $h$ defined in \eqref{defh}. By \eqref{formulaG_0} the  differential equation, 
for a.e.\ $t \in [0, \frac{\rho_0}{2}]$, can be rewritten as 
	\begin{equation}\label{oderho*}
	\dot{\rho}(t) = \max \left\{ 0, \ \frac{\left[ \dot{h}_0(\rho(t) {-} t) - h_1(\rho(t) {-} t)  - \int_{0}^{t}  F(\tau , \tau {+} \rho(t) {-} t )  \,d\tau \right]^2 - 2(R{-}\rho(t)) \,e^{\alpha t}\, \kappa(\rho(t))}{\left[ \dot{h}_0(\rho(t) {-} t) - h_1(\rho(t) {-} t)  - \int_{0}^{t}  F(\tau , \tau {+} \rho(t) {-} t )  \,d\tau \right]^2 + 2(R{-}\rho(t))\, e^{\alpha t} \,\kappa(\rho(t))} \right\} .
	\end{equation}
 Notice that the latter equation is strongly coupled with the wave equation in the time-dependent interval $(0,\rho(t))$ due to the term $F$ defined in \eqref{eqkernelFh}. In the following section we shall see how to solve such differential equation and deal with the coupled problem:  this will show that the definition of dynamic energy release rate (Definition \ref{defenergyreleaserate}) and the formulation of Griffith's criterion \eqref{Griffithcriterion} are well posed in the case of radial solutions.

\section{Evolution of the debonding front} \label{sec:coupled}

In this section the function $\rho$ is unknown and it is found as a solution to Griffith's criterion \eqref{Griffithcriterion}. We prove the existence of a pair $(u,\rho)$ which solves the coupled problem \eqref{princeq}--\eqref{Griffithcriterion} (see Definition \ref{def:sol-coupled} below). 
In general, uniqueness only holds among \emph{radial} functions; however, if the data of the problem are sufficiently regular, it turns out that the displacement $u$ is the only solution of the wave equation in the growing domain $\mathcal{C}_{R-\rho(t),R}$: see Remark \ref{rmk:radial-disp} for details.
	
	Recalling the definition of $\kappa$ in \eqref{formulaenergydissk}, we will assume that 
	\begin{equation}  
	\label{assumptionkappa} \mbox{there exist } c_1, c_2 >0 \mbox{ such that } 0< c_1 \leq \kappa (r) \leq c_2 \quad \forall x \in \R^2, \,\rho_0 \leq r < R,
	\end{equation}
	\begin{equation}  
    \label{epsiloncondition}
	\mbox{for every } r \in[\rho_0,R) \mbox{ there exists } \varepsilon=\varepsilon(r)>0 \mbox{ such that } \kappa \in C^{0,1}([r,r+\varepsilon]).
	\end{equation}
	Throughout this section we will assume that the data of the problem $u_0,u_1, w$, originally introduced in \eqref{uinitialcondition}, satisfy the following assumptions,
	\begin{equation}\label{uinitialcondition*}
	w \in C^{0,1}(0,+\infty), \quad u_0 \in C^{0,1}_{\text{rad}}(\mathcal{C}_{R-\rho_0,R}), \quad u_1 \in L^\infty_{\text{rad}}(\mathcal{C}_{R-\rho_0,R}),
	\end{equation}
where the subscript $\textnormal{rad}$ means that the initial data are radial, as above. 
	Then we give the following definition of \emph{radial} solution to the coupled problem \eqref{princeq}--\eqref{Griffithcriterion}.
	
	\begin{defin} \label{def:sol-coupled}
		Let $\rho:[0,T] \to [\rho_0,R)$  be as in \eqref{hprho}.  Let $u: [0,T] \times B(0,R) \to \R$ be a measurable function of class $W^{1,2}$. Define $\mathcal{G}_\beta$ as in \eqref{primaformulaGgeneral} and \eqref{def:G0}. 
		We say that the pair $(u,\rho)$ is a radial solution to the coupled problem \eqref{princeq}--\eqref{Griffithcriterion} if:
		\begin{enumerate}
			\item $u(t,\cdot) \in W^{1,2}_{\mathrm{rad}}(\mathcal{C}_{R-\rho(t),R})$ for every $t \in [0,T]$; 	
			\item u solves problem \eqref{princeq} in $ \mathcal{O}_\rho$,  where $\mathcal{O}_\rho$ is as in \eqref{def:domain-time};
			\item $u\equiv 0$ outside $\overline{\mathcal{O}}_\rho$;
			\item $(u,\rho)$ satisfies Griffith's criterion \eqref{Griffithcriterion} for a.e. $t \in [0,T]$. 
		\end{enumerate}
	\end{defin}

 In the existence proof we resort to the variable $h$ defined in \eqref{defv} and \eqref{defh}.  Hence we reduce to a one-dimensional problem, \eqref{princeqh}--\eqref{oderho},
which is solved using the results of \cite[Section 4]{RiNa2020}. In this section we only mention the parts of the proof which differ from \cite{RiNa2020}. 
    As already done before, we first consider a small time interval $[0,T]$ and assume $T \in \left(0,\frac{\rho_0}{2}\right)$. 
Under this assumption, Griffith's criterion is equivalent to the ordinary differential equation  \eqref{oderho*},  with initial condition $\rho(0)=\rho_0$.
 In order to study such equation, we  introduce an auxiliary function $\lambda$, defined as the inverse of the map $t \mapsto t -\rho(t)$. In particular we have   $\rho(t) = t - \lambda^{-1}(t)$. 
    \begin{rem}\label{rmk:notationsec3}
	From now on, if not otherwise specified, $\rho$ and $\lambda$ will be tacitly related as just  described.  Moreover we will stress the dependence on $\lambda$ (and hence on $\rho$)  by writing  $\Omega^\lambda$, $P^\lambda$, $\h^\lambda$, $\Phi^\lambda[F]$, instead of $\Omega$, $P$, $\h$, and $\Phi[F]$, respectively; cf. \eqref{def:Omega}, \eqref{R(t,r)}, \eqref{eq:undamp}, and \eqref{defmathcalF}. Moreover, for the dynamic energy release rate we will use the notation $\mathcal{G}_0^{h,\lambda}$ instead of $\mathcal{G}_0$; see \eqref{formulaG_beta*} and \eqref{formulaG_0} for the dependence on the solution $h$ given by \eqref{defv} and \eqref{defh}.  
	\end{rem}
    Since $0 \leq \dot{\rho} < 1$ a.e., the function $\lambda$ is absolutely continuous. By \eqref{oderho} and by deriving the relation $\lambda(t{-}\rho(t)) = t$ we get 
    \begin{equation*}
    \begin{dcases}
    \dot{\lambda}(t{-}\rho(t)) = 1 &\quad
    \mbox{if }  \mathcal{G}^{h,\lambda}_0(t) \le \kappa(\rho(t)) , 
    \\
    \dot{\lambda}(t{-}\rho(t)) = \frac{\mathcal{G}^{h,\lambda}_0(t)  +\kappa(\rho(t))}{2\kappa(\rho(t))} &\quad \mbox{if }  \mathcal{G}^{h,\lambda}_0(t) > \kappa(\rho(t)) . 
    \end{dcases}
    \end{equation*}
    Hence,  recalling also \eqref{eqinitialdatav}, \eqref{eqinitialdatah}, and \eqref{formulaG_0},  we obtain that $\lambda$ satisfies the following differential equation:
    \begin{equation}\label{odelambda}
    \begin{dcases}
    \dot{\lambda} (s) = \frac{1}{2} \left( 1+ \max \{\Lambda^{h,\lambda}(s),1 \} \right) \quad & s \in [-\rho_0,\lambda^{-1}(T)],
    \\
    \lambda(-\rho_0) = 0 ,
    \end{dcases}
    \end{equation} 
    where for a.e.\ $s \in [-\rho_0, \lambda^{-1}(T)]$
    \begin{equation}\label{lambdaequation}
    \Lambda^{h,\lambda}(s) := \frac{\left(\dot{h}_0(-s) - h_1(-s) - \int_{0}^{\lambda(s)} F(\tau,\tau{-}s) \, d\tau\right)^2}{2(R{-}\lambda(s) {+} s) \, e^{\alpha\lambda(s)} \, \kappa(\lambda(s) {-} s)}.
    \end{equation}
    We observe that, by assumption \eqref{uinitialcondition*}, $\Lambda^{h,\lambda}$ is $L^{\infty}(-\rho_0, \lambda^{-1}(T))$ which implies, by \eqref{odelambda}, that $\lambda$ is Lipschitz. 
	
	Now, we recast the coupled problem \eqref{princeq}--\eqref{oderho} into a fixed point problem.  Recalling that we have assumed $T \in \left(0,\frac{\rho_0}{2}\right)$, by \eqref{A(t,r)}, \eqref{formulah}, \eqref{defmathcalF}, and \eqref{eqkernelFh}) the coupled problem is equivalent to 
	\begin{equation}\label{coupledproblemhlambda}
	\begin{dcases}
	h(t,r) = \left(\h^{ \lambda}(t,r) + \frac{1}{2} \Phi^{ \lambda}[F](t,r) \right) \chi_{\Omega^{ \lambda}}(t,r) & \mbox{ for a.e. } (t,r) \in (0,T) \times (0,R),
	\\
	\lambda (s) = \frac{1}{2} \int_{-\rho_0}^{s} \left( 1+ \max \{\Lambda^{h,\lambda}(\sigma),1 \} \right) \,d\sigma & \mbox{ for every } s \in [-\rho_0, \lambda^{-1}(T)].
	\end{dcases}
	\end{equation}
 For $y \in (0,\rho_0)$ 
\begin{figure}[b]
 \centering
{\tiny
 \psfrag{r}{$r$}
 \psfrag{t}{$t$}
 \psfrag{T}{$T$}
 \psfrag{L}{$\rho^\lambda$}
 \psfrag{0}{\hspace{.5em}$\rho_0$}
 \psfrag{9}{\hspace{-.7em}$\rho_0{-y}$}
 \includegraphics[width=.6\textwidth]{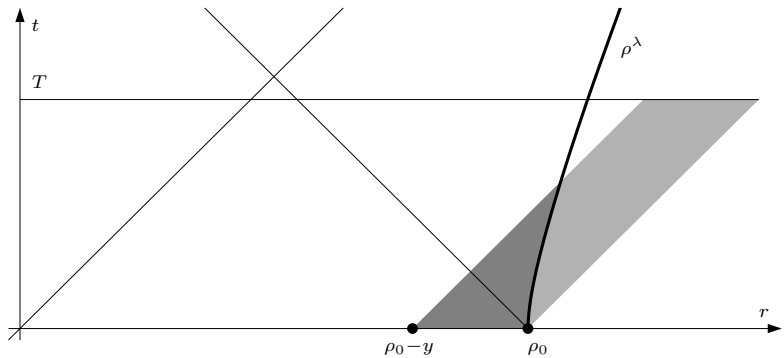}
}
 \caption{Sets defined in formula \eqref{I^yQ^T}: 
$Q^{\lambda,T,y}$ is the region in dark gray, while $Q^{T,y}$ is the union of the regions in dark and light gray.}
 \label{fig:corner}
 \end{figure}
we define the following sets (cf.\ Figure \ref{fig:corner}):
	\begin{equation}\label{I^yQ^T}
	\begin{split}
	I^y &:= [-\rho_0, -\rho_0+y],
	\\
	Q^{T,y}&:= \{ (t,r) \,|\, 0\leq t\leq T ,\, \rho_0 -y+t \leq r \leq \rho_0+t\},
	\\
	Q^{\lambda, T,y}&:= Q^{T,y} \cap \overline{\Omega}^{\lambda}.
	\end{split}
	\end{equation} 
	Moreover, for $M>0$ we introduce the following spaces:
	\begin{equation*}
	\begin{split}
	\mathcal{X}_1(M,T,y) &:= \{ h \in C^0(Q^{T,y}) \,|\, \|h\|_{C^0(Q^{T,y})}\leq M \},
	\\
	\mathcal{X}_2(T,y) &:= \{ \lambda \in C^0(I^y) \,|\, \lambda(-\rho_0)=0,\, \|\lambda\|_{C^0(I^y)} \leq T, \, s \mapsto \lambda(s){-}s \mbox{ is nondecreasing}  \},
	\\
	\mathcal{X}(M,T,y) &:= \mathcal{X}_1(M,T,y) \times \mathcal{X}_2(T,y).
	\end{split}
	\end{equation*} 
	Then, for $(h,\lambda) \in \mathcal{X}(M,T,y)$, we consider the following two operators associated with the coupled problem \eqref{coupledproblemhlambda}, and defined by
	\begin{align}
	\Psi_1[h,\lambda] (t,r) &:= \left(\h^{\lambda}(t,r) + \frac{1}{2} \Phi^{\lambda}[F](t,r) \right) \chi_{\Omega^{\lambda}}(t,r) && \text{for}\ (t,r) \in Q^{T,y},  \label{psi1equation}
	\\ 
	\Psi_2[h,\lambda] (s) &:= \frac{1}{2} \int_{-\rho_0}^{s} \left( 1+ \max \{\Lambda^{h,\lambda}(\sigma),1 \} \right) \,d\sigma && \text{for}\ s \in I^{y}\label{psi2equation}.
	\end{align}
	Finally we set 
	\begin{equation}\label{psiequation}
	\Psi[h,\lambda] := (\Psi_1[h,\lambda],\Psi_2[h,\lambda]). 
	\end{equation} 
	Then we have the following result.
	
	\begin{lemma}\label{lemmay}  Assume $T \in \left(0,\frac{\rho_0}{2}\right)$. 
		 For every $M > 0$, there exists $\tilde{y} \in (0,\rho_0)$ such that,  for every $y\in(0,\tilde{y})$, 
 the operator $\Psi$ defined by \eqref{psi1equation}--\eqref{psiequation} maps $\mathcal{X}(M,T,{y})$ into itself.
	\end{lemma}

	\begin{proof}
		Fix $M>0$ and let $ y \in (0,\rho_0)$. Consider $(h,\lambda) \in \mathcal{X}(M,T, y)$. By Lemma \ref{lemmaAandPhi}, we know that $\Psi_1[h,\lambda]$ is continuous on $Q^{T, y}$, while, by construction, $\Psi_2[h,\lambda]$ is absolutely continuous on $ I^{y}$, being $\Lambda^{h,\lambda} \in L^{\infty}(-\rho_0,\lambda^{-1}(T))$. Moreover, 
		\begin{equation*}
		\Psi_2[h,\lambda] (-\rho_0)=0 \quad\mbox{ and }\quad \frac{d}{dy} \Psi_2[h,\lambda] (s) \geq 1 \quad \mbox{ for a.e. } s \in  I^{y}.
		\end{equation*}
		Hence, it is enough to find $ y\in (0,\rho_0)$ such that 
		\begin{equation*}
			\|\Psi_1[h,\lambda]\|_{ C^0(Q^{T,y})} \leq M \quad\mbox{ and }\quad \Psi_2[h,\lambda] (-\rho_0 + y) \leq T.
		\end{equation*}
		So let us consider $(t,r) \in  Q^{\lambda,T,y}$ and estimate $\Psi_1[h,\lambda] (t,r)$ as follows:
		\begin{equation}\label{psi1bound} 
		\begin{split}
		|\Psi_1[h,\lambda] (t,r)| & \leq  |\h^{\lambda}(t,r)| + \frac{1}{2} \, |\Phi^{\lambda}[F](t,r)|
		\\
		&\leq |\h^{\lambda}(t,r)| + \frac{1}{8} \left| \iint_{P^{\lambda}(t,r)} \left( \alpha^2 + \frac{1}{(R{-}\sigma)^2} \right) h(\tau,\sigma)  \,d\tau d\sigma \right|
		\\
		&\leq \int_{\rho_0 - { y}}^{\rho_0} (|\dot{h}_0(s)| + |h_1(s)|) \,ds + \frac{1}{8} \left( \alpha^2 + \frac{1}{(R{-}\rho^\lambda(T))^2} \right)  M T y\, ,
		\end{split}
		\end{equation}
		where we have used \eqref{A(t,r)} (observe that $ Q^{\lambda,T,y} \subset \Omega_1' \cup \Omega_3'$) and the fact that $  |Q^{\lambda,T,y}| \leq T y$. For $\Psi_2[h,\lambda]  (-\rho_0+y)$ we proceed as follows:
		\begin{equation*}
		\begin{split}
		\Psi_2[h,\lambda] (-\rho_0+y) & =  \frac{1}{2} \int_{-\rho_0}^{-\rho_0+y} \left( 1+ \max \{\Lambda^{h,\lambda}(\sigma),1 \} \right) \,d\sigma
		\\
		&\leq \frac{1}{2} \int_{-\rho_0}^{-\rho_0+y} \left( 2 + \Lambda^{h,\lambda}(\sigma) \right) \,d\sigma \leq { y} + \frac{1}{2} \int_{-\rho_0}^{-\rho_0+y} \Lambda^{h,\lambda}(\sigma) \,d\sigma
		\\
		& \leq { y} + \frac{1}{2}  \int_{-\rho_0}^{-\rho_0+y} \frac{\left(\dot{h}_0(-\sigma) - h_1(-\sigma) - \int_{0}^{\lambda(\sigma)} F(\tau,\tau{-}\sigma) \, d\tau\right)^2}{2(R{-}\lambda(\sigma) {+} \sigma) \, e^{\alpha\lambda(\sigma)} \,\kappa(\lambda(\sigma) {-} \sigma)} \, d\sigma
		\\
		& \leq { y} + \frac{1}{2}  \int_{-\rho_0}^{-\rho_0+y} \frac{\left(\dot{h}_0(-\sigma) - h_1(-\sigma) \right)^2 + \left(\int_{0}^{\lambda(\sigma)} F(\tau,\tau{-}\sigma) \, d\tau\right)^2}{c_1(R{-}c_2{-}\rho_0)} \, d\sigma.
		\end{split}
		\end{equation*} 
		Then plugging the definition of $F$, cf. \eqref{eqkernelFh}, and recalling that $ |Q^{\lambda,T,y}| \leq T y$, we obtain
		\begin{equation}\label{psi2bound}
		\begin{split}
		0\le \Psi_2[h,\lambda] (-\rho_0+y)
		&\leq { y} + \frac{1}{2 c_1(R{-}c_2{-}\rho_0)} \int_{-\rho_0}^{-\rho_0+y}  \left(\dot{h}_0(-\sigma) - h_1(-\sigma) \right)^2 \, d\sigma
		\\
		& \quad +  \frac{1}{32 c_1(R{-}c_2{-}\rho_0)} M^2 T^2 \left( \alpha^2 + \frac{1}{(R{-}  c_2 {-} \rho_0)^2} \right)^2 { y}\,.
		\end{split}
		\end{equation}
		Then we conclude since both the expressions on the right-hand side of \eqref{psi1bound} and \eqref{psi2bound} go to $0$ as $ y \to 0^+$.
	\end{proof}
	
	\begin{lemma}\label{lemmaequi} Assume $T \in \left(0,\frac{\rho_0}{2}\right)$. 
Fix $M>0$ and let $\tilde{y} \in (0,\rho_0)$ be given by Lemma \ref{lemmay}. Then,  for every $y\in(0,\tilde{y})$, $\Psi_1(\mathcal{X}(M,T,{y}))$ is an equicontinuous family of $\mathcal{X}_1(M,T,{y})$.
	\end{lemma}

	\begin{proof}
		 Let $ y\in(0,\tilde{y})$ and $(h,\lambda) \in \mathcal{X}(M,T, y)$ and fix $\varepsilon > 0$. First we observe that 
\[ |P^{\lambda}(t_1,r_1) \triangle P^{\lambda}(t_2,r_2)| \leq C_1 \,\textnormal{Per}(P^{\lambda}(t_1,r_1)) \,|(t_1,r_1)-(t_2,r_2)| \leq  C_2 (T + y) |(t_1,r_1)-(t_2,r_2)|\,, \]
for every $(t_1,r_1), (t_2,r_2) \in Q^{\lambda,T,y}$, where $\textnormal{Per}(P^{\lambda}(t_1,r_1))$ is the perimeter of the set $P^{\lambda}(t_1,r_1)$ and $C_1$ and $C_2$ are two positive constants. 
Moreover, by \eqref{uinitialcondition*} there exists $\delta_1 >0$ such that for every $a,b \in [0,\rho_0]$ satisfying $|a-b| \leq \delta_1$ it holds
			\begin{equation*}
			|h_0(a)-h_0(b)| + \left| \int_{a}^{b} h_1(s)  \,ds \right| \leq \frac{\varepsilon}{2}\,.
			\end{equation*}
		Now define 
		\begin{equation*}
		 \delta := \min \left\{ \frac{\delta_1}{2} , \frac{4 \varepsilon}{M C_2 (T + y) \left( \alpha^2 + \frac{1}{(R - \rho^\lambda(T))^2} \right)} \right\}
		\end{equation*}
		and take $(t_1,r_1), (t_2,r_2) \in  Q^{\lambda,T,  y}$ such that $|(t_1,r_1)-(t_2,r_2)| \leq \delta$. Then,
		\begin{equation*}
		|\Psi_1[h,\lambda](t_1,r_1) - \Psi_1[h,\lambda](t_2,r_2) | \leq J_1 + J_2,
		\end{equation*}
		where
		\begin{align*}
		& J_1 :=  |\h^{\lambda}(t_1,r_1)\chi_{\Omega^{\lambda}}(t_1,r_1) - \h^{\lambda}(t_2,r_2) \chi_{\Omega^{\lambda}}(t_2,r_2) |,
		\\
		& J_2 :=  \frac{1}{2} \left|\Phi^{\lambda}[F](t_1,r_1)  \chi_{\Omega^{\lambda}}(t_1,r_1) - \Phi^{\lambda}[F](t_2,r_2)  \chi_{\Omega^{\lambda}}(t_2,r_2) \right|.
		\end{align*}
		 We now estimate $J_2$. Since $\h^{ \lambda}\chi_{\Omega^{ \lambda}}$ and $\Phi^{ \lambda}[F] \chi_{\Omega^{ \lambda}} $ are continuous on $ Q^{\lambda,T,y}$ and vanish on $ Q^{T,y} \setminus \Omega^{ \lambda}$, it is enough to consider the case in which both $(t_1,r_1)$ and $(t_2,r_2)$ are in $\Omega^{ \lambda}$. Then for $J_2$ we obtain
		\begin{equation*}
		\begin{split}
		J_2 &\leq \frac{1}{8} \int_{P^{ \lambda}(t_1,r_1) \triangle P^{ \lambda}(t_2,r_2)} \left(\alpha^2 + \frac{1}{(R{-}\sigma)^2}\right) |h(\tau,\sigma)| \, d\tau d\sigma 
		\\
		&  \leq \frac{1}{8} \left(\alpha^2 + \frac{1}{(R{-}\rho^\lambda(T))^2}\right) M |P^{ \lambda}(t_1,r_1) \triangle P^{ \lambda}(t_2,r_2)|
		\\
		&  \leq \frac{1}{8} \left(\alpha^2 + \frac{1}{(R{-}\rho^\lambda(T))^2}\right) M C_2 (T + y) |(t_1,r_1)-(t_2,r_2)|
		\\
		& \leq \frac{1}{8} \left(\alpha^2 + \frac{1}{(R{-}\rho^\lambda(T))^2}\right)  M C_2 (T + y) \,\delta \leq \frac{\varepsilon}{2} .
		\end{split}
		\end{equation*}
	For the estimate of $J_1$, one proceeds in the same way as in the proof of \cite[Lemma 4.4]{RiNa2020}, obtaining $J_1 < \frac{\varepsilon}{2}$. Hence, the above estimates for $J_1$ and $J_2$ yield to 
	\begin{equation*}
	|\Psi_1[h,\lambda](t_1,r_1) - \Psi_1[h,\lambda](t_2,r_2) | \leq \varepsilon
	\end{equation*}
	and the lemma is proven.
	\end{proof}

Fix $M > 0$ and  let $y \in (0,\tilde{y})$ with $\tilde{y}$ be given by Lemma \ref{lemmay}. Then, consider the following space: 
\begin{equation*}
\mathcal{Z}(M,T,y) := \mathrm{cl}_{\,C^0} \left(\Psi_1(\mathcal{X}(M,T,y)) \right) \times \mathcal{X}_2(T,y) ,
\end{equation*}
endowed with the distance defined as 
\begin{equation*}
d((h^{(1)},\lambda^{(1)}),(h^{(2)},\lambda^{(2)})) := \max \left\{ \|h^{(1)}-h^{(2)}\|_{ L^2(Q^{T,y})} , \|\lambda^{(1)} - \lambda^{(2)}\|_{ C^0 (I^{y})} \right\}
\end{equation*}
for every $(h^{(1)},\lambda^{(1)}),(h^{(2)},\lambda^{(2)}) \in  \mathcal{Z}(M,T,y)$. Then by Lemma \ref{lemmaequi} and by the Ascoli-Arzel\`a Theorem, $(\mathcal{Z}(M,T,y),d)$ is a complete metric space. Hence we have the following result.
 In the proof we show that we can reduce to the setting of \cite{RiNa2020}, where a corresponding result is proven for the kernel $F=h$; after seeing this, the proof then follows from \cite[Proposition 4.5]{RiNa2020}.

\begin{prop}\label{propcontractionZ}
Assume $T \in \left(0,\frac{\rho_0}{2}\right)$.
In addition to \eqref{assumptionkappa}--\eqref{epsiloncondition}, assume that $\kappa \in C^{0,1}([\rho_0,R))$.
Then for every $M>0$ there exists $ {y} \in (0,\rho_0)$ such that the operator $\Psi$ is a contraction from the $(\mathcal{Z}(M,T,{y}),d)$ into itself. 
\end{prop}

\begin{proof}
	From now on, we will denote by $C$ a constant that may change from line to line. Fix $(h^{(1)},\lambda^{(1)}), (h^{(2)},\lambda^{(2)}) \in \mathcal{Z}(M,T,{y})$,
 where ${y}$ will be chosen later on. By now we require that $y\in(0,\tilde{y})$, where $\tilde{y}$ is given by Lemma \ref{lemmay}. 
	
\paragraph{\bf Lipschitz  estimate  on $\Psi_2$.}
	For a.e. $ s \in I^{y}$ we introduce the functions
	\begin{equation*}
		j(s) := |\dot{h}_0(-s)| + |h_1(-s)| + 1
	\end{equation*}
	and
	\begin{equation*}
		J[h^{(i)},\lambda^{(i)}] (s) := \dot{h}_0(-s) - h_1(-s) - \int_{0}^{\lambda^{(i)}(s)} F^{(i)} (\tau, \tau {-} s) \, d\tau ,
	\end{equation*}
	where, for $i=1,2$, the functions $F^{(i)}$ are defined as in \eqref{eqkernelFh} using the functions $h^{(i)}$, respectively. First we observe that, by \eqref{uinitialcondition*}, $j \in L^2(-\rho_0,0)$ and there exists a positive constant $C$ such that 
	\begin{equation}\label{J<j}
		|J[h^{(i)},\lambda^{(i)}](s)| \leq C j(s) \quad \mbox{for a.e. } s \in  I^{y}.
	\end{equation}
	Then, recalling \eqref{lambdaequation} and \eqref{psi2equation} we have the following estimate:
	\begin{equation*} 
		\begin{split}
			&\|\Psi_2[h^{(1)},\lambda^{(1)}] - \Psi_2[h^{(2)},\lambda^{(2)}] \|_{C^0(I^{y})} \leq \frac{1}{2} \int_{-\rho_0}^{-\rho_0 + y} \left| \Lambda^{h^{(1)},\lambda^{(1)}}(\sigma) - \Lambda^{h^{(2)},\lambda^{(2)}]}(\sigma) \right| \,d\sigma
			\\
			&\leq \frac{1}{2} \int_{-\rho_0}^{-\rho_0 + y}
			\left|\frac{(J[h^{(1)},\lambda^{(1)}] (\sigma))^2}{2(R{-}\lambda^{(1)}(\sigma) {+} \sigma)\, e^{\alpha\lambda^{(1)}(\sigma)}\, \kappa(\lambda^{(1)}(\sigma) {-} \sigma)} - \frac{(J[h^{(2)},\lambda^{(2)}] (\sigma))^2}{2(R{-}\lambda^{(2)}(\sigma) {+} \sigma) \,e^{\alpha\lambda^{(2)}(\sigma)} \,\kappa(\lambda^{(2)}(\sigma) {-} \sigma)}\right| \, d\sigma
			\\
			&\leq \frac{1}{2} \int_{-\rho_0}^{-\rho_0 + y}
			\left|\frac{d^{(2)}(\sigma) (J[h^{(1)},\lambda^{(1)}] (\sigma))^2 - d^{(1)}(\sigma) (J[h^{(2)},\lambda^{(2)}] (\sigma))^2}{D(\sigma)}\right| \, d \sigma ,
		\end{split}
	\end{equation*}
	 where for $\sigma \in I^y$  we set 
	\begin{align*}
	d^{(1)}(\sigma) &:= (R{-}\lambda^{(1)}(\sigma) {+} \sigma)\, e^{\alpha\lambda^{(1)}(\sigma)}\, \kappa(\lambda^{(1)}(\sigma) {-} \sigma),
	\\
	d^{(2)}(\sigma)&:= (R{-}\lambda^{(2)}(\sigma) {+} \sigma)\, e^{\alpha\lambda^{(2)}(\sigma)} \,\kappa(\lambda^{(2)}(\sigma) {-} \sigma),
	\\
	D(\sigma) &:= 2 d^{(1)}(\sigma)\, d^{(2)}(\sigma).
	\end{align*}
	We observe that, by the Lipschitz condition on $\lambda^{(1)}$ and $\lambda^{(2)}$ and by assumption \eqref{assumptionkappa}, for a.e. $\sigma \in I^y$ the term $|D(\sigma)|$ is bounded from below by a positive constant, $d_1(\sigma)$ and $d_2(\sigma)$ are bounded from above, and the quantity $|d_1(\sigma)-d_2(\sigma)|$ can be estimated in terms of $|\lambda_1(\sigma) - \lambda_2(\sigma)|$. Hence, by the triangle inequality and by \eqref{J<j} we get
	\begin{equation*}
	\begin{split}
	&\|\Psi_2[h^{(1)},\lambda^{(1)}] - \Psi_2[h^{(2)},\lambda^{(2)}] \|_{C^0(I^{y})}
	\\
	&\leq 
	C \int_{-\rho_0}^{-\rho_0 + y} d_2(\sigma) j(\sigma) \left|J[h^{(1)},\lambda^{(1)}] (\sigma)	 - J[h^{(2)},\lambda^{(2)}] (\sigma)\right|\, d \sigma  
	+ C  \int_{-\rho_0}^{-\rho_0 + y} (j (\sigma))^2 \left|d_1(\sigma) - d_2(\sigma)\right|\, d \sigma
	\\
	&\leq
	C \int_{-\rho_0}^{-\rho_0 + y} d_2(\sigma) j(\sigma)
	\left| \int_{0}^{\lambda^{(1)}(\sigma)} h^{(1)} (\tau, \tau {-} \sigma) \, d\tau - \int_{0}^{\lambda^{(2)}(\sigma)} h^{(2)} (\tau, \tau {-} \sigma) \, d\tau\right|\,d\sigma
	\\
	&
	\quad
	+ C  \int_{-\rho_0}^{-\rho_0 + y} (j (\sigma))^2 \left|\lambda_1(\sigma) - \lambda_2(\sigma)\right|\, d \sigma,
	\end{split}
	\end{equation*}
	where we have used the fact that
	\begin{equation*}
	\begin{split}
	\left|J[h^{(1)},\lambda^{(1)}] (\sigma)	 - J[h^{(2)},\lambda^{(2)}] (\sigma)\right| 
	& \leq 
	\left| \int_{0}^{\lambda^{(1)}(\sigma)} F^{(1)} (\tau, \tau {-} \sigma) \, d\tau - \int_{0}^{\lambda^{(2)}(\sigma)} F^{(2)} (\tau, \tau {-} \sigma) \, d\tau\right|
	\\
	&\leq C \left| \int_{0}^{\lambda^{(1)}(\sigma)} h^{(1)} (\tau, \tau {-} \sigma) \, d\tau - \int_{0}^{\lambda^{(2)}(\sigma)} h^{(2)} (\tau, \tau {-} \sigma) \, d\tau\right|.
	\end{split}
	\end{equation*}
	Then, one can proceed in the same way as in the proof of  \cite[inequality (4.15)]{RiNa2020}  (originally written in the case $F^{(i)}=h^{(i)}$) deducing that, choosing $y$ small enough, one gets:
	\begin{equation*}
		\|\Psi_2[h^{(1)},\lambda^{(1)}] - \Psi_2[h^{(2)},\lambda^{(2)}] \|_{ C^0(I^{y})} \leq \frac{1}{2} \, d((h^{(1)},\lambda^{(1)}),(h^{(2)},\lambda^{(2)})).
	\end{equation*}
\paragraph{\bf Lipschitz estimate on $\Psi_1$.} 
	By the triangle inequality 
	\begin{align*}
	\|\Psi_1[h^{(1)},\lambda^{(1)}] - \Psi_1[h^{(2)},\lambda^{(2)}] \|_{ L^2(Q^{T,y})} &\leq\|\h^{ \lambda^{(1)}} \chi_{\Omega^{ \lambda^{(1)}}} - \h^{ \lambda^{(2)}} \chi_{\Omega^{ \lambda^{(2)}}}\|_{ L^2(Q^{T,y})}
	\\
	& \quad + \frac{1}{2} \left\|\Phi^{ \lambda^{(1)}}[F^{(1)}]  \chi_{\Omega^{ \lambda^{(1)}}} - \Phi^{ \lambda^{(2)}}[F^{(2)}]  \chi_{\Omega^{ \lambda^{(2)}}}\right\|_{ L^2(Q^{T,y})}.
	\end{align*} 
	We proceed estimating the two norms separately. For the first term $\|\h^{ \lambda^{(1)}} \chi_{\Omega^{ \lambda^{(1)}}} - \h^{ \lambda^{(2)}} \chi_{\Omega^{ \lambda^{(2)}}}\|_{ L^2(Q^{T,y})}$
	one can proceed in the same way of the proof  of \cite[inequality (4.20)]{RiNa2020}, obtaining 
	\begin{equation*}
	\|\h^{ \lambda^{(1)}} \chi_{\Omega^{ \lambda^{(1)}}} - \h^{ \lambda^{(2)}} \chi_{\Omega^{ \lambda^{(2)}}}\|_{ L^2(Q^{T,y})} \leq C  \sqrt{y} \|\lambda^{(1)}-\lambda^{(2)}\|_{ C^0(I^{y})}. 
	\end{equation*}
	We detail the estimate of the second term: 
	\begin{align*}
	&\left\|\Phi^{\lambda^{(1)}}[F^{(1)}]  \chi_{\Omega^{\lambda^{(1)}}} - \Phi^{\lambda^{(2)}}[F^{(2)}]  \chi_{\Omega^{\lambda^{(2)}}}\right\|^2_{L^2(Q^{T,y})}
	\\
	&=
	\iint_{ Q^{\lambda^{(1)},T,y} \setminus Q^{\lambda^{(2)},T,y} }
	\left|\iint_{P^{\lambda^{(1)}}(t,r) } \frac{1}{4}\left(\alpha^2 + \frac{1}{(R{-}\sigma)^2}\right) h^{(1)}(\tau,\sigma) \, d\tau d\sigma  \right|^2 \, dr \, dt 
	\\ 
	&\quad + \iint_{ Q^{\lambda^{(2)},T,y} \setminus Q^{\lambda^{(1)},T,y} }
	\left|\iint_{P^{\lambda^{(2)}}(t,r) } \frac{1}{4}\left(\alpha^2 + \frac{1}{(R{-}\sigma)^2}\right) h^{(2)}(\tau,\sigma) \, d\tau d\sigma  \right|^2 \, dr \, dt 
	\\
	& \quad+ \iint_{Q^{\lambda^{(1)},T,y} \cap \, Q^{\lambda^{(2)},T,y} }  
	\left|\iint_{P^{\lambda^{(1)}}(t,r) } F^{(1)}(\tau,\sigma) \, d\tau d\sigma - \iint_{P^{\lambda^{(2)}}(t,r) } F^{(2)}(\tau,\sigma) \, d\tau d\sigma  \right|^2 \, dr \, dt
	\\
	&
	\leq 
	C \iint_{ Q^{\lambda^{(1)},T,y} \setminus Q^{\lambda^{(2)},T,y} }
	\left|\iint_{P^{\lambda^{(1)}}(t,r) } h^{(1)}(\tau,\sigma) \, d\tau d\sigma  \right|^2 \, dr \, dt 
	\\ 
	&\quad +C \iint_{ Q^{\lambda^{(2)},T,y} \setminus Q^{\lambda^{(1)},T,y} }
	\left|\iint_{P^{\lambda^{(2)}}(t,r) }  h^{(2)}(\tau,\sigma) \, d\tau d\sigma  \right|^2 \, dr \, dt 
	\\
	& \quad+C \iint_{Q^{\lambda^{(1)},T,y} \cap \, Q^{\lambda^{(2)},T,y} }  
	\left|\iint_{P^{\lambda^{(1)}}(t,r) } h^{(1)}(\tau,\sigma) \, d\tau d\sigma - \iint_{P^{\lambda^{(2)}}(t,r) } h^{(2)}(\tau,\sigma) \, d\tau d\sigma  \right|^2 \, dr \, dt ,
	\end{align*} 
	where we have used the fact that the kernel $\left(\alpha^2 + \frac{1}{(R-\sigma)^2}\right)$ is bounded by a constant  on each domain of integration. Now the proof proceeds in the same way of the proof of  \cite[inequality (4.22)]{RiNa2020} (where an analogous estimate is proven for $F^{(i)} = h^{(i)}$), which shows that
	\begin{equation*}
	\left\|\Phi^{\lambda^{(1)}}[F^{(1)}]  \chi_{\Omega^{\lambda^{(1)}}} - \Phi^{\lambda^{(2)}}[F^{(2)}]  \chi_{\Omega^{\lambda^{(2)}}}\right\|^2_{L^2(Q^{T,y})} \leq y \, C d((h^{(1)},\lambda^{(1)}),(h^{(2)},\lambda^{(2)}))^2.
	\end{equation*}
	Then by choosing $ y$ small enough we finally deduce 
	\begin{equation*}
	\|\Psi_1[h^{(1)},\lambda^{(1)}] - \Psi_1[h^{(1)},\lambda^{(2)}] \|_{ L^2(Q^{T,y})} \leq \frac{1}{2}  d((h^{(1)},\lambda^{(1)}),(h^{(2)},\lambda^{(2)})).
	\end{equation*}
	The conclusion of the proof follows.
\end{proof}

Now we are in the position to state and prove the main result of this section.  Here we remove the restriction $T \in \left(0,\frac{\rho_0}{2}\right)$ made in the previous technical results. 
The existence of a solution holds either in the time interval $[0,+\infty)$, or up to a time where the debonding front reaches the origin, so the film is completely debonded, apart from a point. (Recall that the origin is degenerate for our statement of Griffith's criterion.) 

\begin{teo} \label{thm:main}
 Let $\kappa$  satisfy assumptions \eqref{assumptionkappa}--\eqref{epsiloncondition}. Assume \eqref{uinitialcondition*} and \eqref{uinitialconditioncomp}. 
	Then, there exist $T^\ast>0$ and a  unique  radial solution $(u,\rho)$ to the coupled problem \eqref{princeq}--\eqref{Griffithcriterion} in $[0,T^\ast]$, in the sense of Definition \ref{def:sol-coupled}. Moreover: 
	\begin{enumerate}
		\item it either holds $T^\ast= +\infty$, or $T^\ast < +\infty$ and $\displaystyle\lim\limits_{s \to T^\ast} \rho(s) = R$; 
		\item $u$ has a continuous representative on  $[0,T^\ast)\times B(0,R)$ and 
		\begin{equation*}
		u \in C^0([0,T^\ast); W^{1,2}_{\textnormal{rad}}(   B(0,R)  )) \cap C^1([0,T^\ast); L^2_{\textnormal{rad}}( B(0,R) )).
		\end{equation*}
	\end{enumerate}
\end{teo}

\begin{proof}
It suffices to prove that there exists a unique pair $(h,\rho)$ solution of  \eqref{princeqh}--\eqref{oderho} and thus obtain the desired solution $(u,\rho)$ to the coupled problem \eqref{princeq}--\eqref{Griffithcriterion}, in the sense of Definition \ref{def:sol-coupled}. The solution we find is unique among radial solutions; for further comments on the uniqueness of displacements, see the Remarks below. 

	We observe that, by restricting \eqref{princeqh} to the triangle $\mathbf{T}= \{(t,r) \,|\, 0\leq t\leq \rho_0, \ 0 \leq r \leq \rho_0 {-} t \}$ and arguing as in  Proposition \ref{propcontration}, there exist a unique  solution $h^{(0)}$ of \eqref{princeqh} in $\mathbf{T}$. 
 Observe that this domain is  independent of the debonding evolution. 
	
	Fix now $\varepsilon = \varepsilon(\rho_0) >0$ as in assumption \eqref{epsiloncondition},  and introduce a toughness $\tilde{\kappa}$ that coincides with $\kappa$ in $[\rho_0,\rho_0 + \varepsilon]$ and is constantly equal to $\kappa(\rho_0+\varepsilon)$ after $\rho_0+\varepsilon$.  Clearly $\tilde{\kappa} \in C^{0,1}([\rho_0,R))$.
	We can now apply Proposition \ref{propcontractionZ} and find ${y} \in (0,\rho_0)$ and $T^{{y}} \in \left(0,\frac{\rho_0}{2}\right)$ such that there exists a unique pair $(h^{(1)},\lambda^{(1)})$ with
	\begin{equation}\label{symh^1rho^1}
	\begin{dcases}
	h^{(1)}(t,r) = \left( \h^{\lambda^{(1)}}(t,r) + \frac{1}{2} \Phi^{\lambda^{(1)}}[F](t,r) \right) \chi_{\Omega^{\lambda^{(1)}}}(t,r) & \mbox{ for a.e. } (t,r) \in  Q^{T^{{y}},{y}},
	\\
	\lambda^{(1)} (s) = \frac{1}{2} \int_{-\rho_0}^{s} \left( 1+ \max \{\Lambda^{h^{(1)},\lambda^{(1)}}(\sigma),1 \} \right) \,d\sigma & \mbox{ for every } s \in I^{{y}}.
	\end{dcases}
	\end{equation}  
	Now let  $\rho^{(1)}$ be defined by  $\rho^{(1)}:=t-(\lambda^{(1)})^{-1}(t)$. Notice that $\rho^{(1)}(0)=\rho_0$. Then, by the continuity of $\tilde{\kappa}$, by the fact that $\tilde{\kappa}=\kappa$ on $[\rho_{0},\rho_0+\varepsilon]$, and by the regularity properties given by Lemma \ref{lemmaAandPhi}, there exists $T \in (0,T^{{y}})$ such that 
	\begin{enumerate}
		\item the pair $(h^{(1)},\lambda^{(1)})$ is the unique solution of \eqref{symh^1rho^1} in $[0,T]$, where  $\kappa$ is replaced by $\tilde{\kappa}$;
		\item the function $h$ defined as 
		\begin{equation*}
		h= h^{(0)} \chi_{\mathbf{T}} + h^{(1)} \chi_{Q^{T,{y}} \setminus \mathbf{T}}
		\end{equation*}
		solves \eqref{princeqh} on $\mathbf{T} \cup  Q^{T,{y}}$;
		\item the pair $(h,\rho^{(1)})$ solves \eqref{princeqh}--\eqref{oderho} in $[0,T]$ and 
		\begin{equation*}
		h \in C^0([0,T]; W^{1,2}(0,R)) \cap C^1([0,T]; L^2(0,R)).
		\end{equation*}
	\end{enumerate}
	
	This provides a short-time existence and uniqueness result for the coupled problem \eqref{princeqh}--\eqref{oderho}. Observe that a solution $(h, \rho)$ to the problem is locally unique as long as it is well-defined. Indeed, denote  (with a slight abuse of notation) again with $[0, T]$ a given interval where \eqref{princeqh}--\eqref{oderho} holds. For fixed $\overline{t} \in [0, T]$ one can consider the auxiliary problem \eqref{symh^1rho^1}, replacing $\rho_0$ by $\overline{\rho}_0 := \rho(\overline{t})$, $h_0$ by $h(\overline{t}, \cdot) \in W^{1,2}(0, \overline{\rho}_0)$ and $h_1$ by $h_t(\overline{t}, \cdot) \in L^2(0, \overline{\rho}_0)$; notice that $h(\overline{t}, 0) = z(\overline{t})$ and
	$h(\overline{t}, \overline{\rho}_0) = 0$, so the compatibility conditions are satisfied.  The same argument above then entails the existence of a uniquely defined solution in a neighborhood of $\overline{t}$, which must therefore coincide with the given one. By a standard prolongation argument, we can eventually define $T^\ast$ as the maximal time of existence for the solution $(h,\rho)$ of \eqref{princeqh}--\eqref{oderho} in $[0,T^\ast)$.
	If $T^\ast = +\infty$ we conclude. 	If $T^\ast < +\infty$, then, since $\rho(s) < R$ for every $s \in (0,T^\ast)$, surely one has $\lim\limits_{s \to T^\ast} \rho(s) \leq R$. Set $\rho(T^\ast)$ to be equal to this limit. Now, if by contradiction  a strict inequality $\rho(T^\ast)< R$ holds, the prolongation argument sketched above may still be applied (with $\overline{t}$ replaced by $T^\ast$). This contradicts the maximality of $T^\ast$ and gives $\lim\limits_{s \to T^\ast} \rho(s) = R$. The proof is concluded.
\end{proof}

 We conclude the discussion of the coupled problem by observing that the debonding front $\rho$ found in the previous theorem is more regular if the data of the problem are regular. As a consequence, we can prove that the corresponding displacement $u$ is the unique solution of the wave equation in the growing domain given by $\rho$. This justifies to some extent the ansatz of radial solutions: indeed, starting from regular radial data, there is a radial debonding front satisfying Griffith's criterion and the corresponding displacement is always radial. 

\begin{rem}\label{rmk:rho-regularity}
	With suitable assumptions on the  the data of the problem,  the regularity of the function $\rho$  found in the previous theorem actually increases. Namely, assume that the toughness $\kappa$ is in the class $C^{0,1}([\rho_0,R])$ and that the initial data $u_0$, $u_1$, and $w$ satisfy 
	\begin{equation}\label{rmk:assumptionsdata}
		 w \in C^{1,1}(0,+\infty), \quad u_0 \in C^{1,1}_{\textnormal{rad}}(\mathcal{C}_{R-\rho_0,R}), \quad u_1 \in C^{0,1}_{\textnormal{rad}}(\mathcal{C}_{R-\rho_0,R}) ,
	\end{equation}
 where $\textnormal{rad}$ means that the initial data are radial, as above. 
Then, we have that $\rho \in C^{1,1}([0,T^\ast])$. 

To see this, we first observe that
	\begin{equation}\label{rmk:eq-f}
	t \mapsto f(t):= \int_{0}^{t} F(\tau,\tau {+}\rho(t) {-} t) \,d\tau \in C^{0,1}\left(\left[0,\frac{\rho_0}{2}\right]\right) 
	\end{equation} 
(cf.\ definition and regularity of $F$ in \eqref{eqkernelFh}--\eqref{regularityGF}). Indeed, the derivative of the right-hand side of \eqref{rmk:eq-f} is given by
    \begin{equation}\label{rmk:eq-f'1}
    \dot{f}(t) = F(t,\rho(t)) + \int_{0}^{t} \frac{d}{dt}F(\tau,\tau {+}\rho(t) {-} t) \,d\tau ,
    \end{equation}
    where 
    \begin{equation}\label{rmk:eq-f'2}
    \begin{split}
    \frac{d}{dt}F(\tau,\tau {+}\rho(t) {-} t) &= \frac{\partial}{\partial \sigma} [F(\tau,\tau {+}\rho(t) {-} t)] \,(\dot{\rho} (t) {-} 1) 
    \\
    &= \frac{1}{4} \,\frac{2}{\left(R{-}\tau {-}\rho(t) {+} t\right)^3} \, h(\tau , \tau {+}\rho(t) {-} t)\, (\dot{\rho} (t) {-} 1) 
    \\
    & \quad+ \frac{1}{4} \left( \alpha^2 + \frac{1}{\left(R{-}\tau {-}\rho(t) {+} t\right)^2}\right)  h_r(\tau , \tau {+}\rho(t) {-} t) \,(\dot{\rho} (t) {-} 1)\,.
    \end{split}
    \end{equation}
    Recalling  that $h \in C^0([0,T^\ast); W^{1,2}(0,R))$, 
$h_r \in C^0([0,T^\ast); L^2(0,R))$, and $\dot{\rho} \in L^\infty\left(\left[0,\frac{\rho_0}{2}\right]\right)$, we have $\dot{f} \in L^\infty\left(\left[0,\frac{\rho_0}{2}\right]\right) $, so that $f \in C^{0,1}\left(\left[0,\frac{\rho_0}{2}\right]\right)$. From this, and by \eqref{rmk:assumptionsdata}, we deduce that 
	\begin{equation*}
	t \mapsto \max \left\{ 0, \ \frac{\mathcal{G}_0(t) - \kappa(\rho(t))}{\mathcal{G}_0(t) + \kappa(\rho(t))} \right\} \in C^{0,1}\left(\left[0,\frac{\rho_0}{2}\right]\right)
	\end{equation*}
	(cf. \eqref{formulaG_0}). By Remarks \ref{rmk:transl} and \ref{remG}, we can extend this regularity property to the whole interval $[0,T^\ast]$. We conclude that the right-hand side of \eqref{oderho} belongs to $C^{1,1}([0,T^\ast])$ as well, so that $\rho \in C^{1,1}([0,T^\ast])$.
\end{rem}

\begin{rem} \label{rmk:radial-disp}
	Consider the pair $(u,\rho)$  found in the previous theorem, solving the coupled problem \eqref{princeq}--\eqref{Griffithcriterion} in the sense of Definition \ref{def:sol-coupled}. With suitable assumptions on the data of the problem, one can also show that the function $u$ is the unique solution of  \eqref{princeq} corresponding to the  debonding  evolution $t\mapsto \rho(t)$,  \emph{without} the constraint on the displacement to be radial.
	Namely, assume that the toughness $\kappa$ is of class $C^{1,1}([\rho_0,R])$, that the initial data $u_0$, $u_1$, and $w$ satisfy 
	\begin{equation}\label{rmk:assumptionsdata*}
w \in C^{2,1}(0,+\infty), \quad u_0 \in C^{2,1}_{\textnormal{rad}}(\mathcal{C}_{R-\rho_0,R}), \quad u_1 \in C^{1,1}_{\textnormal{rad}}(\mathcal{C}_{R-\rho_0,R}) ,
	\end{equation}
	and that the following first-order compatibility conditions hold:
	\begin{equation}\label{rmk:firstordercompdata*}
	\begin{dcases}
	h_1(0) = \dot{z}(0) ,
	\\
	h_1(\rho_0) + \dot{h}_0(0) \max \left\{ 0, \frac{(\dot{h}_0(\rho_0) - h_1(\rho_0))^2 - 2 (R{-}\rho_0)\kappa (\rho_0)}{(\dot{h}_0(\rho_0) - h_1(\rho_0))^2 + 2 (R{-}\rho_0)\kappa(\rho_0)}  \right\} = 0 ,
	\end{dcases}
	\end{equation} 
	where $h$ is given by \eqref{defh} and $z$ by \eqref{eqinitialdatah}. We then claim that 
\begin{equation}\label{eq: claim}
	\frac{\mathcal{G}_0(t) - \kappa(\rho(t))}{\mathcal{G}_0(t) + \kappa(\rho(t))} \in C^{1,1}([0,T^\ast]).
	\end{equation}

To see this, observe that, since $\rho \in C^{1,1}([0,T^\ast])$ by Remark \ref{rmk:rho-regularity}, with \eqref{rmk:assumptionsdata*}, \eqref{rmk:firstordercompdata*},
{ one has that $h \in C^{1,1}(\overline{\Omega}^\rho)$.} With this, and again using that $\rho \in C^{1,1}([0,T^\ast])$, the function $\dot{f}$ in \eqref{rmk:eq-f'1}--\eqref{rmk:eq-f'2} actually belongs to $C^{0,1}\left(\left[0,\frac{\rho_0}{2}\right]\right)$, hence $f \in C^{1,1}\left(\left[0,\frac{\rho_0}{2}\right]\right)$. Again exploiting Remarks \ref{rmk:transl} and \ref{remG}, and by the regularity assumption \eqref{rmk:assumptionsdata*}, we get \eqref{eq: claim}.

Now \eqref{oderho} ensures that $\dot\rho$ coincides with the positive part of the $C^{1,1}$ function in \eqref{eq: claim}, therefore we may (just) say that $\rho$ is $C^{2,1}$ in an open neighborhood of  each point $\overline{t} \in A$, where 
	\begin{equation*}
	A:=\{t \in [0,T^\ast]: \, \dot{\rho}(t) \neq 0 \},
	\end{equation*} 
	since $\mathcal{G}_0(t) - \kappa(\rho(t)) > 0$ for every $t$ in an open neighborhood of $\overline{t} \in A$. 
	For instance, we observe that if the following condition is fulfilled,
	\begin{equation*}
	\left(\dot{v}_0(\rho_0) - v_1(\rho_0)\right)^2 > 2 \kappa(\rho_0),
	\end{equation*}
	then $\dot{\rho}(0)>0$, hence it exists $\overline{\tau} >0$ such that $\rho \in C^{2,1}([0,\overline{\tau}])$. Therefore,  by applying Proposition \ref{prop-app-uniq},  the function $u$ given by Theorem \ref{thm:main} is the unique solution of \eqref{princeq} in the growing domain $\mathcal{C}_{R-\rho(t),R}$. 
	
	We also observe that, if $\partial A$ is finite (which can be checked {\it a posteriori}), then $\rho$ is $C^{2,1}$-piecewise in $[0,T^\ast]$. In this latter case one can apply Remark \ref{rmk-app} and obtain again that the function $u$ is the unique solution of \eqref{princeq}  in $\mathcal{C}_{R-\rho(t),R}$. 
\end{rem}

\appendix

\section{Proof of Proposition \ref{propderivT}}\label{secproofprop}
 In this appendix we provide a proof of Proposition \ref{propderivT}. The proof follows the lines of \cite[Proposition 2.1]{RiNa2020}, with some modifications due to the nonconstant kernel in \eqref{princeqh},
however we present the argument in order to highlight the role of the terms appearing in the energy balance formula \eqref{formuladerivativeT}--\eqref{formuladerivativeT-2}.

As in Remark \ref{remv} we denote by $\vi$ the solution to the pure wave equation \eqref{eq:undamp} with data $v_0$, $v_1$, and $w$ as in \eqref{vinitialcondition}--\eqref{vinitialconditioncomp}.
We extend $\vi$ by setting $\vi(t,r)=0$ if $r\in(\rho(t),R)$.
For $(t,r)\in\Omega'$ there is a formula for $\vi$ analogous to \eqref{A(t,r)}
(where $\Omega'$ is defined in \eqref{def:Omega}).
It is convenient to use the equivalent formula 
\begin{equation}\label{V}
\vi (t,r) = \vi_1(t{+}r) + \vi_2(t{-}r) \quad\forall (t,r) \in \Omega' ,
\end{equation}
where 
\begin{equation}\label{a_1a_2}
\begin{split}
\vi_1(s) := 	
\begin{dcases}
\frac{1}{2} v_0(s) + \frac{1}{2} \int_{0}^{s} v_1(\sigma) \,d\sigma &\mbox{ if } s \in (0,\rho_0],
\\
-\frac{1}{2} v_0(-\omega(s)) + \int_{0}^{ - \omega(s) } v_1(\sigma) \,d\sigma &\mbox{ if } s \in (\rho_0,2t^\ast),
\end{dcases}
\\
\vi_2(s) := 	
\begin{dcases}
\frac{1}{2} v_0(-s) - \frac{1}{2} \int_{0}^{-s} v_1(\sigma) \,d\sigma &\mbox{ if } s \in (-\rho_0,0],
\\
w(s) -\frac{1}{2} v_0(s) - \frac{1}{2} \int_{0}^{s} v_1(\sigma) \,d\sigma &\mbox{ if } s \in (0,\rho_0),
\end{dcases}
\end{split}
\end{equation}
and $t^\ast$ is defined by $t^\ast=T$ if $\rho(T)>T$, otherwise $t^\ast$ is the unique solution of $t^\ast=\rho(t^\ast)$. 
 In fact notice that if $(t,r)\in\Omega'_1$, then $t{-}r\in(-\rho_0,0]$ and $t{+}r\in(0,\rho_0]$; 
if $(t,r)\in\Omega'_2$, then $t{-}r,t{+}r\in(0,\rho_0]$; 
if $(t,r)\in\Omega'_3$, then $t{-}r\in(-\rho_0,0]$ and $t{+}r\in(\rho_0,2t^\ast)$; cf.\ also Figure \ref{fig:cone}. 

We will need the derivatives of $\vi$,  
as well as the derivatives of a generic function \eqref{defmathcalF} defined by double integration on the domain $P(t,r)$ given in \eqref{R(t,r)};
in the following lemma we recall the corresponding formulas, computed in \cite[Lemmas 1.10 and 1.11]{RiNa2020} to which we refer for the proof.

\begin{lemma}\label{lemmaAandPhiappedix}
	The following hold true.
	\begin{enumerate}
		\item[(i)]  Let $\vi$, $\vi_1$, and $\vi_2$ be as in \eqref{V}--\eqref{a_1a_2}.
Then for every $t \in [0,\frac{\rho_0}{2}]$ and for a.e.\ $r \in (0,R)$ the partial derivatives of $\vi$ are given by 
		\begin{equation*}
		\begin{split}
		\vi_t(t,r) &=
		\begin{dcases}
		\dot{\vi}_1(t{+}r) + \dot{\vi}_2(t{-}r) &\mbox{ if } r \in (0,\rho(t)),
		\\
		0&\mbox{ if } r \in (\rho(t),R),
		\end{dcases}
		\\
		\vi_r(t,r) &=
		\begin{dcases}
		\dot{\vi}_1(t{+}r) - \dot{\vi}_2(t{-}r)&\mbox{ if } r \in (0,\rho(t)),
		\\
		0 &\mbox{ if } r \in (\rho(t),R).
		\end{dcases}
		\end{split}
		\end{equation*}
		
		\item[(ii)] Let $H \in L^2(\Omega')$ and $\Phi[H]$ as in \eqref{defmathcalF}. Then for every $t \in [0,\frac{\rho_0}{2}]$ the partial derivatives $\Phi[H]_t(t,\cdot)$, $\Phi[H]_r(t,\cdot)$ exist and for every $t \in [0,\frac{\rho_0}{2}]$ and for a.e.\ $r \in (0,R)$ we have
		\begin{align*}
		&\Phi[H]_t(t,r) =
		\begin{dcases}
		g_1(t,r) + g_2(t,r) &\mbox{ if } r \in (0,\rho(t)),
		\\
		0&\mbox{ if } r \in (\rho(t),R),
		\end{dcases}
		\\
		&\Phi[H]_r(t,r) =
		\begin{dcases}
		g_1(t,r) - g_2(t,r) &\mbox{ if } r \in (0,\rho(t)),
		\\
		0&\mbox{ if } r \in (\rho(t),R),
		\end{dcases}
		\end{align*}
		with
		\begin{equation}\label{deff_1}
		g_1(t,r):=
		\begin{dcases}
		\int_0^t H(\tau,t{+}r{-}\tau) \,d\tau &\mbox{if } 0\leq r \leq \rho_0{-}t,
		\\
		-\dot{\omega}(t{+}r) \int_{0}^{\psi^{-1}(t{+}r)} \!\!\!\!\!\!\!\!\!\! H(\tau,\tau {-}\omega(t{+}r)) \,d\tau  + \int_{\psi^{-1}(t{+}r)}^t \!\!\!\!\!\!\!\!\!\! H(\tau,t{+}r{-}\tau) \,d\tau  &\mbox{if }\rho_0{-}t\leq r \leq \rho(t)
		\end{dcases}
		\end{equation}
		and  
		\begin{equation}\label{deff_2}
		 g_2(t,r):=
		\begin{dcases}
		\int_0^t H(\tau, r{-}t{+}\tau) \,d\tau &\mbox{if } t\leq r \leq \rho(t),
		\\
		- \int_0^{t-r} H(\tau,t{-}r{-}\tau) \,d\tau + \int_{t-r}^t H(\tau,r{-}t{+}\tau) \,d\tau  &\mbox{if } 0 \leq r \leq t .
		\end{dcases}
		\end{equation}
	\end{enumerate}	
\end{lemma}

 We remark that formulas \eqref{deff_1}--\eqref{deff_2} correspond to an integration on the boundary of $P(t,r)$. 

\begin{proof}[Proof of Proposition \ref{propderivT}]
	 We detail the proof of formula \eqref{formuladerivativeT}.  
	Recalling that $\mathcal{A}$ is absolutely continuous by construction, cf.\ \eqref{eqA}, changing variables one has for a.e. $t \in [0,T]$
	\begin{equation*}
	\dot{\mathcal{A}}(t) = 2\pi \alpha \int_{0}^{\rho(t)} (R{-}r) \, v_t^2(t,r) \,dr .
	\end{equation*}
	The internal energy defined in \eqref{eqE} is rewritten as 
	\begin{equation}\label{eqEappendix}
	\mathcal{E}(t) = \pi \int_{0}^{\rho(t)} (R{-}r)\, (v_t^2(t,r) + v_r^2(t,r)) \,dr .
	\end{equation}
	By Remark \ref{remv}, we know that for every $(t,r) \in \overline{\Omega}_T$
	\begin{equation*}
	v (t,r) = \vi(t,r) + \frac{1}{2} \iint_{R(t,r)} \left( -\frac{v_r(\tau , \sigma)}{R{-}\sigma}  - \alpha \, v_t(\tau , \sigma) \right) \,d\sigma \,d\tau = \vi(t,r) + \frac{1}{2} \Phi[G](t,r) ,
	\end{equation*}
	where $\Phi$ is defined as in \eqref{defmathcalF}  and $G$ is given by \eqref{eqkernelFv}. Hence, by Lemma \ref{lemmaAandPhiappedix}(i), for every $t \in [0,T]$ we get
	\begin{equation}\label{eqv_tv_r}
	\begin{split}
	v_t(t,r) = \dot{\vi}_1(t{+}r) + \dot{\vi}_2(t{-}r) + \frac{1}{2} \Phi[G]_t(t,r) \quad&\mbox{ for a.e. } r \in [0,\rho(t)],
	\\
	v_r(t,r) = \dot{\vi}_1(t{+}r) - \dot{\vi}_2(t{-}r) + \frac{1}{2} \Phi[G]_r(t,r) \quad&\mbox{ for a.e. } r \in [0,\rho(t)].
	\end{split}
	\end{equation}
 We now use the formulas for $\Phi[G]_t(t,r)$ and $\Phi[G]_r(t,r)$ provided by Lemma \ref{lemmaAandPhiappedix}(ii),
where we still denote by $g_1$ and $g_2$ the functions given by \eqref{deff_1}--\eqref{deff_2} with $H=G$. 
	Hence by \eqref{eqEappendix}  and by the  relation $(a{+}b)^2+(a{-}b)^2=2(a^2{+}b^2)$  we have
	\begin{equation}\label{eqE_0}
	\begin{split}
	&\frac{1}{2\pi}\mathcal{E}(t) = 2 \int_{0}^{\rho(t)} (R{-}r) (v_t^2(t,r) + v_r^2(t,r)) \,dr  
	\\
	&= \int_{t}^{t+\rho(t)} (R {-} y {+} t )  \left( \dot{\vi}_1(y) + \frac{1}{2}g_1(t,y{-}t) \right)^2 dy  
	+ \int_{t-\rho(t)}^{t} (R {+} y {-} t )  \left( \dot{\vi}_2(y) + \frac{1}{2}g_2(t,t{-}y) \right)^2 dy  .
	\end{split}
	\end{equation}
 	 It is now convenient to compute the expression of $g_1(t,y{-}t)$ separately for $y\in(t,\rho_0)$ and $y\in(\rho_0,t{+}\rho(t))$, obtaining by \eqref{deff_1}
 	\begin{equation*}
 	g_1(t,y-t)=
 	\begin{dcases}
 	\int_0^t G(\tau,y{-}\tau) \,d\tau &\mbox{ if } t\leq y \leq \rho_0,
 	\\
 	-\dot{\omega}(y) \int_{0}^{\psi^{-1}(y)} G(\tau,\tau {-}\omega(y)) \,d\tau + \int_{\psi^{-1}(y)}^t G(\tau,y{-}\tau) \,d\tau  &\mbox{ if }\rho_0\leq y \leq t+\rho(t),
 	\end{dcases}
 	\end{equation*}
	 where one can observe that the $t$-dependence is only at one limit of the integrals. 
	Analogously, computing $g_2(t,t{-}y)$ for $y\in(t{-}\rho(t),0)$ and $y\in(0,t)$, by \eqref{deff_2} one has 
 	\begin{equation*}
	g_2(t,t-y)=
 	\begin{dcases}
 	\int_0^t G(\tau, \tau{-}y) \,d\tau &\mbox{ if } t-\rho(t)\leq y \leq 0,
 	\\
 	- \int_0^{y} G(\tau,y{+}\tau) \,d\tau + \int_{y}^t G(\tau,\tau{-}y) \,d\tau  &\mbox{ if } 0 \leq y \leq t .
 	\end{dcases}
 	\end{equation*}
	It follows that
 	\begin{equation}\label{d/dt g = G}
 	\frac{d}{dt} \, g_1(t,y{-}t) = G(t,y{-}t),
 	\qquad
 	\frac{d}{dt} \, g_2(t,t{-}y) = G(t,t{-}y).
 	\end{equation}
 	 We now want to derive \eqref{eqE_0}. To this aim, we analyse the regularity of the term $g_1$ for $t \in \left[0, \frac{\rho_0}{2} \right]$ and $y \in [\rho_0, t{+}\rho(t)]$, the case $y\in[t,\rho_0]$ being similar. We define $f(t,y) := g_1(t,y-t)$ and set $I^y:=\{ t \in \left[0,\frac{\rho_0}{2} \right] \,|\, y \leq t+\rho(t) \}$ and $\widehat{\Omega}:= \left\{(t,y) \,| \, 0\leq t \leq \frac{\rho_0}{2} , \ \rho_0 \leq y \leq t{+}\rho(t) \right\}$. 
Since $G \in C^0([0,T];W^{1,2}(0,R))$ and $\omega$ is Lipschitz, the following hold:
 	\begin{enumerate}
 		\item For every $t \in  \left[0, \frac{\rho_0}{2}  \right]$ we have $f(t,\cdot)\in L^\infty(\rho_0,t + \rho(t))$.
 		\item For a.e.\ $y \in [\rho_0,t + \rho(t)]$ we have $f(\cdot,y)\in W^{1,2}(I^y)$.
 		\item We have $\frac{\partial}{\partial t} f = G \in L^2(\widehat{\Omega})$ (cf.\ \eqref{d/dt g = G}). 
 	\end{enumerate} 
Analogous observations hold for $g_2$ with slight modifications. 
Moreover, the function $t\mapsto t +\rho(t)$ is Lipschitz and nondecreasing. By \cite[Theorem A.8]{RiNa2020}, these  properties  guarantee that we can apply the chain rule when deriving \eqref{eqE_0}.  
    Hence  we can compute the total derivative of the energy $\mathcal{E}$, which we split into two terms, 
	\begin{equation*}
	\dot{\mathcal{E}}(t) = S_1(t) + S_2(t) ,
	\end{equation*}
 where $S_1$ contains the boundary terms and $S_2$ contains the integrals terms. Precisely, 
	the function $S_1$ is given by 
	\begin{equation}\label{eqS_1}
	\begin{split}
	S_1(t)  := & - 2\pi R  \left(  \dot{\vi}_1(t) + \frac{1}{2}g_1(t,0) \right)^2
	+ 2\pi(1{+}\dot{\rho}(t)) \, (R{-}\rho(t)) \left(  \dot{\vi}_1(t {+} \rho(t)) + \frac{1}{2}g_1(t,\rho(t)) \right)^2
	\\
	&
	 - 2\pi(1{-}\dot{\rho}(t)) \, (R{-}\rho(t)) \left( \dot{\vi}_2(t {-} \rho(t)) + \frac{1}{2}g_2(t,\rho(t)) \right)^2  + 2\pi R \left(  \dot{\vi}_2(t) + \frac{1}{2}g_2(t,0) \right)^2.
	\end{split}
	\end{equation}
	By straightforward computations  and changing back the variables to $r= y{-}t$ and $r=t{-}y$, respectively, the function  $S_2$ is given by
	\begin{equation*}
	\begin{split}
	S_2(t) &:= 2\pi \int_{t}^{t+\rho(t)} \left( \dot{\vi}_1(y) + \frac{1}{2}g_1(t,y{-}t) \right)^2 \,dy  -  2\pi\int_{t-\rho(t)}^t \left( \dot{\vi}_2(y) + \frac{1}{2}g_2(t,t{-}y) \right)^2 \,dy
	\\
	& \phantom{:=}\
	+ 2\pi\int_{t}^{t+\rho(t)} (R{-}y{+}t) \left( \dot{\vi}_1(y) + \frac{1}{2}g_1(t,y{-}t) \right) \left( -\frac{v_r(t , y{-}t)}{R{-}y{+}t}  - \alpha \, v_t(t , y{-}t) \right) \, dy
	\\
	& \phantom{:=}\
	+ 2\pi \int_{t-\rho(t)}^{t} (R {+}y {-}t) \left( \dot{\vi}_2(y) +  \frac{1}{2}g_2(t,t{-}y) \right) \left( -\frac{v_r(t,t{-}y)}{R {+} y {-}t}  - \alpha \, v_t(t , t{-}y) \right) \, dy
	\\
	& \phantom{:}
	= - 2\pi \alpha \int_{0}^{\rho(t)} (R{-}r)\, v_t^2(t,r) \,dr ,
	\end{split}
	\end{equation*}
	 where we have used the formulas for $v_t$ and $v_r$ in \eqref{eqv_tv_r}. Notice, moreover, that $S_2 + \dot{\mathcal{A}} =0$, hence
	\begin{equation*}
	\dot{\mathcal{T}}(t) = \, S_1(t) + S_2 (t) + \dot{\mathcal{A}} (t) = S_1 (t). 
	\end{equation*}
 We are left with the term $S_1$. In order to simplify that expression, 
we plug the definitions of $g_1$ and $g_2$ \eqref{deff_1}--\eqref{deff_2}, recalling that $H=G$. 
 By an explicit computation of the derivatives of the functions $\vi_1$ and $\vi_2$ given in \eqref{a_1a_2}, from \eqref{eqS_1} we get 
	\begin{equation*}
	\begin{split}
	S_1(t) =& -2\pi R  \left[ \frac{\dot{v}_0(t) + v_1(t)}{2} + \frac{1}{2} \int_{0}^{t} G(\tau, t{-}\tau) \,d \tau  \right]^2 
	\\
	&  + 2\pi(1 {+} \dot{\rho}(t)) \,(R{-}\rho(t))\, \dot{\omega}^2(t {+} \rho(t))  \left[\frac{\dot{v}_0(\rho(t) {-} t) - v_1(\rho(t) {-} t)}{2}  - \frac{1}{2} \int_{0}^{t} G(\tau,\tau {+} \rho(t) {-} t) \,d\tau \right]^2
	\\
	& 
	- 2\pi (1{-}\dot{\rho}(t))\,  (R{-}\rho(t)) \left[ \frac{\dot{v}_0(\rho(t) {-} t) - v_1(\rho(t) {-} t)}{2}  - \frac{1}{2} \int_{0}^{t} G(\tau,\tau {+} \rho(t) {-} t) \,d\tau \right]^2
	\\
	& 
	+ 2\pi R  \left[ \dot{w}(t) - \left( \frac{\dot{v}_0(t) + v_1(t)}{2} + \frac{1}{2} \int_{0}^{t} G(\tau, t{-}\tau) \,d \tau \right) \right]^2 .
	\end{split}
	\end{equation*}
 By using the relation $-a^2 + (b{-}a)^2 = b \,(b{-}2a)$ and the formula $\dot{\omega}(t + \rho(t)) = \frac{1-\dot{\rho}(t)}{1 + \dot{\rho}(t)}$ we obtain 
	\begin{equation*}
	\begin{split}
&	S_1(t) 	=  2\pi R\, \dot{w}(t) \left[ \dot{w}(t) - \left(\dot{v}_0(t) + v_1(t)  + \int_{0}^{t} G(\tau,t{-}\tau) \right) \right] 
\\ & +  \frac\pi2  (R{-}\rho(t)) \left( (1{+}\dot{\rho}(t)) \frac{(1{-}\dot{\rho}(t))^2}{(1 {+} \dot{\rho}(t))^2}  - (1{-}\dot{\rho}(t))  \right) 
	 \left[ \dot{v}_0(\rho(t) {-} t) - v_1(\rho(t) {-} t)  - \int_{0}^{t} G(\tau,\tau {+} \rho(t) {-} t) \,d\tau \right]^2 .
	\end{split}
	\end{equation*}
Finally,  in this expression we plug 
	\begin{equation*}
	\left( (1{+}\dot{\rho}(t)) \,  \frac{(\dot{\rho}(t){-}1)^2}{(1 {+} \dot{\rho}(t))^2} - (1{-}\dot{\rho}(t) ) \right)  = - 2 \dot{\rho}(t) \, \frac{1{-}\dot{\rho}(t)}{1 {+} \dot{\rho}(t)} ,
	\end{equation*}
 thus we see that formula \eqref{formuladerivativeT} holds. Moreover, arguing by translation as in Remark \ref{rmk:transl},
we deduce that $\mathcal{T} \in AC([0,T])$.  
	
	In order to prove the validity of formula \eqref{formuladerivativeT-2} one can observe that since $h(t,r) = (R{-}r)^{\frac{1}{2}} \, e^{\frac{\alpha}{2}t} \,v(t,r)$, then one can rewrite
	\begin{equation*}
	\mathcal{E}(t) = \pi e^{-\alpha t} \int_{0}^{\rho(t)}  \left(h_t(t,r) -\frac{\alpha}{2} \,h(t,r) \right)^2 + 
	\left( h_r(t,r) + \frac{1}{2} (R{-}r)^{-1}\,h(t,r)\right)^2\,dr
	\end{equation*}
	and proceed with analogous computations as above.
\end{proof}

\section{A general uniqueness result} \label{app:uniq}

In this appendix we provide a uniqueness result for problem \eqref{princeq} when the debonding front $\rho$ is prescribed and is sufficiently regular; together with the results of Section \ref{sec1}, this in particular implies that the solution is radial. To this end, we follow a classic technique using changes of variables \cite{DautrayLions}: in our case, the space domain $\mathcal{C}_{R-\rho(t), R}$ is transformed into the fixed annulus $\mathcal{C}_{R-\rho_0, R}$. Correspondingly, the wave equation is transformed into a second order hyperbolic equation in divergence form. 
These problems are well studied (see e.g.\ \cite{LionsMagenes1972} for hyperbolic equations with coefficients and \cite{Lions64,Rog66,BC73,Coop75,Sikorav90,Kozh01} for applications to time-dependent domains),
however we were unable to find in literature a proof directly applicable to our setting. Therefore we provide a uniqueness result for such equation employing methods by Ladyzenskaya \cite{Lady}.  Our proof is based on a suitable modification of a result proven in \cite[Thm. 3.10]{DMLuc15}, where the authors consider a domain with a growing crack.

We first state a uniqueness result for a damped wave equation with coefficients dependent on time and space. Afterwards we will recast problem \eqref{princeq} in an equation of this type.
Let $T>0$ and $\mathcal{C}$ be an open domain of $\R^2$ with Lipschitz boundary. 
We introduce the coefficients with the following assumptions:
\begin{equation} \label{assump-general}
\begin{split}
B  &\in \mathrm{Lip}([0,T]; L^\infty(\mathcal{C};\R^{2\times2}_{sym})) \quad \text{with coercivity constant}\ c_B>0 , \\ 
a & \in \mathrm{Lip}([0,T]; L^\infty(\mathcal{C};\R^2)) , \\
b & \in \mathrm{Lip}([0,T]; L^\infty(\mathcal{C};\R^2)) \quad \text{with}\ \mathrm{div} \, b \in \mathrm{Lip}([0,T];L^\infty(\mathcal{C})) ,  \\ 
c  & \in \mathrm{Lip}([0,T]; L^\infty(\mathcal{C})) .
\end{split}
\end{equation}
Then we introduce the following definition.

\begin{defin}\label{defgensol}
	We say that $U:[0,T] \times \mathcal{C} \to \R$ is a generalized solution of problem
	\begin{equation}\label{princeqU}
	\begin{dcases}
	U_{tt} -  \mathrm{div} (B(t) \nabla U)  + a(t) \cdot \nabla U - b(t) \cdot\nabla U_t  + c(t) U_t  = 0   & \text{in}\ (0,T)\times\mathcal{C} ,
	\\
	U(t,\xi) = 0 \qquad &  t \in (0,T), \, \xi \in  \partial\mathcal{C},
	\\
	U(0,\xi) = 0 \qquad &  \xi \in \mathcal{C},
	\\
	U_t(0,\xi) = 0 \qquad &  \xi \in \mathcal{C},
	\end{dcases}
	\end{equation}
	if
	\begin{enumerate}
		\item[(i)] $U \in L^\infty((0,T); H^1_0(\mathcal{C})$,
		\item[(ii)] $U_t \in L^\infty((0,T); L^2(\mathcal{C}))$, 
		\item[(iii)] $U_{tt} \in L^2((0,T); H^{-1}_0(\mathcal{C})$,
	\end{enumerate}
	and $U$ {satisfies}
	\begin{align*}
	\langle U_{tt}(t), V\rangle_{H^1_0(\mathcal{C})} + \langle B(t) \nabla U(t), \nabla V \rangle_{L^2(\mathcal{C})} &+ \langle a(t) \cdot \nabla U(t), V \rangle_{L^2(\mathcal{C})} 
	\\
	&+ \langle U_t(t), \mathrm{div}(b(t) V) \rangle_{L^2(\mathcal{C})}  
	+ \langle U_t(t), c(t) V \rangle_{L^2(\mathcal{C})} = 0
	\end{align*}
	for a.e.\ $t \in [0,T]$ and every $V \in H^1_0(\mathcal{C})$.
\end{defin}

Then we have the following result.

\begin{lemma}\label{uniquelemma}
	There {is} a unique generalized solution for problem \eqref{princeqU}, i.e.\ $U=0$.
\end{lemma}

\begin{proof}
	We follow the proof of Dal Maso and Lucardesi \cite[Thm. 3.10]{DMLuc15} highlighting the differences in our case.
Let $U$ be a generalized solution of \eqref{princeqU}: we shall prove that $U=0$.
Let us fix $s \in (0,T)$ and choose a specific test function $V \in L^2((0,T); H^1_0(\mathcal{C}))$ defined by
	\begin{equation*}
	V(t):=
	\begin{dcases}
	- \int_{t}^{s} U(\tau) \, d\tau &\mbox{if } t \in [0,s],
	\\
	0 &\mbox{if } t \in [s,T].
	\end{dcases}
	\end{equation*} 
	Note that $V(T)=V(s)=0$ and that $V,V_t \in L^\infty((0,T); H^1_0(\mathcal{C})$ since
$U \in L^\infty((0,T); H^1_0(\mathcal{C})$ and
	\begin{equation*}
	V_t(t):=
	\begin{dcases}
	U(t) &\mbox{if } t \in [0,s),
	\\
	0 &\mbox{if } t \in (s,T].
	\end{dcases}
	\end{equation*} 
Using this specific choice of test function in Definition \ref{defgensol} and integrating over time we get
	\begin{align}\label{weakintfor}
	\nonumber
&	\int_{0}^{s} \langle U_{tt}(t), V(t)\rangle_{H^1_0(\mathcal{C})} \, {dt} + \langle B(t) \nabla U(t), \nabla V(t) \rangle_{L^2(\mathcal{C})} 
	\\
	&+\langle a(t) \cdot \nabla U(t), V(t) \rangle_{L^2(\mathcal{C})} + \langle U_t(t), \mathrm{div}(b(t) V(t)) \rangle_{L^2(\mathcal{C})} 	+ \langle U_t(t), c(t) V \rangle_{L^2(\mathcal{C})} \,dt = 0.
	\end{align}
	We now proceed estimating all five terms in the integration above.
	
	As for the first term, integrating by parts with respect to time we have
	\begin{align}\label{1term}
	\int_{0}^{s} \langle U_{tt}(t), V(t)\rangle_{H^1_0(\mathcal{C})} \,dt &= -\frac{1}{2} \|U(s)\|^2_{L^2(\mathcal{C})} ,
	\end{align}
	where we have used the fact that $U_t(0)=0$ and $V(s)=0$ in $H^1_0(\mathcal{C})$.
	
	For the second term we proceed as follows: since $V \in \mathrm{Lip}([0,T]; H^1_0(\mathcal{C}))$, by \eqref{assump-general} we have $B \nabla V \in \mathrm{Lip}([0,T]; L^2(\mathcal{C}))$, so integrating again by parts with respect to time and using the fact that $V(s)=0$ in $H^1_0(\mathcal{C})$ we get
	\begin{align}\label{2term}
	\int_{0}^{s} \langle B(t) \nabla U(t), \nabla V(t) \rangle_{L^2(\mathcal{C})} \,dt = -\frac{1}{2} \langle \nabla V(0), B(0) \nabla V(0) \rangle_{L^2(\mathcal{C})}  -\frac{1}{2}  \int_{0}^{s} \langle \nabla V(t), B_t(t) \nabla V(t) \rangle_{L^2(\mathcal{C})}\,dt.
	\end{align}
	
	For the third term we proceed as follows: by \eqref{assump-general} we have $a_t \in L^\infty((0,T); L^{\infty}(\mathcal{C}; \R^2))$, 
so integrating by parts with respect to time we obtain
	\begin{equation}\label{3term}\begin{split}
	 &\int_{0}^{s} \langle a(t) \cdot \nabla U(t), V(t) \rangle_{L^2(\mathcal{C})} \,dt = \int_{0}^{s} \langle \nabla V_t(t), a(t) V(t) \rangle_{L^2(\mathcal{C})} \,dt 
	\\
	& = - \langle \nabla V(0), a(0) V(0) \rangle_{L^2(\mathcal{C})} - \int_{0}^{s} \langle \nabla V(t), a_t(t) V(t) \rangle_{L^2(\mathcal{C})} \,dt - \int_{0}^{s} \langle \nabla V(t), a(t) U(t) \rangle_{L^2(\mathcal{C})} \,dt ,
	\end{split}\end{equation}
	where we have used the fact that $V(s)=0$ in $H^1_0(\mathcal{C})$.
	
	For the fourth term we use the following relation: $\mathrm{div}(bV) = V \mathrm{div}\,b + \nabla V \cdot b$. By \eqref{assump-general}  we can proceed as for the second term. Integrating by parts with respect to time we obtain
	\begin{equation}\label{4term-a}\begin{split}
	&\int_{0}^{s} \langle U_t(t), V(t) \, \mathrm{div} \, b(t) \rangle_{L^2(\mathcal{C})} \,dt 
	\\ 
	&= 
	- \int_{0}^{s} \langle U(t),  V_t(t) \, \mathrm{div}\,b(t) \rangle_{L^2(\mathcal{C})} \,dt - \int_{0}^{s} \langle U(t), V(t) \, (\mathrm{div}\,b)_t(t) \rangle_{L^2(\mathcal{C})} \,dt ,
	\end{split}  \end{equation}
	where we have used the fact that $V(s)=0$ and $U_t(0)=0$. Moreover integrating with respect to time and space we get 
	\begin{equation}\label{4term-b}\begin{split}
	& \int_{0}^{s} \langle U_t(t), \nabla V(t) \cdot b(t) \rangle_{L^2(\mathcal{C})} \,dt 
	\\ 
	& = \frac{1}{2} \int_{0}^{s} \langle \mathrm{div} \, b(t) , |U(t)|^2 \rangle_{L^1(\mathcal{C})} \,dt -  \int_{0}^{s}  \langle U(t), \nabla V(t) \cdot b_t(t) \rangle_{L^2(\mathcal{C})} ,
	\end{split}\end{equation}
	where we have used the fact that $V(s)=0$, $U(0)=0$, and $U(t)=0$ on $\partial\mathcal{C}$.
	
	For the fifth term we integrate by parts with respect to time and get
	\begin{equation}\label{5term}
	\int_{0}^{s} \langle U_t(t), c(t) V \rangle_{L^2(\mathcal{C})} \,dt = - \int_{0}^{s} \langle U(t), c(t) U(t) \rangle_{L^2(\mathcal{C})} \,dt
- \int_{0}^{s} \langle U(t), c_t(t) V(t) \rangle_{L^2(\mathcal{C})} \,dt  ,
	\end{equation}
	where we have used the fact that $V(s)=0$ and $U(0)=0$ in $H^1_0(\mathcal{C})$.
	
	Finally \eqref{weakintfor}--\eqref{5term} yield
	\begin{equation}\label{weakintfor2}\begin{split}
	&\frac{1}{2} \|U(s)\|^2_{L^2(\mathcal{C})} + \frac{1}{2} \langle \nabla V(0), B(0) \nabla V(0) \rangle_{L^2(\mathcal{C})} =  -\frac{1}{2}\int_{0}^{s}  \langle \nabla V(t), B_t(t) \nabla V(t) \rangle_{L^2(\mathcal{C})} \, dt 
	\\
	&-\langle \nabla V(0), a(0) V(0) \rangle_{L^2(\mathcal{C})} 
    - \int_{0}^{s} \langle \nabla V(t), a_t(t) V(t) \rangle_{L^2(\mathcal{C})} \,dt - \int_{0}^{s} \langle \nabla V(t), a(t) U(t) \rangle_{L^2(\mathcal{C})} \,dt
	\\
	& 
	- \int_{0}^{s} \langle U(t),  V_t(t) \, \mathrm{div}\,b(t) \rangle_{L^2(\mathcal{C})} \,dt - \int_{0}^{s} \langle U(t), V(t) \, (\mathrm{div}\,b)_t(t) \rangle_{L^2(\mathcal{C})} \,dt
	\\
	&
	+\frac{1}{2} \int_{0}^{s} \langle \mathrm{div} \, b(t) , |U(t)|^2 \rangle_{L^1(\mathcal{C})} \,dt -  \int_{0}^{s}  \langle U(t), \nabla V(t) \cdot b_t(t) \rangle_{L^2(\mathcal{C})} \, dt
\\&
- \int_{0}^{s} \langle U(t), c(t) U(t) \rangle_{L^2(\mathcal{C})} \,dt 
- \int_{0}^{s} \langle U(t), c_t(t) V(t) \rangle_{L^2(\mathcal{C})} \, dt .
	\end{split} 
\end{equation}
By coercivity of $B$ we deduce that
	\begin{equation}\label{ineq1}
	\langle \nabla V(0), B(0) \nabla V(0) \rangle_{L^2(\mathcal{C})} \geq c_B \|\nabla V(0)\|^2_{L^2(\mathcal{C})} .
	\end{equation}
	From now on, $C$ will denote a positive constant which may change from line to line, independent of $s$.
	Since 
$\|B_t(t)\|_{L^\infty(\mathcal{C})} \leq C$ for a.e.\ $t \in (0,T)$, we deduce that
	\begin{equation}\label{ineq2}
		-\frac{1}{2}\int_{0}^{s} \langle \nabla V(t), B_t(t) \nabla V(t) \rangle_{L^2(\mathcal{C})}\,dt \leq C \int_{0}^{s} \| V(t)\|^2_{H^1_0(\mathcal{C})}\,dt .
	\end{equation}
Since $a,a_t\in L^\infty((0,T)\times\mathcal{C})$, we obtain that 
	\begin{equation}\label{ineq3}
	-\int_{0}^{s} \langle \nabla V(t), a_t(t) V(t) \rangle_{L^2(\mathcal{C})} \,dt - \int_{0}^{s} \langle \nabla V(t), a(t) U(t) \rangle_{L^2(\mathcal{C})} \,dt 
	\leq C \int_{0}^{s} (\|V(t)\|^2_{H^1_0(\mathcal{C})} + \|U(t)\|^2_{L^2(\mathcal{C})} ) \, dt .
	\end{equation}
By \eqref{assump-general}
	\begin{equation}\label{ineq4}\begin{split}
&	- \int_{0}^{s} \langle U(t),  V_t(t) \, \mathrm{div}\,b(t) \rangle_{L^2(\mathcal{C})} \,dt - \int_{0}^{s} \langle U(t), V(t) \, (\mathrm{div}\,b)_t(t) \rangle_{L^2(\mathcal{C})} \,dt \\
&	+\frac{1}{2} \int_{0}^{s} \langle \mathrm{div} \, b(t) , |U(t)|^2 \rangle_{L^1(\mathcal{C})} \,dt
    -  \int_{0}^{s}  \langle U(t), \nabla V(t) \cdot b_t(t) \rangle_{L^2(\mathcal{C})} \leq  C \int_{0}^{s} (\|V(t)\|^2_{H^1_0(\mathcal{C})} + \|U(t)\|^2_{L^2(\mathcal{C})} ) \, dt .
	\end{split}\end{equation}
Similarly,
	\begin{equation}\label{ineq5}
- \int_{0}^{s} \langle U(t), c(t) U(t) \rangle_{L^2(\mathcal{C})} \,dt 
- \int_{0}^{s} \langle U(t), c_t(t) V(t) \rangle_{L^2(\mathcal{C})} \, dt
	\leq C \int_{0}^{s} (\|V(t)\|^2_{H^1_0(\mathcal{C})} + \|U(t)\|^2_{L^2(\mathcal{C})} ) \, dt .
	\end{equation}
	Finally, using the fact that $a \in \mathrm{Lip}((0,T); L^\infty(\mathcal{C};\R^2))$, by a weighted Young inequality there is a positive constant $\varepsilon>0$ such that
\[	\langle \nabla V(0), a(0) V(0) \rangle_{L^2(\mathcal{C})} \leq \varepsilon \|\nabla V(0)\|^2_{L^2(\mathcal{C})} + \frac{C}\varepsilon \| V(0)\|^2_{L^2(\mathcal{C})} \leq \varepsilon \|\nabla V(0)\|^2_{L^2(\mathcal{C})} + \frac{C}\varepsilon \int_{0}^{s} \|U(t)\|^2_{L^2(\mathcal{C})}  \, dt . \]
By a suitable choice of $\varepsilon$, \eqref{weakintfor2}--\eqref{ineq5} provide us with the following estimate:
	\begin{align}\label{201023}
	\|U(s)\|^2_{L^2(\mathcal{C})} + C_1 \|\nabla V(0)\|^2_{L^2(\mathcal{C})} \leq C_2 \int_{0}^{s} ( \|U(t)\|^2_{L^2(\mathcal{C})} + \|V(t)\|^2_{H^1_0(\mathcal{C})} ) \, dt ,
	\end{align}
	for two positive constants $C_1$ and $C_2$.
	
	Now we define the auxiliary function
	\begin{equation*}
	Z(s) :=  \int_{0}^{s} U(\tau) \, d\tau ,
	\end{equation*}
	so  we can rewrite $V(t) = Z(t) - Z(s)$ for every $t \in [0,s]$. In particular, we plug the following estimates into \eqref{201023}:
	\begin{align*}
	 \|\nabla V(0)\|^2_{L^2(\mathcal{C})} & = \|\nabla Z(s)\|^2_{L^2(\mathcal{C})} ,
	\\
	 \int_{0}^{s}  \|V(t)\|^2_{H^1_0(\mathcal{C})}  \, dt & \leq 2s \|Z(s)\|^2_{H^1_0(\mathcal{C})} + 2 \int_{0}^{s} \|Z(t)\|^2_{H^1_0(\mathcal{C})} \,dt ,
	\\
	 \|Z(s)\|^2_{L^2(\mathcal{C})} & \leq 2T \int_{0}^{s} \|U(t)\|^2_{L^2(\mathcal{C})} \,dt .
	\end{align*}
	Hence \eqref{201023} implies
	\begin{equation*}
	\|U(s)\|^2_{L^2(\mathcal{C})} + (C_1 {-} 2C_2s) \|Z(s)\|^2_{H^1_0(\mathcal{C})} \leq (2TC_1 {+} 2C_2) \int_{0}^{s} ( \|U(t)\|^2_{L^2(\mathcal{C})} + \|Z(t)\|^2_{H^1_0(\mathcal{C})} ) \, dt .
	\end{equation*}
By choosing $s=s_0$ sufficiently small, we can apply Gronwall's Lemma and obtain $U=0$ in $[0,s_0]$.
By recursively applying the same argument to $[s_0,2s_0]$, $[2s_0,3s_0]$, etc., in a finite number of steps we obtain uniqueness on all the interval $[0,T]$.
\end{proof}

In the next lemma we consider a family of time-dependent Lipschitz domains $\mathcal{C}(t)$ and denote $\mathcal{C}:=\mathcal{C}(0)$.
We assume that $\mathcal{C}(t)$ is mapped into  $\mathcal{C}$ via a time-dependent diffeomorphism. Define 
\begin{equation*}
\mathcal{O} := \{(t,x) \in [0,+\infty) \times \R^2 \,|\, 0< t < T, \ x \in {\mathcal{C}}(t) \}.
\end{equation*} 
We check that, under suitable assumptions on such diffeomorphisms, a hyperbolic equation is still transformed into a hyperbolic equation.

\begin{lemma}\label{lemmaB}
	 Let $\Psi: \overline{\mathcal{O}} \to \R^2$ and $\Phi: [0,T] \times \overline{\mathcal{C}} \to \R^2$ be two functions of class $C^{1,1}$ such that, for every $t \in [0,T]$, $\Psi(t,\cdot)$ maps $\overline{\mathcal{C}}(t)$ into $ \overline{\mathcal{C}}$ and $\Phi(t,\cdot)$ maps $\overline{\mathcal{C}}$ into $\overline{\mathcal{C}}(t)$, and
	\begin{equation}\begin{split}
	\Psi(t, \Phi(t,\xi))&=\xi  \quad \forall t \in [0,T],  \forall \xi \in \overline{\mathcal{C}},
	\\
	\Phi(t, \Psi(t,x))&= x \quad \forall t \in [0,T], \forall x \in  \overline{ \mathcal{C}}(t).\label{eq: phipsi}
	\end{split}\end{equation}
	For every $t \in [0,T]$, assume that $\Psi(t,\cdot)$ and $\Phi(t,\cdot)$ are $C^{2,1}$-diffeomorphisms and that there exists $\delta>0$ such that 
	\begin{equation}\label{phicondition}
	|\dot{\Phi}(t,\xi)|^2 \leq 1-\delta \quad \forall t \in [0,T],  \forall \xi \in   \overline{\mathcal{C}}.
	\end{equation}
	Then the matrix $B$ given by 
	\begin{equation}\label{eqB}
	B(t,\xi) := D \Psi(t,\Phi(t,\xi)) D \Psi(t,\Phi(t,\xi))^T - \dot{\Psi}(t,\Phi(t,\xi)) \otimes  \dot{\Psi}(t,\Phi(t,\xi))  
	\end{equation}
	belongs to $ \mathrm{Lip}([0,T]; L^\infty(\mathcal{C};\R^{2\times2}_{sym}))$ and it is coercive.
\end{lemma}

\begin{proof}
	The regularity properties and the symmetry of $B$ follow directly by definition and by the assumptions on $\Psi$ and $\Phi$. We prove it is coercive. We first recall that it holds
	\begin{equation*}
	\dot{\Psi}(t,\Phi(t,\xi)) = - D \Psi(t,\Phi(t,\xi)) \, \dot{\Phi}(t,\xi)  \quad \forall t \in [0,T],  \forall \xi \in   \overline{\mathcal{C}},
	\end{equation*}
	which stems out  from  deriving \eqref{eq: phipsi} with respect to time on both sides. Furthermore, under the above regularity assumptions, the matrix $D \Psi(\cdot,\cdot) D \Psi(\cdot,\cdot)^T$ is positive definite uniformly with respect to $(t,x) \in \overline{\mathcal{O}}$, so that it exists a constant $c_\Psi>0$ with
	\begin{equation*}
	(D \Psi(t,\Phi(t,\xi)) D \Psi(t,\Phi(t,\xi))^T\eta)\cdot \eta =\left|D \Psi(t,\Phi(t,\xi))^T\eta\right|^2\ge c_\Psi|\eta|^2
	\end{equation*}
	for all $ (t, \xi) \in  [0,T] \times \overline{\mathcal{C}}$.
	Hence, by the definition of $B$  (cf.\ \eqref{eqB}), for every $\eta \in \R^2$, one has
	\begin{align*}
	&(B(t,\xi) \eta)\cdot \eta
	=
	(D \Psi(t,\Phi(t,\xi)) D \Psi(t,\Phi(t,\xi))^T\eta)\cdot \eta - (\dot{\Psi}(t,\Phi(t,\xi)) \otimes  \dot{\Psi}(t,\Phi(t,\xi)) \eta)\cdot \eta 
	\\
	&=(D \Psi(t,\Phi(t,\xi)) D \Psi(t,\Phi(t,\xi))^T \eta)\cdot \eta   - [D \Psi(t,\Phi(t,\xi)) \, \dot{\Phi}(t,\xi) \otimes D \Psi(t,\Phi(t,\xi)) \, \dot{\Phi}(t,\xi)]\eta \cdot \eta
	\\
	&=(D \Psi(t,\Phi(t,\xi)) D \Psi(t,\Phi(t,\xi))^T \eta)\cdot \eta   - \left|(D \Psi(t,\Phi(t,\xi)) \, \dot{\Phi}(t,\xi) \cdot \eta\right|^2 
	\\
	&\geq|(D \Psi(t,\Phi(t,\xi))^T \eta|^2   - |\dot{\Phi}(t,\xi)|^2 |D \Psi(t,\Phi(t,\xi))^T \eta|^2 
	\\
	&\geq \delta  |D \Psi(t,\Phi(t,\xi))^T \eta|^2 \geq \delta c_\Psi|\eta|^2,
	\end{align*}
	for every $t \in [0,T]$ and for every $\xi \in \overline{\mathcal{C}}$, where in the last inequality we have used the assumption on $\dot{\Phi}$ (cf.\ \eqref{phicondition}).
\end{proof}

We are now ready to prove the uniqueness of a solution for problem \eqref{princeq} for fixed debonding front $\rho$, by providing an explicit diffeomorphism that maps the annulus $\mathcal{C}_{R-\rho(t),R}$ into the fixed domain $\mathcal{C}_{R-\rho_0,R}$ for every $t \in [0,T]$.  (Recall \eqref{hprho}--\eqref{def:domain-time} for the definition of the domains.) The idea is to consider an affine transformation of the radius depending on time, namely $r \mapsto p(t)r + q(t)$, such that for every $t \in [0,T]$
\begin{equation*}
\begin{cases}
p(t)R + q(t) = R,
\\
p(t)(R{-}\rho(t)) + q(t) = \rho_0,
\end{cases}
\end{equation*}
which yields to
\begin{equation*}
\begin{cases}
p(t)= \frac{R-\rho_0}{\rho(t)},
\\
q(t) = R \left( 1- \frac{R-\rho_0}{\rho(t)} \right).
\end{cases}
\end{equation*}
This generates the following change of variables in the plane depending on time: 
\begin{equation*}
\Psi: [0,T] \times (\R^2{\setminus}\{0\}) \to \R^2 , \quad (t,x) \mapsto \Psi(t,x),
\end{equation*}
where $\Psi(t,x)$ is given by
\begin{equation}\label{psieq}
\Psi(t,x):= \left(p(t)x_1 + q(t) \frac{x_1}{|x|} , \, p(t)x_2 + q(t) \frac{x_2}{|x|} \right),
\end{equation}
and, for every $t \in [0,T]$, $\Psi(t, \overline{\mathcal{C}}_{R-\rho(t),R}) =  \overline{\mathcal{C}}_{R-\rho_0,R}$. Moreover we define 
\begin{equation*}
\Phi: [0,T] \times (\R^2{\setminus}\{0\}) \to \R^2, \quad (t,\xi) \mapsto \Phi(t,\xi) ,
\end{equation*}
where $\Phi(t,\xi)$ is given by
\begin{equation}\label{phieq}
\Phi(t,\xi):= \left( \frac{|\xi| {-} q(t)}{p(t)} \frac{\xi_1}{|\xi|} , \, \frac{|\xi| {-} q(t)}{p(t)} \frac{\xi_2}{|\xi|} \right) ,
\end{equation} 
and, for every $t \in [0,T]$, $\Phi(t, \overline{\mathcal{C}}_{R-\rho_0,R}) =  \overline{\mathcal{C}}_{R-\rho(t),R}$.
We are now in a position to deduce that there is only one solution to problem  \eqref{princeq}, provided some regularity is assumed on the function $\rho$.
In particular it follows that such solution is the one determined in \ref{teoexistenceh} and fulfilling \eqref{formulah}.
\begin{prop}\label{prop-app-uniq}
	Let $\rho: [0, T] \to [\rho_0, R)$ be such that $\rho \in C^{2,1}([0,T])$, $\rho(0) = \rho_0  >0 $ and $0 \leq \dot{\rho}(t) <1$ for every $t \in [0,T]$.
	Then problem \eqref{princeq} admits at most one solution.
\end{prop}

\begin{proof}
	By the regularity of $\rho$ and by the fact that $\rho(t)\geq \rho_0 >0$ for every $t \in [0,T]$, one deduces that the functions $\Psi$ and $\Phi$ defined by \eqref{psieq} and \eqref{phieq} 
 are of class $C^{1,1}$, they satisfy \eqref{eq: phipsi},  and, for every $t\in[0,T]$, $\Psi(t,\cdot)$ and $\Phi(t,\cdot)$ are $C^{2,1}$-diffeomorphisms.  
By the continuity of $\dot \rho$, we may fix $\delta$ such that $\dot \rho(t)^2\le 1-\delta$ for all $t\in [0, T]$.
 Now let $\mathcal{O}_\rho$ be as in \eqref{def:domain-time} and $u \in W^{1,2}(\mathcal{O}_\rho)$ be a solution of problem \eqref{princeq}. For brevity let $\mathcal{C} := \mathcal{C}_{R-\rho_0,R}$, and introduce the following auxiliary function $U$, defined by
	\begin{equation*}
	U(t,\xi) := u(t,\Phi(t,\xi)) \quad \mbox{for every } t \in [0,T], \xi \in \mathcal{C}.
	\end{equation*}
	It holds equivalently $u(t, x)=U(t, \Psi(t, x))$ for every $(t,x) \in \mathcal{O}_\rho$. Performing the change of variables (cf., for instance,  \cite[Equation (2.25)]{DMLuc15}), one deduces that $U$ is a generalized solution of 
	\begin{equation*}
	\begin{dcases}
U_{tt} -  \mathrm{div} (B(t) \nabla U)  + a(t) \cdot \nabla U - 2b(t) \cdot\nabla U_t  + c(t) U_t  = 0
& \text{in}\ (0,T)\times\mathcal{C},
	\\
	U(t,\xi) = w(t)  &  t \in (0,T), \, \xi = R,
	\\
	U(t,\xi) = 0  & t \in (0,T), \, \xi = R-\rho_0,
 	\\
	U(0,\xi) = u_0(\Phi(0,\xi))  &  \xi \in \mathcal{C},
	\\
	U_t(0,\xi) = u_1(\Phi(0,\xi)) + \dot{\Phi}(0,\xi) \cdot \nabla u_0(\Phi(0,\xi))  &  \xi \in \mathcal{C},
	\end{dcases}
	\end{equation*}
	where for every $ t \in [0,T]$ and $ \xi \in \overline{\mathcal{C}}$ we have defined
	\begin{align*}
	B(t,\xi) &:= D \Psi(t,\Phi(t,\xi)) D \Psi(t,\Phi(t,\xi))^T - \dot{\Psi}(t,\Phi(t,\xi)) \otimes  \dot{\Psi}(t,\Phi(t,\xi)),
	\\
	a(t,\xi) &:= - \{B^T(t,\xi) \nabla (\mbox{det} D\Phi(t,\xi)) +\partial_t [ b(t,\xi) \,\mbox{det}D\Phi(t,\xi)]\} \,\mbox{det}D\Psi(t,\Phi(t,\xi))) + \alpha\dot{\Psi}(t,\Phi(t,\xi)),
	\\
	b(t,\xi) &:= - \dot{\Psi}(t,\Phi(t,\xi)),
	\\
	c(t,\xi) &:= \alpha  .
	\end{align*}
	By the regularity of $\rho$ and the consequential regularity of $\Phi$ and $\Psi$, we deduce that  assumptions \eqref{assump-general} hold.
	Moreover by the fact that $\dot{\rho}(t)^2 \leq 1-\delta$ for every $t \in [0,T]$, we can deduce, by direct computation, that $|\dot{\Phi}(t,\xi)|^2 \leq 1-\delta$ for every $t \in [0,T]$ and $\xi \in  \overline{\mathcal{C}}$. In fact,
	\begin{equation*}
	|\dot{\Phi}(t,\xi)|^2 = \left| \frac{\dot{q}(t) p(t) + (|\xi| {-} q(t)) \, \dot{p}(t)}{p^2(t)}\right|^2 = \frac{\left(R\dot{\rho}(t)\frac{R-\rho_0}{\rho^2(t)} - |\xi|\dot{\rho}(t) \frac{R-\rho_0}{\rho^2(t)} \right)^2}{\left(\frac{R-\rho_0}{\rho(t)}\right)^4} = \left(\frac{R{-}|\xi|}{R{-}\rho_0}\right)^2 (\dot{\rho}(t))^2 
	\end{equation*}
	where $\left(\frac{R-|\xi|}{R-\rho_0}\right)^2 \leq 1$ for every $\xi \in  \overline{\mathcal{C}}$.
	Then by Lemma \ref{lemmaB} we deduce that $ B \in \mathrm{Lip}([0,T]; L^\infty(\mathcal{C};\R^{2\times2}_{sym}))$ and it is coercive.  Hence the conclusion follows by Lemma \ref{uniquelemma}.
\end{proof}

\begin{rem}\label{rmk-app}
	We observe that the hypotheses of Proposition \ref{prop-app-uniq} can be weakened in the following way. Let $\rho: [0, T] \to [\rho_0, R)$ be such that $\rho \in C^{1}([0,T])$, $\rho(0) = \rho_0  >0 $ and $0 \leq \dot{\rho}(t) <1$ for  a.e.\ $t \in [0,T]$. Moreover, assume that there exists a finite partition $\{[t_{i},t_{i+1}]\}_{i=0}^m$ of the interval $[0,T]$, such that $\rho$ is $C^{2,1}([t_{i},t_{i+1}])$ for $i=0,\dots,m{-}1$. Then, it is enough to apply Proposition \ref{prop-app-uniq} on each subinterval $[t_{i},t_{i+1}]$ and  to update $\rho_0 = \rho(t_i)$ at each $i$-step. By  continuity of $\rho$, the uniqueness on the whole interval $[0,T]$ follows. 
\end{rem}

\bigskip

{\small
\subsection*{Acknowledgments} 
The authors would like to thank Corrado Maurini and Filippo Riva for fruitful discussions.
This work is part of the Project {\em Variational methods for stationary and evolution problems with singularities and interfaces} (PRIN 2017) funded by the Italian Ministry of Education, University, and Research.
GL and FS have been supported by the {\em Istituto Nazionale di Alta Matematica} (INdAM). 
}


\begin{thebibliography}{99}
\bibitem{AbdDeb} 
{\sc R.\ Abdelmoula and G.\ Debruyne}, 
{\em Modal analysis of the dynamic crack growth and arrest in a DCB specimen}, 
Int.\ J.\ Fract., 188 (2014), pp.~187--202.
%
\bibitem{BBL98}	 
{\sc M.~L.\ Bernardi, G.\ Bonfanti, and F.\ Luterotti}, 
{\em On some abstract variable domain hyperbolic differential equations},
 Ann.\ Mat.\ Pura Appl. (4), 174 (1998), pp.~209--239.
%
\bibitem{BouFraMar08} 
{\sc B.~Bourdin, G.~A. Francfort, and J.-J. Marigo}, {\em The variational
  approach to fracture}, J. Elasticity, {91} (2008), no. 1-3, pp.~5--148.
%
\bibitem{Caponi-NoDEA} 
{\sc M.~Caponi},
{\em Existence of solutions to a phase-field model of dynamic fracture with a crack-dependent dissipation},
{NoDEA Nonlinear Differential Equations Appl.}, {27} (2020), art.\ no.~{14}.
%
\bibitem{CapLucTas} 
{\sc M.~Caponi, I.~Lucardesi, and E.~Tasso},
{\em Energy-dissipation balance of a smooth moving crack},
{J.\ Math.\ Anal.\ Appl.}, {483} (2020), art.\ no.~{123656}.
%
\bibitem{CapSap} 
{\sc M.~Caponi and F.~Sapio},
{\em A dynamic model for viscoelastic materials with prescribed growing cracks},
{Ann.\ Mat.\ Pura Appl.\ (4)}, {199} (2020), pp.~{1263--1292}.
%
\bibitem{Coop75} 
{\sc J.\ Cooper},
{\em Local decay of solutions of the wave equation in the exterior of a moving body},
J. Math. Anal. Appl., 49 (1975), pp.~130--153.
%
\bibitem{BC73} 
{\sc J.\ Cooper and C.\ Bardos}, 
{\em A nonlinear wave equation in a time dependent domain},
 J. Math. Anal. Appl., 42 (1973), pp.~29--60.
%
%
\bibitem{DMLarToa15} 
{\sc G.~Dal~Maso, C.~J. Larsen, and R.~Toader}, {\em Existence for constrained
  dynamic {G}riffith fracture with a weak maximal dissipation condition}, J. Mech. Phys. Solids, 95 (2016),  pp.~697--707.
%
\bibitem{DMLarToa17} 
{\sc G.~Dal~Maso, C.~J. Larsen, and R.~Toader}, 
{\em Existence for elastodynamic {G}riffith fracture with a weak maximal dissipation condition}, 
J.\ Math.\ Pures Appl.\ (9), in press (2018), doi:10.1016/j.matpur.2018.08.006.
%
\bibitem{DMLazNar16} 
{\sc G.~Dal Maso, G.~Lazzaroni, and L.~Nardini}, {\em Existence and uniqueness of dynamic evolutions for a peeling test in dimension one},  J. Differential Equations, 261 (2016), pp.~4897--4923.  
%
\bibitem{DMLuc15} 
{\sc G.~Dal~Maso and I.~Lucardesi}, {\em The wave equation on domains with
  cracks growing on a prescribed path: existence, uniqueness, and continuous
  dependence on the data}, 
Appl.\ Math.\ Res.\ Express, 2017 (2016), pp.~184--241.
%
\bibitem{DautrayLions} 
{\sc R.~Dautray and J.-L. Lions},
{\em Analyse math\'ematique et calcul num\'erique pour les sciences et les techniques. Vol. 8: \'Evolution: semi-groupe, variationnel},
Masson, Paris, 1988.
%
\bibitem{DouMarCha08} 
{\sc P.-E. Dumouchel, J.-J. Marigo, and M.~Charlotte}, {\em Dynamic fracture:
  an example of convergence towards a discontinuous quasistatic solution},
  Contin. Mech. Thermodyn., 20 (2008), pp.~1--19.
%
\bibitem{Fre90}
{\sc L.~B. Freund}, {\em Dynamic fracture mechanics}, Cambridge Monographs on
  Mechanics and Applied Mathematics, Cambridge University Press, Cambridge,
  1990.
%
\bibitem{Kozh01} 
{\sc A.~I.\ Kozhanov and N.~A.\ Larkin},
{\em On the solvability of boundary value problems for the wave
              equation with nonlinear dissipation in noncylindrical domains},
Sibirsk. Mat. Zh., 6 (2001), pp.~1278--1299.
%
\bibitem{Lady} 
{\sc O.~A.\ Ladyzenskaya},
{\em On integral estimates, convergence, approximate methods, and solution in functionals for elliptic operators},
Vestnik Leningrad Univ., 13 (1958), pp.~60--69.
%
\bibitem{LarOrtSue10} 
{\sc C.~J. Larsen, C.~Ortner, and E.~S{\"u}li}, {\em Existence of solutions to
  a regularized model of dynamic fracture}, Math. Models Methods Appl. Sci., 20
  (2010), pp.~1021--1048.
%
\bibitem{LBDM12} 
{\sc G.~Lazzaroni, R.~Bargellini, P.-E. Dumouchel, and J.-J. Marigo}, {\em On the role of kinetic energy during unstable propagation in a heterogeneous
  peeling test}, Int. J. Fract., 175 (2012), pp.~127--150.
%
\bibitem{LazNar-qs} 
{\sc G.~Lazzaroni and L.~Nardini}, {\em On the quasistatic limit of dynamic evolutions for a peeling test in dimension one},  J.\ Nonlinear Sci., 28 (2018), pp.~269--304. 
%
\bibitem{LazNar-speed} 
{\sc G.~Lazzaroni and L.~Nardini}, {\em Analysis of a dynamic peeling test with speed-dependent toughness}, 
SIAM J.\ Appl.\ Math., 78 (2018), pp.~1206--1227.
%
\bibitem{LazNar-init} 
{\sc G.~Lazzaroni and L.~Nardini}, {\em On the 1d wave equation in time-dependent domains and the problem of debond initiation}, 
ESAIM Control Optim.\ Calc.\ Var., 25 (2019), art.\ no.~80.
%
\bibitem{LRTT14} 
{\sc G.~Lazzaroni, R.~Rossi, M.~Thomas, and R.~Toader}, {\em Rate-independent
  damage in thermo-viscoelastic materials with inertia}, J.\ Dynam.\ Differential Equations, 30 (2018), pp.~1311--1364. 
%
\bibitem{LionsMagenes1972}
{\sc J.~L.\ Lions and E.\ Magenes}, {\em Non-Homogeneous Boundary Value Problems and Applications, Vol. I}, 
Die Grundlehren der mathematischen Wissenschaften, 181, Springer-Verlag, New York--Heidelberg, 1972.
%
\bibitem{Lions64} 
{\sc J.-L.\ Lions},
{\em Une remarque sur les probl\`emes d'\'{e}volution non lin\'{e}aires dans
              des domaines non cylindriques},
Rev. Roumaine Math. Pures Appl., 9 (1964), pp.~11--18.
%
\bibitem{MieRou15} 
{\sc A.~Mielke  and T.~Roub{\'{\i}}{\v{c}}ek}, {\em Rate-Independent Systems: Theory and Application},  Applied Mathematical Sciences, 193, Springer, New York (2015).
%
\bibitem{NicSae07} 
{\sc S.~Nicaise and A.-M. S{\"a}ndig}, {\em Dynamic crack propagation in a 2{D}
  elastic body: the out-of-plane case}, J. Math. Anal. Appl., 329 (2007),
  pp.~1--30.
%
\bibitem{Riva-MJM} 
{\sc F.\ Riva}, {\em A continuous dependence result for a dynamic debonding model in dimension one}, Milan J.\ Math., 87 (2019), pp.~315--350.
%
\bibitem{Riva-JNLS} 
{\sc F.\ Riva}, {\em On the approximation of quasistatic evolutions for the debonding of a thin film via vanishing inertia and viscosity}, J.\ Nonlinear Sci., 30 (2020), pp.~903--951.
%
\bibitem{RiNa2020} 
{\sc F.\ Riva and L.\ Nardini}, {\em Existence and uniqueness of dynamic evolutions for a one-dimensional debonding model with damping}, J.\ Evol.\ Equ.,  doi:10.1007/s00028-020-00571-4 (2020).
%
\bibitem{Rog66} 
{\sc E.~D.\ Rogak},
{\em A mixed problem for the wave equation in a time dependent domain},
Arch. Ration. Mech. Anal., 22  (1966), pp.~24--36.
%
\bibitem{RosRou11} 
{\sc R.~Rossi and T.~Roub{\'{\i}}{\v{c}}ek}, {\em Thermodynamics and analysis
  of rate-independent adhesive contact at small strains}, Nonlinear Anal., 74
  (2011), pp.~3159--3190.
%
\bibitem{RosTho16} 
{\sc R.~Rossi and M.~Thomas}, {\em From adhesive to brittle delamination in visco-elastodynamics}, 
{Math. Models Methods Appl. Sci.}, {27} (2017), pp.~{1489--1546}.
%
\bibitem{Rou13a} 
{\sc T.~Roub{\'{\i}}{\v{c}}ek}, {\em Adhesive contact of
  visco-elastic bodies and defect measures arising by vanishing viscosity},
  SIAM J. Math. Anal., 45 (2013), pp.~101--126.
%
\bibitem{Sikorav90} 
{\sc J.\ Sikorav},
{\em A linear wave equation in a time-dependent domain},
J. Math. Anal. Appl., 153 (1990), pp.~533--548.
%
%
\end{thebibliography}
\end{document}